\documentclass[11pt,a4paper]{amsart}
\pdfoutput=1
\usepackage[centering, margin=1in]{geometry}
\usepackage{accents,amsmath,amsfonts,amsthm,amssymb,enumitem,graphicx,mathrsfs,mathtools}
\usepackage[stretch=10]{microtype}
\usepackage[unicode]{hyperref}
\usepackage{xcolor}
\definecolor{dark-red}{rgb}{0.4,0.15,0.15}
\definecolor{dark-blue}{rgb}{0.15,0.15,0.4}
\definecolor{medium-blue}{rgb}{0,0,0.5}
\hypersetup{
	pdftitle={On the Random Wave Conjecture for Dihedral Maa\ss{} Forms},
	pdfauthor={Peter Humphries and Rizwanur Khan},
	pdfnewwindow=true,
	colorlinks,
	linkcolor={dark-red},
	citecolor={dark-blue},
	urlcolor={medium-blue}
}
\makeatletter
\newcommand*{\defeq}{\mathrel{\rlap{%
			\raisebox{0.3ex}{$\m@th\cdot$}}%
		\raisebox{-0.3ex}{$\m@th\cdot$}}%
	=}
\newcommand*{\eqdef}{\mathrel{=\llap{%
			\raisebox{0.3ex}{$\m@th\cdot$}}%
		\llap{\raisebox{-0.3ex}{$\m@th\cdot$}}}%
}
\makeatletter
\DeclareRobustCommand{\widecheck}[1]{{\mathpalette\@widecheck{#1}}}
\def\@widecheck#1#2{%
	\setbox\z@\hbox{\m@th$#1#2$}%
	\setbox\tw@\hbox{\m@th$#1%
		\widehat{%
			\vrule\@width\z@\@height\ht\z@
			\vrule\@height\z@\@width\wd\z@}$}%
	\dp\tw@-\ht\z@
	\@tempdima\ht\z@ \advance\@tempdima2\ht\tw@ \divide\@tempdima\thr@@
	\setbox\tw@\hbox{%
		\raise\@tempdima\hbox{\scalebox{1}[-1]{\lower\@tempdima\box
				\tw@}}}%
	{\ooalign{\box\tw@ \cr \box\z@}}}
\makeatother
\newcommand{\bigboxplus}{
	\mathop{
		\vphantom{\bigoplus} 
		\mathchoice
		{\vcenter{\hbox{\resizebox{\widthof{$\displaystyle\bigoplus$}}{!}{$\boxplus$}}}}
		{\vcenter{\hbox{\resizebox{\widthof{$\bigoplus$}}{!}{$\boxplus$}}}}
		{\vcenter{\hbox{\resizebox{\widthof{$\scriptstyle\oplus$}}{!}{$\boxplus$}}}}
		{\vcenter{\hbox{\resizebox{\widthof{$\scriptscriptstyle\oplus$}}{!}{$\boxplus$}}}}
	}\displaylimits 
}
\allowdisplaybreaks
\newcommand{\A}{\mathbb{A}}
\renewcommand{\aa}{\mathfrak{a}}
\renewcommand{\AA}{\mathcal{A}}
\newcommand{\Bgp}{\mathrm{B}}
\newcommand{\bb}{\mathfrak{b}}
\newcommand{\BB}{\mathcal{B}}
\renewcommand{\C}{\mathbb{C}}
\newcommand{\DD}{\mathcal{D}}
\newcommand{\e}{\varepsilon}
\newcommand{\Eis}{\textnormal{Eis}}
\newcommand{\Ell}{\mathcal{L}}
\newcommand{\GG}{\mathcal{G}}
\newcommand{\Hb}{\mathbb{H}}
\newcommand{\hf}{\mathfrak{h}}
\newcommand{\hol}{\textnormal{hol}}
\newcommand{\JJ}{\mathcal{J}}
\newcommand{\Ks}{\mathscr{K}}
\newcommand{\Ls}{\mathscr{L}}
\newcommand{\Maass}{\textnormal{Maa\ss{}}}
\newcommand{\MM}{\mathcal{M}}
\newcommand{\N}{\mathbb{N}}

\newcommand{\NN}{\mathcal{N}}
\newcommand{\Ns}{\mathscr{N}}
\newcommand{\OO}{\mathcal{O}}
\newcommand{\pp}{\mathfrak{p}}
\newcommand{\Q}{\mathbb{Q}}
\newcommand{\R}{\mathbb{R}}
\newcommand{\RR}{\mathcal{R}}
\newcommand{\RS}{\mathrm{RS}}
\newcommand{\spec}{\textnormal{spec}}
\newcommand{\St}{\mathrm{St}}
\newcommand{\Ts}{\mathscr{T}}
\newcommand{\WW}{\mathcal{W}}
\newcommand{\Z}{\mathbb{Z}}
\newcommand{\Zgp}{\mathrm{Z}}
\DeclareMathOperator*{\ad}{ad}
\DeclareMathOperator{\Gal}{Gal}
\DeclareMathOperator{\GL}{GL}
\DeclareMathOperator{\PSL}{PSL}
\DeclareMathOperator*{\Res}{Res}
\DeclareMathOperator{\sgn}{sgn}
\DeclareMathOperator{\SL}{SL}
\DeclareMathOperator{\SO}{SO}
\DeclareMathOperator{\sym}{sym}
\DeclareMathOperator{\Var}{Var}
\DeclareMathOperator{\vol}{vol}
\numberwithin{equation}{section}
\newtheorem{theorem}[equation]{Theorem}
\newtheorem{conjecture}[equation]{Conjecture}
\newtheorem{corollary}[equation]{Corollary}
\newtheorem{lemma}[equation]{Lemma}
\newtheorem{proposition}[equation]{Proposition}
\theoremstyle{remark}
\newtheorem{remark}[equation]{Remark}
\begin{document}

\title[On the Random Wave Conjecture for Dihedral Maa\ss{} Forms]{On the Random Wave Conjecture for Dihedral Maa{\ss} Forms}

\author{Peter Humphries}

\address{Department of Mathematics, University College London, Gower Street, London WC1E 6BT, United Kingdom}

\email{\href{mailto:pclhumphries@gmail.com}{pclhumphries@gmail.com}}

\author{Rizwanur Khan}

\address{Department of Mathematics, University of Mississippi, University, MS 38677, USA}

\email{\href{mailto:rrkhan@olemiss.edu}{rrkhan@olemiss.edu}}


\thanks{The first author is supported by the European Research Council grant agreement 670239.}

\begin{abstract}
We prove two results on arithmetic quantum chaos for dihedral Maa\ss{} forms, both of which are manifestations of Berry's random wave conjecture: Planck scale mass equidistribution and an asymptotic formula for the fourth moment. For level $1$ forms, these results were previously known for Eisenstein series and conditionally on the generalised Lindel\"{o}f hypothesis for Hecke--Maa\ss{} eigenforms. A key aspect of the proofs is bounds for certain mixed moments of $L$-functions that imply hybrid subconvexity.
\end{abstract}

\maketitle

\section{Introduction}

The random wave conjecture of Berry \cite{Ber77} is the heuristic that the eigenfunctions of a classically ergodic system ought to evince Gaussian random behaviour, as though they were random waves, in the large eigenvalue limit. In this article, we study and resolve two manifestations of this conjecture for a particular subsequence of Laplacian eigenfunctions, dihedral Maa\ss{} forms, on the surface $\Gamma_0(q) \backslash \Hb$.

\subsection{The Rate of Equidistribution for Quantum Unique Ergodicity}

Given a positive integer $q$ and a Dirichlet character $\chi$ modulo $q$, denote by $L^2(\Gamma_0(q) \backslash \Hb, \chi)$ the space of measurable functions $f : \Hb \to \C$ satisfying
\[f\left(\frac{az + b}{cz + d}\right) = \chi(d) f(z) \quad \text{for all $\begin{pmatrix} a & b \\ c & d \end{pmatrix} \in \Gamma_0(q)$}\]
and $\langle f,f\rangle_q < \infty$, where $\langle \cdot,\cdot\rangle_q$ denotes the inner product
\[\langle f,g \rangle_q \defeq \int_{\Gamma_0(q) \backslash \Hb} f(z) \overline{g(z)} \, d\mu(z)\]
with $d\mu(z) = y^{-2} \, dx \, dy$ on any fundamental domain of $\Gamma_0(q) \backslash \Hb$.

Quantum unique ergodicity in configuration space for $L^2(\Gamma_0(q) \backslash \Hb, \chi)$ is the statement that for any subsequence of Laplacian eigenfunctions $g \in L^2(\Gamma_0(q) \backslash \Hb, \chi)$ normalised such that $\langle g, g \rangle_q = 1$ with eigenvalue $\lambda_g = 1/4 + t_g^2$ tending to infinity,
\[\int_{\Gamma_0(q) \backslash \Hb} f(z) |g(z)|^2 \, d\mu(z) = \frac{1}{\vol(\Gamma_0(q) \backslash \Hb)} \int_{\Gamma_0(q) \backslash \Hb} f(z) \, d\mu(z) + o_{f,q}(1)\]
for every $f \in C_b\left(\Gamma_0(q) \backslash \Hb\right)$, or equivalently for every indicator function $f = 1_B$ of a continuity set $B \subset \Gamma_0(q) \backslash \Hb$. This is known to be true (and in a stronger form, in the sense of quantum unique ergodicity on phase space), provided each eigenfunction $g$ is a Hecke--Maa\ss{} eigenform, via the work of Lindenstrauss \cite{Lin06} and Soundararajan \cite{Sou10}.

One may ask whether the rate of equidistribution for quantum unique ergodicity can be quantified in some way; Lindenstrauss' proof is via ergodic methods and does not address this aspect. One method of quantification is to give explicit rates of decay as $\lambda_g$ tends to infinity for the terms
\begin{equation}
\label{QUEfErateseq}
\int_{\Gamma_0(q) \backslash \Hb} f(z) |g(z)|^2 \, d\mu(z), \qquad \int_{\Gamma_0(q) \backslash \Hb} E_{\aa}(z,\psi) |g(z)|^2 \, d\mu(z)
\end{equation}
for a fixed Hecke--Maa\ss{} eigenform $f$ or incomplete Eisenstein series $E_{\aa}(z,\psi)$; optimal decay rates for these integrals, namely $O_{q,f,\e}(t_g^{-1/2 + \e})$ and $O_{q,\psi,\e}(t_g^{-1/2 + \e})$ respectively, follow from the generalised Lindel\"{o}f hypothesis \cite[Corollary 1]{Wat08}. Ghosh, Reznikov, and Sarnak have proposed other quantifications \cite[Conjecture A.1 and A.3]{GRS13}.

Another quantification of the rate of equidistribution, closely related to the spherical cap discrepancy discussed in \cite{LS95}, is small scale mass equidistribution. Let $B_R(w)$ denote the hyperbolic ball of radius $R$ centred at $w \in \Gamma_0(q) \backslash \Hb$ with volume $4\pi \sinh^2 (R/2)$. Two small scale refinements of quantum unique ergodicity were studied in \cite{You16} and \cite{Hum18} respectively, namely the investigation of the rates of decay in $R$, with regards to the growth of the spectral parameter $t_g \in [0,\infty) \cup i(0,1/2)$, for which either the asymptotic formula
\begin{equation}
\label{shrinkingball}
\frac{1}{\vol(B_R)} \int_{B_R(w)} |g(z)|^2 \, d\mu(z) = \frac{1}{\vol(\Gamma_0(q) \backslash \Hb)} + o_{q,w}(1)
\end{equation}
or the bound
\begin{equation}
\label{volshrinkingball}
\vol\left(\left\{w \in \Gamma_0(q) \backslash \Hb : \left|\frac{1}{\vol(B_R)} \int_{B_R(w)} |g(z)|^2 \, d\mu(z) - \frac{1}{\vol(\Gamma_0(q) \backslash \Hb)}\right| > c\right\}\right) = o_c(1)
\end{equation}
holds as $t_g$ tends to infinity along any subsequence of $g \in \BB_0^{\ast}(q,\chi)$, the set of $L^2$-normalised newforms $g$ of weight zero, level $q$, nebentypus $\chi$, and Laplacian eigenvalue $\lambda_g = 1/4 + t_g^2$.

\begin{remark}
One can interpret these two small scale equidistribution questions in terms of random variables, as in \cite[Section 1.5]{GW17} and \cite[Section 1.3]{WY19}. We define the random variable $X_{g;R} : \Gamma_0(q) \backslash \Hb \to [0,\infty)$ by
\[X_{g;R}(w) \defeq \frac{1}{\vol(B_R)} \int_{B_R(w)} |g(z)|^2 \, d\mu(z),\]
which has expectation $1/\vol(\Gamma_0(q) \backslash \Hb)$. The asymptotic formula \eqref{shrinkingball} is equivalent to the pointwise convergence of $X_{g;R}$ to $1$, while \eqref{volshrinkingball} is simply the convergence in probability of $X_{g;R}$ to $1$, a consequence of the bound $\Var(X_{g;R}) = o(1)$. One could ask for further refinements of these problems, such as asymptotic formul\ae{} for this variance and a central limit theorem, as studied in \cite{WY19} for toral Laplace eigenfunctions, though we do not pursue these problems.
\end{remark}

For $q = 1$, Young \cite[Proposition 1.5]{You16} has shown that \eqref{shrinkingball} holds when $R \gg t_g^{-\delta}$ with $0 < \delta < 1/3$ under the assumption of the generalised Lindel\"{o}f hypothesis, and that an analogous result with $0 < \delta < 1/9$ is true unconditionally for the Eisenstein series $g(z) = E(z,1/2 + it_g)$ \cite[Theorem 1.4]{You16}. One expects that this is true for $0 < \delta < 1$, but the method of proof of \cite[Proposition 1.5]{You16} is hindered by an inability to detect cancellation involving a spectral sum of terms not necessarily all of the same sign; see \cite[p.~965]{You16}.

This hindrance does not arise for \eqref{volshrinkingball}, and so we are lead to the following conjecture on Planck scale mass equidistribution, which roughly states that quantum unique ergodicity holds for almost every shrinking ball whose radius is larger than the Planck scale $\lambda_g^{-1/2}$.

\begin{conjecture}
\label{Planckconj}
Suppose that $R \gg t_g^{-\delta}$ with $0 < \delta < 1$. Then \eqref{volshrinkingball} holds as $t_g$ tends to infinity along any subsequence of newforms $g \in \BB_0^{\ast}(q,\chi)$.
\end{conjecture}

Via Chebyshev's inequality, the left-hand side of \eqref{volshrinkingball} is bounded by $c^{-2} \Var(g;R)$, where
\[\Var(g;R) \defeq \int_{\Gamma_0(q) \backslash \Hb} \left(\frac{1}{\vol(B_R)} \int_{B_R(w)} |g(z)|^2 \, d\mu(z) - \frac{1}{\vol(\Gamma_0(q) \backslash \Hb)}\right)^2 \, d\mu(w).\]
This reduces the problem to bounding this variance. For $q = 1$, the first author showed that if $R \gg t_g^{-\delta}$ with $0 < \delta < 1$, then $\Var(g;R) = o(1)$ under the assumption of the generalised Lindel\"{o}f hypothesis \cite[Proposition 5.1]{Hum18}; an analogous result is also proved unconditionally for $g(z)$ equal to an Eisenstein series $E(z,1/2 + it_g)$ \cite[Proposition 5.5]{Hum18}. The barrier $R \asymp t_g^{-1}$ is the Planck scale, at which equidistribution need not hold \cite[Theorem 1.14]{Hum18}; as discussed in \cite[Section 5.1]{HR92}, the topography of Maa\ss{} forms below this scale is ``essentially sinusoidal'' and so Maa\ss{} forms should not be expected to exhibit random behaviour, such as mass equidistribution, at such minuscule scales.

\subsection{The Fourth Moment of a Maa\ss{} Form}

Another manifestation of Berry's conjecture is the Gaussian moments conjecture (see \cite[Conjecture 1.1]{Hum18}), which states that the (suitably normalised) $n$-th moment of a real-valued Maa\ss{} newform $g$ restricted to a fixed compact subset $K$ of $\Gamma_0(q) \backslash \Hb$ should converge to the $n$-th moment of a real-valued Gaussian random variable with mean $0$ and variance $1$ as $t_g$ tends to infinity. A similar conjecture may also be posed for complex-valued Maa\ss{} newforms, as well as for holomorphic newforms in the large weight limit; cf.~\cite[Conjectures 1.2 and 1.3]{BKY13}. A closely related conjecture, namely essentially sharp upper bounds for $L^p$-norms of automorphic forms, has been posed by Sarnak \cite[Conjecture 4]{Sar03}. For $n = 2$, the Gaussian moments conjecture is simply quantum unique ergodicity, and for small values of $n$, this is also conjectured to be true for noncompact $K$ (but not for large $n$; cf.~\cite[Section 1.1.2]{Hum18}).

The fourth moment is of particular interest, for, as first observed by Sarnak \cite[p.~461]{Sar03}, it can be expressed as a spectral sum of $L$-functions. The conjecture takes the following form for $K = \Gamma_0(q) \backslash \Hb$.

\begin{conjecture}
\label{fourthmomentconj}
As $t_g$ tends to infinity along a subsequence of real-valued newforms $g \in \BB_0^{\ast}(q,\chi)$,
\[\int_{\Gamma_0(q) \backslash \Hb} |g(z)|^4 \, d\mu(z) = \frac{3}{\vol(\Gamma_0(q) \backslash \Hb)} + o_q(1).\]
\end{conjecture}

This has been proven for $q = 1$ conditionally under the generalised Lindel\"{o}f hypothesis by Buttcane and the second author \cite[Theorem 1.1]{BuK17b}, but an unconditional proof currently seems well out of reach (cf.~\cite[Remark 3.3]{Hum18} and \hyperref[subconvexityobstacleremark]{Remark \ref*{subconvexityobstacleremark}}). Djankovi\'{c} and the second author have formulated \cite{DK18a} and subsequently proven \cite[Theorem 1.1]{DK18b} a regularised version of this conjecture for Eisenstein series, improving upon earlier work of Spinu \cite[Theorem 1.1 (A)]{Spi03} that proves the upper bound $O_{\e}(t_g^{\e})$ in this setting. Numerical investigations of this conjecture for the family of dihedral Maa\ss{} newforms have also been undertaken by Hejhal and Str\"{o}mbergsson \cite{HS01}, and the upper bound $O_{q,\e}(t_g^{\e})$ for dihedral forms has been proven by Luo \cite[Theorem]{Luo14} (cf.~\hyperref[largesieveboundsremark]{Remark \ref*{largesieveboundsremark}}). Furthermore, bounds for the fourth moment in the level aspect have also been investigated by many authors \cite{Blo13,BuK15,Liu15,LMY13}.

\subsection{Results}

This paper gives the first unconditional resolutions of \hyperref[Planckconj]{Conjectures \ref*{Planckconj}} and \ref{fourthmomentconj} for a family of cusp forms. We prove these two conjectures in the particular case when $q = D \equiv 1 \pmod{4}$ is a fixed positive squarefree fundamental discriminant, $\chi = \chi_D$ is the primitive quadratic character modulo $D$, and $t_g$ tends to infinity along any subsequence of dihedral Maa\ss{} newforms $g = g_{\psi} \in \BB_0^{\ast}(D,\chi_D)$.

\begin{theorem}
\label{Planckthm}
Let $D \equiv 1 \pmod{4}$ be a positive squarefree fundamental discriminant and let $\chi_D$ be the primitive quadratic character modulo $D$. Suppose that $R \gg t_g^{-\delta}$ for some $0 < \delta < 1$. Then there exists $\delta' > 0$ dependent only on $\delta$ such that
\begin{equation}
\label{Vargpsiboundeq}
\Var\left(g_{\psi};R\right) \ll_D t_g^{-\delta'}
\end{equation}
as the spectral parameter $t_g$ tends to infinity along any subsequence of dihedral Maa\ss{} newforms $g_{\psi} \in \BB_0^{\ast}(D,\chi_D)$. Consequently,
\[\vol\left(\left\{w \in \Gamma_0(D) \backslash \Hb : \left|\frac{1}{\vol(B_R)} \int_{B_R(w)} |g_{\psi}(z)|^2 \, d\mu(z) - \frac{1}{\vol(\Gamma_0(D) \backslash \Hb)}\right| > c\right\}\right)\]
tends to zero as $t_g$ tends to infinity for any fixed $c > 0$.
\end{theorem}

\begin{theorem}
\label{fourthmomentthm}
Let $D \equiv 1 \pmod{4}$ be a positive squarefree fundamental discriminant and let $\chi_D$ be the primitive quadratic character modulo $D$. Then there exists an absolute constant $\delta' > 0$ such that
\begin{equation}
\label{fourthmomenteq}
\int_{\Gamma_0(D) \backslash \Hb} |g_{\psi}(z)|^4 \, d\mu(z) = \frac{3}{\vol(\Gamma_0(D) \backslash \Hb)} + O_D(t_g^{-\delta'})
\end{equation}
as $t_g$ tends to infinity along any subsequence of dihedral Maa\ss{} newforms $g_{\psi} \in \BB_0^{\ast}(D,\chi_D)$.
\end{theorem}

Dihedral newforms form a particularly thin subsequence of Maa\ss{} forms; the number of dihedral Maa\ss{} newforms with spectral parameter less than $T$ is asymptotic to $c_{1,D} T$, whereas the number of Maa\ss{} newforms with spectral parameter less than $T$ is asymptotic to $c_{2,D} T^2$, where $c_{1,D},c_{2,D} > 0$ are constants dependent only on $D$. We explain in \hyperref[connectionssect]{Section \ref*{connectionssect}} the properties of dihedral Maa\ss{} newforms, not shared by nondihedral forms, that are crucial to our proofs of \hyperref[Planckthm]{Theorems \ref*{Planckthm}} and \ref{fourthmomentthm}.

\begin{remark}
Previous work \cite{Blo13,BuK15,BuK17a,Liu15,LMY13,Luo14} on the fourth moment has been subject to the restriction that $D$ be a prime. We weaken this restriction to $D$ being squarefree. The additional complexity that arises is determining explicit expressions for the inner product of $|g|^2$ with oldforms. Removing the squarefree restriction on $D$, while likely presently feasible, would undoubtedly involve significant extra work.
\end{remark}

\begin{remark}
An examination of the proofs of \hyperref[Planckthm]{Theorems \ref*{Planckthm}} and \ref{fourthmomentthm} shows that the dependence on $D$ in the error terms in \eqref{Vargpsiboundeq} and \eqref{fourthmomenteq} is polynomial.
\end{remark}

\subsection*{Notation}

Throughout this article, we make use of the $\e$-convention: $\e$ denotes an arbitrarily small positive constant whose value may change from occurrence to occurrence. Results are stated involving level $D$ when only valid for positive squarefree $D \equiv 1 \pmod{4}$ and are stated involving level $q$ otherwise. The primitive quadratic character modulo $D$ will always be denoted by $\chi_D$. Since we regard $D$ as being fixed, all implicit constants in Vinogradov $\ll$ and big O notation may depend on $D$ unless otherwise specified. We write $\N_0 \defeq \N \cup \{0\}$ for the nonnegative integers. A dihedral Maa\ss{} newform will be written as $g_{\psi} \in \BB_0^{\ast}(D,\chi_D)$; this is associated to a Hecke Gr\"{o}\ss{}encharakter $\psi$ of $\Q(\sqrt{D})$ as described in \hyperref[Grosssect]{Appendix \ref*{Grosssect}}.

\subsection{Elements of the Proofs}

The proofs of \hyperref[Planckthm]{Theorems \ref*{Planckthm}} and \ref{fourthmomentthm}, which we give in \hyperref[proofPlanckthmsect]{Section \ref*{proofPlanckthmsect}}, follow by combining three key tools; the approach that we follow is that first pioneered by Sarnak \cite[p.~461]{Sar03} and Spinu \cite{Spi03}.

First, we spectrally expand the variance and the fourth moment, obtaining the following explicit formul\ae{}.

\begin{proposition}
\label{dihedralspectralprop}
Let $q$ be squarefree and let $\chi$ be a primitive Dirichlet character modulo $q$. Then for a newform $g \in \BB_0^{\ast}(q,\chi)$, the variance $\Var(g;R)$ is equal to
\begin{multline}
\label{dihedralspectraleq}
\sum_{q_1 q_2 = q} 2^{\omega(q_2)} \frac{\nu(q_2) \varphi(q_2)}{q_2^2} \sum_{f \in \BB_0^{\ast}\left(\Gamma_0\left(q_1\right)\right)} \frac{L_{q_2}(1,\sym^2 f)}{L_{q_2}\left(\frac{1}{2}, f\right)} \left|h_R(t_f)\right|^2 \left|\left\langle |g|^2, f\right\rangle_q\right|^2	\\
+ \frac{2^{\omega(q)}}{4\pi} \int_{-\infty}^{\infty} \left|h_R(t)\right|^2 \left|\left\langle |g|^2, E_{\infty}\left(\cdot, \frac{1}{2} + it\right)\right\rangle_q\right|^2 \, dt,
\end{multline}
where $\BB_0^{\ast}(\Gamma_0(q_1)) \ni f$ is an orthonormal basis of the space of newforms of weight zero, level $q_1$, and principal nebentypus, normalised such that $\langle f, f\rangle_q = 1$, $E_{\infty}(z,s)$ denotes the Eisenstein series associated to the cusp at infinity of $\Gamma_0(q) \backslash \Hb$, and
\[h_R(t) \defeq \frac{R}{\pi \sinh \frac{R}{2}} \int_{-1}^{1} \sqrt{1 - \left(\frac{\sinh \frac{Rr}{2}}{\sinh \frac{R}{2}}\right)^2} e^{iRrt} \, dr.\]

Similarly, the fourth moment $\int_{\Gamma_0(q) \backslash \Hb} |g(z)|^4 \, d\mu(z)$ is equal to
\begin{multline}
\label{dihedralspectral2eq}
\frac{1}{\vol(\Gamma_0(q) \backslash \Hb)} + \sum_{q_1 q_2 = q} 2^{\omega(q_2)} \frac{\nu(q_2) \varphi(q_2)}{q_2^2} \sum_{f \in \BB_0^{\ast}\left(\Gamma_0\left(q_1\right)\right)} \frac{L_{q_2}(1,\sym^2 f)}{L_{q_2}\left(\frac{1}{2}, f\right)} \left|\left\langle |g|^2, f\right\rangle_q\right|^2	\\
+ \frac{2^{\omega(q)}}{4\pi} \int_{-\infty}^{\infty} \left|\left\langle |g|^2, E_{\infty}\left(\cdot, \frac{1}{2} + it\right)\right\rangle_q\right|^2 \, dt.
\end{multline}
\end{proposition}

The arithmetic functions $\omega,\nu,\varphi$ are defined by $\omega(n) \defeq \# \left\{p \mid n\right\}$, $\nu(n) \defeq n \prod_{p \mid n} (1 + p^{-1})$, and $\varphi(n) \defeq n \prod_{p \mid n} (1 - p^{-1})$. We have written $L_p(s,\pi)$ for the $p$-component of the Euler product of an $L$-function $L(s,\pi)$, while
\[L_q(s,\pi) \defeq \prod_{p \mid q} L_p(s,\pi), \qquad L^q(s,\pi) \defeq \frac{L(s,\pi)}{L_q(s,\pi)}, \qquad \Lambda^q(s,\pi) \defeq \frac{\Lambda(s,\pi)}{L_q(s,\pi)},\]
where $\Lambda(s,\pi) \defeq q(\pi)^{s/2} L_{\infty}(s,\pi) L(s,\pi)$ denotes the completed $L$-function with conductor $q(\pi)$ and archimedean component $L_{\infty}(s,\pi)$.

Next, we obtain explicit expressions in terms of $L$-functions for the inner products $|\langle |g|^2, f\rangle_q|^2$ and $|\langle |g|^2, E_{\infty}(\cdot,1/2 + it)\rangle|^2$; this is the Watson--Ichino formula.

\begin{proposition}
\label{dihedraltripleproductprop}
Let $q = q_1 q_2$ be squarefree and let $\chi$ be a primitive Dirichlet character modulo $q$. Then for $g \in \BB_0^{\ast}(q,\chi)$ and for $f \in \BB_0^{\ast}(\Gamma_0(q_1))$ of parity $\epsilon_f \in \{1,-1\}$ normalised such that $\langle g, g\rangle_q = \langle f, f\rangle_q = 1$,
\begin{equation}
\label{gpsi^2fWatsonIchino}
\left|\left\langle |g|^2, f\right\rangle_q\right|^2 = \frac{1 + \epsilon_f}{16 \sqrt{q_1} \nu(q_2)} \frac{\Lambda\left(\frac{1}{2}, f\right) \Lambda\left(\frac{1}{2}, f \otimes \ad g\right)}{\Lambda(1, \ad g)^2 \Lambda(1,\sym^2 f)}.
\end{equation}

Similarly,
\begin{equation}
\label{gpsi^2EWatsonIchino}
\left|\left\langle |g|^2, E_{\infty}\left(\cdot, \frac{1}{2} + it\right)\right\rangle_q\right|^2 = \frac{1}{4q} \left|\frac{\Lambda^q\left(\frac{1}{2} + it\right) \Lambda\left(\frac{1}{2} + it, \ad g\right)}{\Lambda(1, \ad g) \Lambda^q(1 + 2it)}\right|^2.
\end{equation}
\end{proposition}

Now we specialise to $g = g_{\psi} \in \BB_0^{\ast}(D,\chi_D)$. Observe that $\ad g_{\psi}$ is equal to the (noncuspidal) isobaric sum $\chi_D \boxplus g_{\psi^2}$, where $g_{\psi^2} \in \BB_0^{\ast}(D,\chi_D)$ is the dihedral Maa\ss{} newform associated to the Hecke Gr\"{o}\ss{}encharakter $\psi^2$ of $\Q(\sqrt{D})$, and so
\begin{align*}
\Lambda(s, f \otimes \ad g_{\psi}) & = \Lambda(s, f \otimes \chi_D) \Lambda(s, f \otimes g_{\psi^2}),	\\
\Lambda(s, \ad g_{\psi}) & = \Lambda(s, \chi_D) \Lambda(s,g_{\psi^2}),
\end{align*}
which can readily be seen by comparing Euler factors. Then the identity \eqref{gpsi^2fWatsonIchino} holds with $1 + \epsilon_f$ replaced by $2$ as both sides vanish when $f$ is odd: the right-hand side vanishes due to the fact that $\Lambda(1/2,f \otimes \chi_D) = \Lambda(1/2,f) \Lambda(1/2,f \otimes g_{\psi^2}) = 0$, for \hyperref[rootnumberlemma]{Lemma \ref*{rootnumberlemma}} shows that the root number in both cases is $-1$, while the left-hand side vanishes since one can make the change of variables $z \mapsto -\overline{z}$ in the integral over $\Gamma_0(D) \backslash \Hb$, which leaves $|g_{\psi}(z)|^2$ unchanged but replaces $f(z)$ with $-f(z)$.

We have thereby reduced both problems to subconvex moment bounds. To this end, for a function $h : \R \cup i(-1/2,1/2) \to \C$, we define the mixed moments
\begin{align}
\label{MMMaasseq}
\MM^{\Maass}(h) & \defeq \sum_{d_1 d_2 = D} 2^{\omega(d_2)} \frac{\varphi(d_2)}{d_2} \sum_{f \in \BB_0^{\ast}(\Gamma_0(d_1))} \frac{L^{d_2}\left(\frac{1}{2},f\right) L\left(\frac{1}{2},f \otimes \chi_D\right) L\left(\frac{1}{2},f \otimes g_{\psi^2}\right)}{L^{d_2}(1,\sym^2 f)} h(t_f),	\\
\label{MMEiseq}
\MM^{\Eis}(h) & \defeq \frac{2^{\omega(D)}}{2\pi} \int_{-\infty}^{\infty} \left|\frac{\zeta^D\left(\frac{1}{2} + it\right) L\left(\frac{1}{2} + it, \chi_D\right) L\left(\frac{1}{2} + it, g_{\psi^2}\right)}{\zeta^D(1 + 2it)}\right|^2 h(t) \, dt.
\end{align}
We prove the following bounds for these terms for various choices of function $h$.

\begin{proposition}
\label{dihedralmomentsprop}
There exists some $\alpha > 0$ and a constant $\delta > 0$ such that the following hold:
\begin{enumerate}[leftmargin=*]
\item[\emph{(1)}] For $h(t) = 1_{E \cup -E}(t)$ with $E = [T,2T]$ and $T \leq t_g^{1 - \alpha}$,
\[\MM^{\Maass}(h) + \MM^{\Eis}(h) \ll T t_g^{1 - \delta}.\]
\item[\emph{(2)}] For
\[h(t) = \frac{\pi H(t) 1_{E \cup -E}(t)}{8D^2 L(1,\chi_D)^2 L(1,g_{\psi^2})^2}\]
with $H(t)$ as in \eqref{H(t)defeq} and $E = (t_g^{1 - \alpha}, 2t_g - t_g^{1 - \alpha})$,
\[\MM^{\Maass}(h) + \MM^{\Eis}(h) = \frac{2}{\vol(\Gamma_0(D) \backslash \Hb)} + O(t_g^{-\delta}).\]
\item[\emph{(3)}] For $h(t) = 1_{E \cup -E}(t)$ with $E = [T - U,T + U]$, where $2t_g - t_g^{1 - \alpha} \leq T \leq 2t_g + t_g^{1 - \alpha}$ and $\max\{|2t_g - T|,T^{1/3}\} \ll U \leq T$,
\[\MM^{\Maass}(h) + \MM^{\Eis}(h) \ll_{\e} (TU)^{1 + \e}.\]
\item[\emph{(4)}] For $h(t) = 1_{E \cup -E}(t)$ with $E = [T,2T]$ and $T > 2t_g + t_g^{1 - \alpha}$,
\[\MM^{\Maass}(h) + \MM^{\Eis}(h) \ll_{\e} T^{2 + \e}.\]
\item[\emph{(5)}] For $h(t) = 1_{E \cup -E}(t)$ with $E = i(0,1/2)$,
\[\MM^{\Maass}(h) \ll t_g^{1 - \delta}.\]
\end{enumerate}
\end{proposition}

As in \cite[Section 3.2]{Hum18}, this covers the five ranges of the spectral expansion:
\begin{enumerate}[leftmargin=*]
\item[(1)] the short initial range $[-t_g^{1 - \alpha},t_g^{1 - \alpha}]$,
\item[(2)] the bulk range $(-2t_g + t_g^{1 - \alpha}, -t_g^{1 - \alpha}) \cup (t_g^{1 - \alpha}, 2t_g - t_g^{1 - \alpha})$,
\item[(3)] the short transition range $[-2t_g - t_g^{\alpha}, -2t_g + t_g^{1 - \alpha}] \cup [2t_g - t_g^{1 - \alpha}, 2t_g + t_g^{1 - \alpha}]$,
\item[(4)] the tail range $(-\infty, -2t_g - t_g^{1 - \alpha}) \cup (2t_g + t_g^{1 - \alpha},\infty)$, and
\item[(5)] the exceptional range $i(-1/2,1/2) \setminus \{0\}$.
\end{enumerate}

\begin{remark}
For the purposes of proving \hyperref[Planckthm]{Theorem \ref*{Planckthm}}, the exact identities in \hyperref[dihedralspectralprop]{Propositions \ref*{dihedralspectralprop}} and \ref{dihedraltripleproductprop} as well as the asymptotic formula in \hyperref[dihedralmomentsprop]{Proposition \ref*{dihedralmomentsprop} (2)} are superfluous, for we could make do with upper bounds in each case in order to prove the desired upper bound for $\Var(g_{\psi};R)$. These identities, however, are necessary to prove the desired asymptotic formula for the fourth moment of $g_{\psi}$ in \hyperref[fourthmomentthm]{Theorem \ref*{fourthmomentthm}}.
\end{remark}

\begin{remark}
\label{largesieveboundsremark}
The large sieve yields with relative ease the bounds $O_{\e}((Tt_g)^{1 + \e})$ and $O_{\e}(t_g^{\e})$ for \hyperref[dihedralmomentsprop]{Proposition \ref*{dihedralmomentsprop} (1) and (2)} respectively; dropping all but one term then only yields the convexity bound for the associated $L$-functions. These weaker bounds imply that the variance $\Var(g_{\psi};R)$ and the fourth moment of $g_{\psi}$ are both $O_{\e}(t_g^{\e})$, with the latter being a result of Luo \cite[Theorem]{Luo14} and the former falling just short of proving small scale mass equidistribution.
\end{remark}

\subsection{A Sketch of the Proofs and the Structure of the Paper}

We briefly sketch the main ideas behind the proofs of \hyperref[dihedralspectralprop]{Propositions \ref*{dihedralspectralprop}}, \ref{dihedraltripleproductprop}, and \ref{dihedralmomentsprop}.

The proof of \hyperref[dihedralspectralprop]{Proposition \ref*{dihedralspectralprop}}, given in \hyperref[spectralexpansionsect]{Section \ref*{spectralexpansionsect}}, uses the spectral decomposition of $L^2(\Gamma_0(q) \backslash \Hb)$ and Parseval's identity to spectrally expand the variance and the fourth moment. We then require an orthonormal basis in terms of newforms and translates of oldforms together with an explicit description of the action of Atkin--Lehner operators on these Maa\ss{} forms in order to obtain \eqref{dihedralspectraleq} and \eqref{dihedralspectral2eq}.

\hyperref[dihedraltripleproductprop]{Proposition \ref*{dihedraltripleproductprop}} is an explicit form of the Watson--Ichino formula, which relates the integral of three $\GL_2$-automorphic forms to a special value of a triple product $L$-function; we present this material in \hyperref[WatsonIchinosect]{Section \ref*{WatsonIchinosect}}. To ensure that the identities \eqref{gpsi^2fWatsonIchino} and \eqref{gpsi^2EWatsonIchino} are correct not merely up to multiplication by an unspecified constant requires a careful translation of the ad\`{e}lic identity \cite[Theorem 1.1]{Ich08} into the classical language of automorphic forms. Moreover, this identity involves local constants at ramified primes, and the precise set-up of our problem involves determining such local constants, which is undertaken in \hyperref[localconstantsect]{Section \ref*{localconstantsect}}. This problem of the determination of local constants in the Watson--Ichino formula is of independent interest; see, for example, \cite{Col18,Col19,Hu16,Hu17,Wat08}.

The proof of \hyperref[dihedralmomentsprop]{Proposition \ref*{dihedralmomentsprop}} takes up the bulk of this paper, for it is rather involved and requires several different strategies to deal with various ranges. The many (predominantly) standard automorphic tools used in the course of the proof, such as the approximate functional equation, the Kuznetsov formula, and the large sieve, are relegated to \hyperref[toolboxappendix]{Appendix \ref*{toolboxappendix}}; we recommend that on first reading, the reader familiarise themself with these tools via a quick perusal of \hyperref[toolboxappendix]{Appendix \ref*{toolboxappendix}} before continuing on to the proof of \hyperref[dihedralmomentsprop]{Proposition \ref*{dihedralmomentsprop}} that begins in \hyperref[firstmomentsect]{Section \ref*{firstmomentsect}}.

\hyperref[dihedralmomentsprop]{Proposition \ref*{dihedralmomentsprop} (1)}, proven in \hyperref[(1)sect]{Section \ref*{(1)sect}}, requires three different treatments for three different parts of the short initial range. We may use hybrid subconvex bounds for $L(1/2,f \otimes g_{\psi^2})$ and $|L(1/2 + it,g_{\psi^2})|^2$ due to Michel and Venkatesh \cite{MV10} to treat the range $T \leq t_g^{\beta}$ for an absolute constant $\beta > 0$. For $t_g^{\beta} < T \leq t_g^{1/2}$, we use subconvex bounds for $L(1/2, f \otimes \chi_D)$ and $|L(1/2 + it,\chi_D)|^2$ due to Young \cite{You17} together with bounds proven in \hyperref[firstmomentsect]{Section \ref*{firstmomentsect}} for the first moment of $L(1/2,f \otimes g_{\psi^2})$ and of $|L(1/2 + it,g_{\psi^2})|^2$. This approach relies crucially on the nonnegativity of $L(1/2,f \otimes g_{\psi^2})$ (see, for example, the discussion on this point in \cite[Section 1.1]{HT14}). Bounds for the remaining range $t_g^{1/2} < T \leq t_g^{1 - \alpha}$ for \hyperref[dihedralmomentsprop]{Proposition \ref*{dihedralmomentsprop} (1)} are shown in \hyperref[spectralrecsect]{Sections \ref*{spectralrecsect}} and \ref{boundstranssect} to follow from the previous bounds for the range $t_g^{\alpha} \ll T \ll t_g^{1/2}$. This is spectral reciprocity: via the triad of Kuznetsov, Vorono\u{\i}, and Kloosterman summation formul\ae{} (the latter being the Kuznetsov formula in the formulation that expresses sums of Kloosterman sums in terms of Fourier coefficients of automorphic forms), bounds of the form
\[\MM^{\Maass}(h) + \MM^{\Eis}(h) \ll T t_g^{1 - \delta}\]
with $h(t) = 1_{E \cup -E}(t)$ for $E = [T,2T]$ are essentially implied by the same bounds with $E = [t_g/T,2t_g/T]$ together with analogous bounds for moments involving holomorphic cusp forms of even weight $k \in [t_g/T,2t_g/T]$.

The proof of \hyperref[dihedralmomentsprop]{Proposition \ref*{dihedralmomentsprop} (2)} for the bulk range, appearing in \hyperref[(2)sect]{Section \ref*{(2)sect}}, mimics that of the analogous result for Eisenstein series given in \cite{DK18b}. As such, we give a laconic sketch of the proof, highlighting mainly the slight differences compared to the Eisenstein case.

\hyperref[dihedralmomentsprop]{Proposition \ref*{dihedralmomentsprop} (3)} is proven in \hyperref[(3)sect]{Section \ref*{(3)sect}} and relies upon the Cauchy--Schwarz inequality; the resulting short second moment of Rankin--Selberg $L$-functions is bounded via the large sieve, while a bound is also required for a short mixed moment of four $L$-functions. This latter bound is again a consequence of spectral reciprocity, akin to \cite[Theorem]{Jut01}, and is detailed in \hyperref[spectralrec2sect]{Sections \ref*{spectralrec2sect}} and \ref{boundstrans2sect}.

In \hyperref[(4)sect]{Section \ref*{(4)sect}}, we show that \hyperref[dihedralmomentsprop]{Proposition \ref*{dihedralmomentsprop} (4)} is a simple consequence of the large sieve, while \hyperref[dihedralmomentsprop]{Proposition \ref*{dihedralmomentsprop} (5)} is shown in \hyperref[(5)sect]{Section \ref*{(5)sect}} to follow once more from hybrid subconvex bounds for $L(1/2,f \otimes g_{\psi^2})$ and $|L(1/2 + it,g_{\psi^2})|^2$ due to Michel and Venkatesh \cite{MV10}.

\subsection{Further Heuristics}

We give some very rough back-of-the-envelope type calculations to go along with the sketch above. \hyperref[dihedralmomentsprop]{Proposition \ref*{dihedralmomentsprop}} requires the evaluation of a mean value of $L$-functions looking essentially like 
\[\sum_{t_f < 2t_g} \frac{L\left(\frac{1}{2}, f\right)^2 L\left(\frac{1}{2}, f\otimes g_{\psi^2}\right) }{t_f t_g^{1/2} (1 + |2t_g - t_f|)^{1/2}},\]
where we pretend that $D$ equals $1$, since it is anyway fixed. The goal is to extract the main term with an error term bounded by a negative power of $t_g$. The expression remains unchanged if the summand is multiplied by the parity $\epsilon_f = \pm 1$ of $f$, because $L(1/2,f) = 0$ when $\epsilon_f = -1$. Summing over $t_f$ using the opposite-sign case of the Kuznetsov formula gives, in the dyadic range $ t_f\sim T$, an off-diagonal of the shape 
\[\frac{1}{t_g^{1/2} (1 + |2t_g - T|)^{1/2}} \sum_{\substack{n \sim T^2 \\ m \sim t_g (1 + |2t_g - T|)}} \frac{\lambda_{g_{\psi^2}}(m) d(n)}{\sqrt{mn}} \sum_{c \sim t_g^{1/2} (1 + |2t_g - T|)^{1/2}} \frac{S(m,n;c)}{c},\]
where $d(n)$ is the divisor function. Note that for the sake of argument, we use approximate functional equations, although our proof works with Dirichlet series in regions of absolute convergence and continues meromorphically at the last possible moment.

Consider the case $t_g^{\alpha} \leq T \leq 2t_g - t_g^{1 - \alpha}$, which includes the short initial and bulk ranges, so that $m \sim t_g^2$ and $c \sim t_g$. Applying the Vorono\u{\i} summation formula to both $n$ and $m$ returns a sum like
\[\frac{T}{t_g^4} \sum_{n \sim \frac{t_g^2}{T^2}} \sum_{m \sim t_g^2} \sum_{c \sim t_g} \lambda_{g_{\psi^2}}(m) d(n) S(m,n;c).\]
Note that $c \sim (T/2t_g) \sqrt{mn}$, so applying the Kloosterman summation formula gives
\[\frac{T^2}{t_g^4} \sum_{t_f < \frac{2t_g}{T}} \sum_{n \sim \frac{t_g^2}{T^2}} \sum_{m \sim t_g^2} \lambda_f(n) d(n) \lambda_f(m) \lambda_{g_{\psi^2}}(m).\]
This can be recast as essentially
\[\sum_{t_f < \frac{2t_g}{T}} \frac{L\left(\frac{1}{2}, f\right)^2 L\left(\frac{1}{2}, f\otimes g_{\psi^2}\right)}{t_f t_g^{1/2} (1 + |2t_g - t_f|)^{1/2}}.\]
The phenomenon of the same mean value of $L$-functions reappearing but with the range of summation now reciprocated to $t_f < 2t_g/T$ is spectral reciprocity, as alluded to above.

When $T \sim t_g$, the bulk range, we immediately get a satisfactory estimate by inserting subconvexity bounds. When $T < t_g^{1 - \alpha}$, the short initial range, we are not done right away, but we at least reduce to the case $T < t_g^{1/2}$. In this range, we must use a new approach. The idea is to bound, using nonnegativity of central values, $L(1/2,f)^2$ by subconvexity bounds and then to estimate the first moment $\sum_{t_f \sim T} L(1/2, f\otimes g_{\psi^2})$. This is not an easy task because the sum over $t_f$ is very short. We expand the first moment using approximate functional equations, apply the Kuznetsov formula, use the Vorono\u{\i} summation formula, and then estimate; this turns out to be sufficient. Finally, it remains to consider the short transition range $|t_f - 2t_g| \sim T$ with $|T| < t_g^{1 - \alpha}$. Here the strategy is to apply the Cauchy--Schwarz inequality and consider $\sum_{t_f} L(1/2,f)^4$ and $\sum_{t_f} L(1/2, f\otimes g_{\psi^2})^2$, the latter of which can be estimated sharply using the spectral large sieve, while the former can be bounded once again via spectral reciprocity.

\subsection{Related Results for the Fourth Moment and Spectral Reciprocity}

Bounds of the form $O_{\e}(t_g^{\e})$ for the fourth moment of the truncation of an Eisenstein series $E(z,1/2 + it_g)$ or for a dihedral Maa\ss{} form $g = g_{\psi}$ have been proven by Spinu \cite{Spi03} and Luo \cite{Luo14} respectively; the proofs use the Cauchy--Schwarz inequality and the large sieve to bound moments of $L$-functions and rely on the factorisation of the $L$-functions appearing in the Watson--Ichino formula. In applying the large sieve to the bulk range, this approach loses the ability to obtain an asymptotic formula.

Sarnak and Watson \cite[Theorem 3(a)]{Sar03} noticed that via the $\GL_3$ Vorono\u{\i} summation formula coupled with the convexity bound for $L(1/2,f \otimes \sym^2 g)$, one could prove the bound $O_{\e}(t_g^{\e})$ for the bulk range of the spectral expansion of the fourth moment of a Maa\ss{} cusp form (cf.~\cite[Remark 3.3]{Hum18}). This approach was expanded upon by Buttcane and the second author \cite{BuK17b}, where an asymptotic for this bulk range was proven under the assumption of the generalised Lindel\"{o}f hypothesis. Asymptotics for a moment closely related to that appearing in \hyperref[dihedralmomentsprop]{Proposition \ref*{dihedralmomentsprop} (2)} are proven in \cite{BuK17a}; the method is extremely similar to that used in \cite{BuK17b}. Finally, asymptotics for the bulk range appearing in the spectral expansion of the regularised fourth moment of an Eisenstein series are proven in \cite{DK18b} (and \hyperref[dihedralmomentsprop]{Proposition \ref*{dihedralmomentsprop} (2)} is proven via minor modifications of this proof). These results all follow via the triad of Kuznetsov, Vorono\u{\i}, and Kloosterman summation formul\ae{}, and are cases of spectral reciprocity: the moment of $L$-functions in the bulk range is shown to be equal to a main term together with a moment of $L$-functions that is essentially extremely short, namely involving forms $f$ for which $t_f \ll t_g^{\e}$.

This nonetheless leaves the issue of dealing with the short initial and transition ranges. Assuming the generalised Lindel\"{o}f hypothesis, it is readily seen that these are negligible. Spectral reciprocity in the short initial range is insufficient to prove this, since it merely replaces the problem of bounding the contribution from the range $[T,2T]$ with that of the range $[T/t_g,2T/t_g]$. Our key observation is that spectral reciprocity reduces the problem to the range $T < t_g^{1/2}$, at which point we may employ a different strategy, namely subconvex bounds for $L(1/2,f) L(1/2, f \otimes \chi_D)$ together with a bound for the first moment of $L(1/2, f \otimes g_{\psi^2})$. This approach, albeit in a somewhat disguised form, is behind the success of the unconditional proofs of the negligibility of the short initial and transition ranges for the regularised fourth moment of an Eisenstein series. These follow from the work of Jutila \cite{Jut01} and Jutila and Motohashi \cite{JM05}; see \cite[Lemmata 3.7 and 3.8]{Hum18}.

\subsection{Connections to Subconvexity}
\label{connectionssect}

Quantifying the rate of equidistribution for quantum unique ergodicity in terms of bounds for \eqref{QUEfErateseq} is, via the Watson--Ichino formula, equivalent to determining subconvex bounds for $L(1/2,f \otimes \ad g)$ in the $t_g$-aspect. Such bounds are yet to be proven except in a select few cases, namely when $g$ is dihedral or an Eisenstein series, where $L(1/2, f \otimes \ad g)$ factorises as
\[\begin{dcases*}
L\left(\frac{1}{2}, f \otimes \chi_D\right) L\left(\frac{1}{2}, f \otimes g_{\psi^2}\right) & if $g = g_{\psi} \in \BB_0^{\ast}(D,\chi_D)$ is dihedral,	\\
L\left(\frac{1}{2}, f\right) L\left(\frac{1}{2} + 2it_g, f\right) L\left(\frac{1}{2} - 2it_g, f\right) & if $g(z) = E(z,1/2 + it_g)$.
\end{dcases*}\]
Indeed, quantum unique ergodicity was already known for Eisenstein series \cite{LS95} before the work of Lindenstrauss \cite{Lin06} and Soundararajan \cite{Sou10}, and for dihedral Maa\ss{} forms \cite{Blo05} with quantitative bounds for \eqref{QUEfErateseq} shortly thereafter (see also \cite{Sar01,LY02,LLY06a,LLY06b}). The proofs of \hyperref[Planckthm]{Theorems \ref*{Planckthm}} and \ref{fourthmomentthm}, as well as their Eisenstein series counterparts \cite{DK18b,Hum18}, rely crucially on these factorisations, and the chief hindrance behind the lack of an unconditional proof of these theorems for an arbitrary Maa\ss{} cusp form is the lack of such a factorisation.

In proving \hyperref[Planckthm]{Theorem \ref*{Planckthm}}, on the other hand, we require bounds for the moments given in \hyperref[dihedralmomentsprop]{Proposition \ref*{dihedralmomentsprop}}, most notably in the range $E = [T,2T]$ with $T < t_g^{1 - \alpha}$. Dropping all but one term in this range implies the hybrid subconvex bounds
\begin{align*}
L\left(\frac{1}{2},f\right) L\left(\frac{1}{2},f \otimes \chi_D\right) L\left(\frac{1}{2},f \otimes g_{\psi^2}\right) & \ll t_f t_g^{1 - \delta},	\\
\left|\zeta\left(\frac{1}{2} + it\right) L\left(\frac{1}{2} + it, \chi_D\right) L\left(\frac{1}{2} + it, g_{\psi^2}\right)\right|^2 & \ll |t| t_g^{1 - \delta}
\end{align*}
for these products of $L$-functions with analytic conductors $\asymp (t_f t_g)^4$ and $\asymp (|t| t_g)^4$ respectively. Such bounds for product $L$-functions were previously known, and at various points in the proof of \hyperref[dihedralmomentsprop]{Proposition \ref*{dihedralmomentsprop}} we make use of known subconvex bounds for individual $L$-functions in this product; what is noteworthy is that individual subconvex bounds are insufficient for proving \hyperref[Planckthm]{Theorems \ref*{Planckthm}} and \ref{fourthmomentthm}, but rather bounds for moments that imply subconvexity are required.

\begin{remark}
\label{subconvexityobstacleremark}
This demonstrates the difficulty of proving \hyperref[Planckthm]{Theorems \ref*{Planckthm}} and \ref{fourthmomentthm} unconditionally for arbitrary Hecke--Maa\ss{} eigenforms $g$: as mentioned in \cite[p.~1493]{BuK17b}, we would require a subconvex bound of the form $L(1/2,f \otimes \ad g) \ll t_g^{1 - \delta}$ uniformly in $t_f < t_g^{\delta'}$ for some $\delta' > 0$, a well-known open problem. On the other hand, Sarnak \cite[Conjecture 4]{Sar03} conjectures the weaker upper bound $O_{\e}(t_g^{\e})$ for the fourth moment of an arbitrary Hecke--Maa\ss{} eigenform $g$, which would not require such a subconvex bound.
\end{remark}

\section{Proofs of \texorpdfstring{\hyperref[Planckthm]{Theorems \ref*{Planckthm}} and \ref{fourthmomentthm}}{Theorems \ref{Planckthm} and \ref{fourthmomentthm}}}
\label{proofPlanckthmsect}

\begin{proof}[{Proofs of {\hyperref[Planckthm]{Theorems \ref*{Planckthm}} and \ref{fourthmomentthm}} assuming {\hyperref[dihedralspectralprop]{Propositions \ref*{dihedralspectralprop}}}, {\ref{dihedraltripleproductprop}}, and {\ref{dihedralmomentsprop}}}]
From \hyperref[dihedralspectralprop]{Propositions \ref*{dihedralspectralprop}} and \ref{dihedraltripleproductprop}, $\Var(g_{\psi};R)$ is equal to the sum of
\begin{multline}
\label{Varcusptermeq}
\frac{\pi}{8D^2 L(1,\chi_D)^2 L(1,g_{\psi^2})^2} \sum_{d_1 d_2 = D} 2^{\omega(d_2)} \frac{\varphi(d_2)}{d_2}	\\
\times \sum_{f \in \BB_0^{\ast}\left(\Gamma_0\left(d_1\right)\right)} \frac{L^{d_2}\left(\frac{1}{2}, f\right) L\left(\frac{1}{2}, f \otimes \chi_D\right) L\left(\frac{1}{2}, f \otimes g_{\psi^2}\right)}{L^{d_2}(1,\sym^2 f)} \left|h_R(t_f)\right|^2 H(t_f)
\end{multline}
and
\begin{equation}
\label{VarEistermeq}
\frac{2^{\omega(D)}}{16D^2 L(1,\chi_D)^2 L(1,g_{\psi^2})^2} \int_{-\infty}^{\infty} \left|\frac{\zeta^D\left(\frac{1}{2} + it\right) L\left(\frac{1}{2} + it, \chi_D\right) L\left(\frac{1}{2} + it, g_{\psi^2}\right)}{\zeta^D(1 + 2it)}\right|^2 \left|h_R(t)\right|^2 H(t) \, dt,
\end{equation}
with
\begin{multline}
\label{H(t)defeq}
H(t) \defeq \frac{\Gamma\left(\frac{1}{4} + \frac{i(2t_g + t)}{2}\right) \Gamma\left(\frac{1}{4} + \frac{i(2t_g - t)}{2}\right) \Gamma\left(\frac{1}{4} - \frac{i(2t_g + t)}{2}\right) \Gamma\left(\frac{1}{4} - \frac{i(2t_g - t)}{2}\right)}{\Gamma\left(\frac{1}{2} + it_g\right)^2 \Gamma\left(\frac{1}{2} - it_g\right)^2}	\\
\times \frac{\Gamma\left(\frac{1}{4} + \frac{it}{2}\right)^2 \Gamma\left(\frac{1}{4} - \frac{it}{2}\right)^2}{\Gamma\left(\frac{1}{2} + it\right) \Gamma\left(\frac{1}{2} - it\right)}.
\end{multline}
Via Stirling's formula
\begin{equation}
\label{Stirlingeq}
\Gamma(s) = \sqrt{2\pi} s^{s - \frac{1}{2}} e^{-s} \left(1 + O\left(\frac{1}{|s|}\right)\right)
\end{equation}
for $\left|\arg s\right| < \pi$ \cite[8.327.1]{GR07},
\begin{multline*}
H(t) = \frac{8\pi e^{-\pi \Omega(t,t_g)}}{(1 + |t|)(1 + |2t_g + t|)^{1/2} (1 + |2t_g - t|)^{1/2}}	\\
\times \left(1 + O\left(\frac{1}{1 + |t|} + \frac{1}{1 + |2t_g + t|} + \frac{1}{1 + |2t_g - t|}\right)\right)
\end{multline*}
for $t \in \R \cup i(-1/2,1/2)$, where
\[\Omega(t,t_g) = \begin{dcases*}
0 & if $|t| \leq 2t_g$,	\\
|t| - 2t_g & if $|t| > 2t_g$.
\end{dcases*}\]
It follows that
\[\Var\left(g_{\psi};R\right) \ll \frac{\MM^{\Maass}(h) + \MM^{\Eis}(h)}{L(1, g_{\psi^2})^2}\]
with
\begin{equation}
\label{h(t)asympeq}
h(t) = \frac{|h_R(t)|^2 e^{-\pi \Omega(t,t_g)}}{(1 + |t|)(1 + |2t_g + t|)^{1/2} (1 + |2t_g - t|)^{1/2}}.
\end{equation}

We recall the bound $L(1, g_{\psi^2}) \gg 1/\log t_g$, as well as \cite[Lemma 4.2]{Hum18}, which states that as $R$ tends to zero,
\begin{equation}
\label{hR(t)asympeq}
h_R(t) \sim \begin{dcases*}
1 & if $Rt$ tends to zero,	\\
\frac{2 J_1(Rt)}{Rt} & if $Rt \in (0,\infty)$,	\\
\frac{1}{\sqrt{\pi}} \left(\frac{2}{Rt}\right)^{3/2} \sin\left(Rt - \frac{\pi}{4}\right) & if $Rt$ tends to infinity,
\end{dcases*}
\end{equation}
where $J_{\nu}(z)$ denotes the Bessel function of the first kind. Moreover, $h_R(t) \ll 1$ if $R \ll 1$ and $t \in i(0,1/2)$.

We bound $\MM^{\Maass}(h) + \MM^{\Eis}(h)$ by breaking this up into intervals for which we can apply \hyperref[dihedralmomentsprop]{Proposition \ref*{dihedralmomentsprop}} and using the bounds \eqref{h(t)asympeq} and \eqref{hR(t)asympeq}: for the short initial and tail ranges, we use dyadic intervals, while for the short transition range, we divide into intervals of the form $[T - U,T + U]$ with $T = 2t_g \mp 3 \cdot 2^{-n - 1} t_g^{1 - \alpha}$ and $U = 2^{-n - 1} t_g^{1 - \alpha}$ for positive integers $n \leq (\frac{2}{3} - \alpha) \frac{\log t_g}{\log 2} - 1$, as well as the interval $[2t_g - t_g^{1/3},2t_g + t_g^{1/3}]$. The fact that $R \gg t_g^{-\delta}$ with $\delta < 1$ implies that $h_R(t)$ has polynomial decay in $t$ when $t$ is in the bulk range; the proof of \hyperref[Planckthm]{Theorem \ref*{Planckthm}} is thereby complete.

\hyperref[fourthmomentthm]{Theorem \ref*{fourthmomentthm}} is proven much in the same way, as the fourth moment is equal to the sum of $1/\vol(\Gamma_0(D) \backslash \Hb)$, \eqref{Varcusptermeq}, and \eqref{VarEistermeq} with $h_R(t)$ replaced by $1$. We find that the short initial, short transition, tail, and exceptional ranges all contribute at most $O(t_g^{-\delta'})$, while the bulk range contributes $2/\vol(\Gamma_0(D) \backslash \Hb) + O(t_g^{-\delta'})$.
\end{proof}

\begin{remark}
The method of proof also gives $\Var(g_{\psi};R) \sim 2/\vol(\Gamma_0(D) \backslash \Hb)$ if $R \ll t_g^{-\delta}$ with $\delta > 1$, while a modification of \hyperref[dihedralmomentsprop]{Proposition \ref*{dihedralmomentsprop} (2)} implies that there exists an absolute constant $\alpha > 0$ such that for $t_g^{-1 - \alpha} \ll R \ll t_g^{-1 + \alpha}$,
\begin{align*}
\Var\left(g_{\psi};R\right) & \sim \frac{4}{\pi R^2 t_g^2 \vol(\Gamma_0(D) \backslash \Hb)} \int_{0}^{1} \frac{J_1(2Rt_g t)^2}{t^2 \sqrt{1 - t^2}} \, dt	\\
& = \frac{2}{\vol(\Gamma_0(D) \backslash \Hb)} \prescript{}{2}{F}_3\left(\frac{1}{2},\frac{3}{2};1,2,3;-4 R^2 t_g^2\right),
\end{align*}
where $\prescript{}{p}{F}_q$ denotes the generalised hypergeometric function. This corrects an erroneous asymptotic formula in \cite[Remark 5.4]{Hum18}.
\end{remark}

\section{The Spectral Expansion of \texorpdfstring{$\Var(g;R)$}{Var(g;R)} and the Fourth Moment}
\label{spectralexpansionsect}

\subsection{An Orthonormal Basis of Maa\ss{} Cusp Forms for Squarefree Levels}

The proof of \hyperref[dihedralspectralprop]{Proposition \ref*{dihedralspectralprop}}, which we give in \hyperref[proofofdihedralspectralpropsect]{Section \ref*{proofofdihedralspectralpropsect}}, invokes the spectral decomposition of $L^2(\Gamma_0(q) \backslash \Hb)$, which involves a spectral sum indexed by an orthonormal basis $\BB_0(\Gamma_0(q))$ of the space of Maa\ss{} cusp forms of weight zero, level $q$, and principal nebentypus. This space has the Atkin--Lehner decomposition
\[\bigoplus_{q_1 q_2 = q} \bigoplus_{\ell \mid q_2} \iota_{\ell} \C \cdot \BB_0^{\ast}\left(\Gamma_0(q_1)\right),\]
where $(\iota_{\ell} f)(z) \defeq f(\ell z)$, but this decomposition is not orthogonal for $q > 1$. Nevertheless, an orthonormal basis can be formed using linear combinations of elements of this decomposition.

\begin{lemma}[{\cite[Proposition 2.6]{ILS00}}]
\label{ILSlemma}
An orthonormal basis of the space of Maa\ss{} cusp forms of weight zero, squarefree level $q$, and principal nebentypus is given by
\[\BB_0\left(\Gamma_0(q)\right) = \left\{f_{\ell} : f \in \BB_0^{\ast}\left(\Gamma_0(q_1)\right), \ q_1 q_2 = q, \ \ell \mid q_2\right\},\]
where each newform $f \in \BB_0^{\ast}\left(\Gamma_0(q_1)\right)$ is normalised such that $\langle f, f \rangle_q = 1$ and
\[f_{\ell} \defeq \left(L_{\ell}(1,\sym^2 f) \frac{\varphi(\ell)}{\ell}\right)^{1/2} \sum_{vw = \ell} \frac{\nu(v)}{v} \frac{\mu(w) \lambda_f(w)}{\sqrt{w}} \iota_v f.\]
\end{lemma}

\begin{proof}
In \cite[Proposition 2.6]{ILS00}, this is proved with
\[f_{\ell} \defeq \left(\frac{\ell}{\prod_{p \mid \ell} \left(1 - \frac{\lambda_f(p)^2 p}{(p + 1)^2}\right)}\right)^{1/2} \sum_{vw = \ell} \frac{\mu(w) \lambda_f(w)}{\sqrt{v} \nu(w)} \iota_v f.\]
Using the fact that $\lambda_f(p)^2 = \lambda_f(p^2) + 1$ and
\[L_p(s,\sym^2 f) = \frac{1}{1 - \lambda_f(p^2) p^{-s} + \lambda_f(p^2) p^{-2s} - p^{-3s}}\]
for $p \nmid q_1$, this simplifies to the desired identity.
\end{proof}

We record here the following identities, which follow readily from the multiplicativity of the summands involved.

\begin{lemma}
\label{EllMaassidentitylemma}
Suppose that $q_1,q_2$ are squarefree with $(q_1,q_2) = 1$. Then for a newform $f \in \BB_0^{\ast}(\Gamma_0(q_1))$ and $\ell \mid q_2$, we have that
\begin{align*}
\sum_{vw = \ell} \frac{\nu(v)}{v} \frac{\mu(w) \lambda_f(w)}{\sqrt{w}} & = \frac{1}{L_{\ell}\left(\frac{1}{2},f\right)},	\\
\sum_{\ell \mid q_2} \frac{L_{\ell}(1,\sym^2 f)}{L_{\ell}\left(\frac{1}{2},f\right)^2} \frac{\varphi(\ell)}{\ell} & = 2^{\omega(q_2)} \frac{\nu(q_2) \varphi(q_2)}{q_2^2} \frac{L_{q_2}(1,\sym^2 f)}{L_{q_2}\left(\frac{1}{2},f\right)}.
\end{align*}
\end{lemma}

\subsection{An Orthonormal Basis of Eisenstein Series for Squarefree Levels}

A similar orthonormal basis exists for Eisenstein series. Instead of the usual orthonormal basis
\[\left\{E_{\aa}(z,1/2 + it) : \aa \text{ is a cusp of $\Gamma_0(q) \backslash \Hb$}\right\},\]
we may form an orthonormal basis out of Eisenstein series newforms and oldforms: a basis of the space of Eisenstein series of weight zero, level $q$, and principal nebentypus is given by
\[\left\{(\iota_{\ell} E_1)\left(z,\frac{1}{2} + it\right) : \ell \mid q\right\}.\]
Here
\[E_1(z,s) \defeq \frac{1}{\sqrt{\nu(q)}} E(z,s), \qquad (\iota_{\ell} E_1)\left(z,\frac{1}{2} + it\right) \defeq E_1\left(\ell z,\frac{1}{2} + it\right),\]
where $E(z,s)$ is the usual Eisenstein series on $\Gamma \backslash \Hb$, defined for $\Re(s) > 1$ by
\[E(z,s) \defeq \sum_{\gamma \in \Gamma_{\infty} \backslash \Gamma} \Im(\gamma z)^s,\]
with $\Gamma \defeq \SL_2(\Z)$ and $\Gamma_{\infty} \defeq \{\gamma \in \Gamma : \gamma \infty = \infty\}$ the stabiliser of the cusp at infinity. For $t \in \R \setminus \{0\}$, this has the Fourier expansion
\[E\left(z,\frac{1}{2} + it\right) = y^{\frac{1}{2} + it} + \frac{\Lambda(1 - 2it)}{\Lambda(1 + 2it)} y^{\frac{1}{2} - it} + \sum_{\substack{n = -\infty \\ n \neq 0}}^{\infty} \rho(n,t) W_{0,it}(4\pi|n|y) e(nx)\]
with $W_{\alpha,\beta}$ the Whittaker function,
\[\rho(n,t) = \frac{\lambda(|n|,t)}{\sqrt{|n|}} \rho(1,t), \qquad \lambda(n,t) = \sum_{ab = n} a^{it} b^{-it}, \qquad \rho(1,t) = \frac{1}{\Lambda(1 + 2it)}.\]
The Eisenstein series $E(z,1/2 + it)$ is normalised such that its formal inner product with itself on $\Gamma \backslash \Hb$ is $1$ (in the sense of \cite[Proposition 7.1]{Iwa02}), and so the formal inner product of $E_1(z,1/2 + it)$ with itself on $\Gamma_0(q) \backslash \Hb$ is $1$.

This basis is not orthogonal for $q > 1$, but Young \cite{You19} has shown that there exists an orthonormal basis derived from this basis just as for Maa\ss{} cusp forms, as in \hyperref[ILSlemma]{Lemma \ref*{ILSlemma}}.

\begin{lemma}[{\cite[Section 8.4]{You19}}]
\label{YoungEislemma}
An orthonormal basis of the space of Eisenstein series of weight $0$, level $q$, and principal nebentypus is given by
\[\left\{E_{\ell}\left(z,\frac{1}{2} + it\right) : \ell \mid q\right\},\]
where $E_{\ell}(z,1/2 + it)$ is defined to be
\[\left(\zeta_{\ell}(1 + 2it) \zeta_{\ell}(1 - 2it)\right)^{1/2} \sum_{vw = \ell} \frac{\nu(v)}{v} \frac{\mu(w) \lambda(w,t)}{\sqrt{w}} (\iota_v E_1)\left(z, \frac{1}{2} + it\right).\]
\end{lemma}

As with \hyperref[EllMaassidentitylemma]{Lemma \ref*{EllMaassidentitylemma}}, we have the following identities.

\begin{lemma}
\label{EllEisidentitylemma}
For squarefree $q$ and $\ell \mid q$, we have that
\begin{align*}
\sum_{vw = \ell} \frac{\nu(v)}{v} \frac{\mu(w) \lambda(w,t)}{\sqrt{w}} & = \frac{1}{\zeta_{\ell}\left(\frac{1}{2} + it\right) \zeta_{\ell}\left(\frac{1}{2} - it\right)},	\\
\sum_{\ell \mid q} \frac{\zeta_{\ell}(1 + 2it) \zeta_{\ell}(1 - 2it)}{\zeta_{\ell}\left(\frac{1}{2} + it\right)^2 \zeta_{\ell}\left(\frac{1}{2} - it\right)^2} & = 2^{\omega(q)} \frac{\nu(q)}{q} \frac{\zeta_q(1 + 2it) \zeta_q(1 - 2it)}{\zeta_q\left(\frac{1}{2} + it\right) \zeta_q\left(\frac{1}{2} - it\right)}.
\end{align*}
\end{lemma}

\subsection{Inner Products with Oldforms and Eisenstein Series}
\label{AtkinLehnersect}

To deal with inner products involving oldforms and Eisenstein series, we use Atkin--Lehner operators. For squarefree $q$, write $q = vw$, and denote by
\[W_w \defeq \begin{pmatrix} a\sqrt{w} & b/\sqrt{w} \\ cv \sqrt{w} & d\sqrt{w} \end{pmatrix}\]
the Atkin--Lehner operator on $\Gamma_0(q)$ associated to $w$, where $a,b,c,d \in \Z$ and $\det W_w = adw - bcv = 1$. We denote by $\BB_{\hol}^{\ast}(q,\chi)$ the set of holomorphic newforms $f$ of level $q$, nebentypus $\chi$, and arbitrary even weight $k_f \in 2\N$; again, we write $\BB_{\hol}^{\ast}(\Gamma_0(q))$ when $\chi$ is the principal character.

\begin{lemma}[{\cite[Theorem 2.1]{AL78}}; see also {\cite[Proposition A.1]{KMV02}}]
\label{ALlemma}
Let $q = vw$ be squarefree and let $\chi$ be a Dirichlet character of conductor $q_{\chi}$ dividing $q$, so that we may write $\chi = \chi_v \chi_w$. Then for $g \in \BB_0^{\ast}(q,\chi)$, $g(W_w z)$ is equal to $\eta_g(w) (g \otimes \overline{\chi_w})(z)$, where $g \otimes \overline{\chi_w} \in \BB_0^{\ast}(q, \chi_v \overline{\chi_w})$ with
\begin{align*}
\lambda_{g \otimes \overline{\chi_w}}(n) & = \begin{dcases*}
\overline{\chi_w}(n) \lambda_g(n) & if $(n,w) = 1$,	\\
\chi_v(n) \overline{\lambda_g}(n) & otherwise,
\end{dcases*}	\\
\eta_g(w) & = \overline{\chi_w}(b) \overline{\chi_v}(a) \frac{\tau(\chi_w)}{\lambda_g(w) \sqrt{w}}.
\end{align*}
In particular, $|\eta_g(w)| = 1$. Moreover, the same result holds for $g \in \BB_{\hol}^{\ast}(q,\chi)$, so that $g \otimes \overline{\chi_w} \in \BB_{\hol}^{\ast}(q, \chi_v \overline{\chi_w})$.
\end{lemma}

We call $\eta_g(w)$ the Atkin--Lehner pseudo-eigenvalue; note that it is independent of $a,b,c,d \in \Z$ when either $\chi$ is the principal character or $a \equiv 1 \pmod{v}$ and $b \equiv 1 \pmod{w}$, or equivalently $d \equiv \overline{w} \pmod{v}$ and $c \equiv \overline{v} \pmod{w}$.

\begin{lemma}
\label{AtkinLehnerlemma}
Let $q = q_1 q_2$ be squarefree, let $\chi$ be a Dirichlet character modulo $q$, and let $g \in \BB_0^{\ast}\left(q,\chi\right)$ and $f \in \BB_0^{\ast}\left(\Gamma_0(q_1)\right)$. Then for $vw = q_2$, so that $\chi = \chi_v \chi_w \chi_{q_1}$,
\[\left\langle |g|^2, \iota_v f \right\rangle_q = \left\langle |g \otimes \overline{\chi_v}|^2, f \right\rangle_q.\]
\end{lemma}

\begin{proof}
Since the Atkin--Lehner operators normalise $\Gamma_0(q)$,
\[\left\langle |g|^2, \iota_v f \right\rangle_q = \int_{\Gamma_0(q) \backslash \Hb} |g(W_w z)|^2 f\left(\begin{pmatrix} \sqrt{v} & 0 \\ 0 & 1/\sqrt{v} \end{pmatrix} W_w z\right) \, d\mu(z).\]
By \hyperref[ALlemma]{Lemma \ref*{ALlemma}}, $|g(W_w z)|^2 = |(g \otimes \overline{\chi_w})(z)|^2$, while
\[\begin{pmatrix} \sqrt{v} & 0 \\ 0 & 1/\sqrt{v} \end{pmatrix} W_w = \begin{pmatrix} a & bv \\ cq_1 & dw\end{pmatrix} \begin{pmatrix} \sqrt{q_2} & 0 \\ 0 & 1/\sqrt{q_2} \end{pmatrix},\]
and so as $f$ is invariant under the action of $\Gamma_0(q_1)$,
\[f\left(\begin{pmatrix} \sqrt{v} & 0 \\ 0 & 1/\sqrt{v} \end{pmatrix} W_w z\right) = f(q_2 z).\]
So whenever $v$ divides $q_2$, $\left\langle |g|^2, \iota_v f \right\rangle_q = \left\langle |g \otimes \overline{\chi_w}|^2, \iota_{q_2} f \right\rangle_q$. Taking $v = 1$, $w = q_2$, and replacing $g$ with $g \otimes \overline{\chi_v}$, which has nebentypus $\overline{\chi_v} \chi_w \chi_{q_1}$, then shows that $\left\langle |g \otimes \overline{\chi_v}|^2, f \right\rangle_q = \left\langle |g \otimes \overline{\chi_w}|^2, \iota_{q_2} f \right\rangle_q$.
\end{proof}

We now prove an analogous result for Eisenstein series. In this case, we may use Eisenstein series indexed by cusps (though later we will find it advantageous to work with Eisenstein newforms and oldforms). As $q$ is squarefree, a cusp $\aa$ of $\Gamma_0(q) \backslash \Hb$ has a representative of the form $1/v$ for some divisor $v$ of $q$, and every cusp has a unique representative of this form; when $\aa \sim \infty$, for example, we have that $v = q$. We define the Eisenstein series
\[E_{\aa}(z,s) \defeq \sum_{\gamma \in \Gamma_{\aa} \backslash \Gamma_0(q)} \Im\left(\sigma_{\aa}^{-1} \gamma z\right)^s,\]
which converges absolutely for $\Re(s) > 1$ and $z \in \Hb$, where
\[\Gamma_{\aa} \defeq \left\{\gamma \in \Gamma_0(q) : \gamma \aa = \aa\right\}\]
is the stabiliser of the cusp $\aa$, and the scaling matrix $\sigma_{\aa} \in \SL_2(\R)$ is such that
\[\sigma_{\aa} \infty = \aa, \qquad \sigma_{\aa}^{-1} \Gamma_{\aa} \sigma_{\aa} = \Gamma_{\infty}.\]
The Eisenstein series $E_{\aa}(z,s)$ is independent of the choice of scaling matrix.

Writing $q = vw$, we may choose $\sigma_{\aa} = W_w$ with
\[W_w = \begin{pmatrix} \sqrt{w} & b/\sqrt{w} \\ v \sqrt{w} & d \sqrt{w} \end{pmatrix}\]
the Atkin--Lehner operator on $\Gamma_0(q)$ associated to $w$, where $dw - bv = 1$.

\begin{lemma}
\label{Eisenlemma}
Let $g \in \BB_0^{\ast}(q,\chi)$ with $q$ squarefree, and let $\aa \sim 1/v$ be a cusp of $\Gamma_0(q) \backslash \Hb$. Then
\[\left\langle |g|^2, E_{\aa}(\cdot,s)\right\rangle_q = \left\langle |g \otimes \overline{\chi_v}|^2, E_{\infty}(\cdot,s)\right\rangle_q.\]
\end{lemma}

\begin{proof}
By unfolding, using \hyperref[ALlemma]{Lemma \ref*{ALlemma}}, and folding, we find that
\begin{align*}
\left\langle |g|^2, E_{\aa}(\cdot,s)\right\rangle_q & = \int_{\Gamma_{\aa} \backslash \Hb} |g(z)|^2 \Im\left(\sigma_{\aa}^{-1} \gamma z\right)^{\overline{s}} \, d\mu(z)	\\
& = \int_{0}^{\infty} \int_{0}^{1} \left|g\left(\sigma_{\aa} z\right)\right|^2 y^{\overline{s}} \, \frac{dx \, dy}{y^2}	\\
& = \int_{0}^{\infty} \int_{0}^{1} |(g \otimes \overline{\chi_v})(z)|^2 y^{\overline{s}} \, \frac{dx \, dy}{y^2}	\\
& = \left\langle |g \otimes \overline{\chi_v}|^2, E_{\infty}(\cdot,s)\right\rangle_q.
\qedhere
\end{align*}
\end{proof}

Finally, we claim that twisting $g$ leaves these inner products unchanged. Alas, we do not know a simple proof of this fact; as such, the proof is a consequence of calculations in \hyperref[WatsonIchinosect]{Sections \ref*{WatsonIchinosect}} and \ref{localconstantsect}.

\begin{lemma}
\label{abstwistlemma}
For $q = q_1 q_2$ squarefree and $g \in \BB_0^{\ast}(q,\chi)$ with $\chi$ primitive, we have that
\[\left\langle |g \otimes \overline{\chi_{q_2}}|^2, E_{\infty}(\cdot,s)\right\rangle_q = \left\langle |g|^2, E_{\infty}(\cdot,s)\right\rangle_q.\]
Furthermore, for $f \in \BB_0^{\ast}(\Gamma_0(q_1))$ and $w \mid q_2$,
\[\left\langle |g \otimes \overline{\chi_w}|^2, f \right\rangle_q = \left\langle |g|^2, f \right\rangle_q.\]
\end{lemma}

\begin{proof}
The former is a consequence of \hyperref[g^2Eiscor]{Corollary \ref*{g^2Eiscor}}, while the latter follows upon combining \hyperref[AtkinLehnerlemma]{Lemma \ref*{AtkinLehnerlemma}} with \hyperref[sqfreetripleproductcor]{Corollary \ref*{sqfreetripleproductcor}}.
\end{proof}

\subsection{Proof of \texorpdfstring{\hyperref[dihedralspectralprop]{Proposition \ref*{dihedralspectralprop}}}{Proposition \ref{dihedralspectralprop}}}
\label{proofofdihedralspectralpropsect}

\begin{proof}[Proof of {\hyperref[dihedralspectralprop]{Proposition \ref*{dihedralspectralprop}}}]
An application of Parseval's identity, using the spectral decomposition of $L^2(\Gamma_0(q) \backslash \Hb)$ \cite[Theorem 15.5]{IK04}, together with the fact that
\[\frac{1}{\vol(B_R)} \int_{B_R(w)} f(z) \, d\mu(z) = h_R(t_f) f(w)\]
for any Laplacian eigenfunction $f$ \cite[Lemma 4.3]{Hum18}, yields
\[\Var\left(g;R\right) = \sum_{f \in \BB_0\left(\Gamma_0(q)\right)} \left|h_R(t_f)\right|^2 \left| \left\langle |g|^2,f \right\rangle_q \right|^2 + \sum_{\aa} \frac{1}{4\pi} \int_{-\infty}^{\infty} \left|h_R(t)\right|^2 \left|\left\langle |g|^2, E_{\aa}\left(\cdot, \frac{1}{2} + it\right) \right\rangle_q \right|^2 \, dt;\]
see \cite[Proof of Proposition 5.2]{Hum18}. By \hyperref[ILSlemma]{Lemmata \ref*{ILSlemma}}, \ref{EllMaassidentitylemma}, \ref{AtkinLehnerlemma}, and \ref{abstwistlemma},
\[\sum_{\substack{f \in \BB_0\left(\Gamma_0(q)\right) \\ t_f = t}} \left|\left\langle |g|^2, f \right\rangle_q \right|^2 = \sum_{q_1 q_2 = q} 2^{\omega(q_2)} \frac{\nu(q_2) \varphi(q_2)}{q_2^2} \sum_{\substack{f \in \BB_0^{\ast}\left(\Gamma_0\left(q_1\right)\right) \\ t_f = t}} \frac{L_{q_2}\left(1,\sym^2 f\right)}{L_{q_2}\left(\frac{1}{2}, f\right)} \left| \left\langle |g|^2, f \right\rangle_q \right|^2\]
for any $t \in [0,\infty) \cup i(0,1/2)$. Similarly, \hyperref[Eisenlemma]{Lemmata \ref*{Eisenlemma}} and \ref{abstwistlemma} imply that
\[\sum_{\aa} \left|\left\langle |g|^2, E_{\aa}\left(\cdot, \frac{1}{2} + it\right) \right\rangle_q \right|^2 = 2^{\omega(q)} \left|\left\langle |g|^2, E_{\infty}\left(\cdot, \frac{1}{2} + it\right) \right\rangle_q \right|^2\]
for any $t \in \R$. This gives the desired spectral expansion for $\Var(g;R)$, while the spectral expansion for the fourth moment of $g$ follows similarly, noting that the constant term $1/\sqrt{\vol(\Gamma_0(q) \backslash \Hb)}$ in the spectral expansion gives rise to the term $1/\vol(\Gamma_0(q) \backslash \Hb)$ in \eqref{dihedralspectral2eq}.
\end{proof}

\section{The Watson--Ichino Formula}
\label{WatsonIchinosect}

\subsection{The Watson--Ichino Formula for Eisenstein Series}

We require explicit expressions in terms of $L$-functions for $|\langle |g|^2,f \rangle_q|^2$ and $|\langle |g|^2,E_{\infty}(\cdot,1/2 + it) \rangle_q|^2$. This is the contents of the Watson--Ichino formula. In the latter case, this result is simply the Rankin--Selberg method, which far predates the work of Watson and Ichino; it can be proven by purely classical means via unfolding the Eisenstein series, as we shall now detail.

Recall that a Maa\ss{} newform $g \in \BB_0^{\ast}(q,\chi)$ has the Fourier expansion about the cusp at infinity of the form
\[g(z) = \sum_{\substack{n = -\infty \\ n \neq 0}}^{\infty} \rho_g(n) W_{0,it_g}\left(4\pi |n| y\right) e(nx),\]
where the Fourier coefficients $\rho_g(n)$ satisfy $\rho_g(n) = \epsilon_g \rho_g(-n)$, with the parity $\epsilon_g$ of $g$ equal to $1$ if $g$ is even and $-1$ if $g$ is odd.
The Hecke eigenvalues $\lambda_g(n)$ of $g$ satisfy
\begin{align}
\label{cuspmult}
\lambda_g(m) \lambda_g(n) & = \sum_{d \mid (m,n)} \chi(d) \lambda_g\left(\frac{mn}{d^2}\right) \quad \text{for all $m,n \geq 1$,}	\\
\label{cuspconj}
\overline{\lambda_g}(n) & = \chi(n) \lambda_g(n) \quad \text{for all $n \geq 1$ with $(n,q) = 1$,}	\\
\label{cusprholambda}
\rho_g(1) \lambda_g(n) & = \sqrt{n} \rho_g(n) \quad \text{for all $n \geq 1$.}
\end{align}

\begin{lemma}
Let $g \in \BB_0^{\ast}(q_1,\chi)$ with $q_1 q_2 = q$ and $q_1 \equiv 0 \pmod{q_{\chi}}$, where $q_{\chi}$ is the conductor of $\chi$. We have that
\begin{equation}
\label{|g|^2Einfty}
\left\langle |g|^2, E_{\infty}(\cdot,s)\right\rangle_q = \frac{|\rho_g(1)|^2}{\pi^{\overline{s}}} \frac{\Gamma\left(\frac{\overline{s}}{2} + it_g\right) \Gamma\left(\frac{\overline{s}}{2}\right)^2 \Gamma\left(\frac{\overline{s}}{2} - it_g\right)}{\Gamma(\overline{s})} \sum_{\substack{n = 1}}^{\infty} \frac{\left|\lambda_g(n)\right|^2}{n^{\overline{s}}}.
\end{equation}
\end{lemma}

\begin{proof}
Unfolding the integral and using Parseval's identity and \eqref{cusprholambda} yields
\[\left\langle |g|^2, E_{\infty}(\cdot,s)\right\rangle_q = \frac{2 |\rho_g(1)|^2}{(4\pi)^{\overline{s} - 1}} \sum_{\substack{n = 1}}^{\infty} \frac{\left|\lambda_g(n)\right|^2}{n^{\overline{s}}} \int_{0}^{\infty} y^{\overline{s} - 1} W_{0, it_g}(y)^2 \, \frac{dy}{y}\]
after the change of variables $y \mapsto y/(4\pi|n|y)$. The result then follows via the Mellin--Barnes formula \cite[6.576.4]{GR07}.
\end{proof}

\begin{lemma}
\label{rho(1)^2lemma}
Let $q$ be squarefree, and let $g \in \BB_0^{\ast}(q_1,\chi)$ with $q_1 q_2 = q$ and $q_1 \equiv 0 \pmod{q_{\chi}}$. We have that
\begin{equation}
\label{|lambda_g|^2}
\sum_{n = 1}^{\infty} \frac{|\lambda_g(n)|^2}{n^s} = \frac{\zeta(s) L(s, \ad g)}{\zeta(2s)} \prod_{p \mid q_1} \frac{1}{1 + p^{-s}}
\end{equation}
for $\Re(s) > 1$ and that
\begin{equation}
\label{rho_g(1)^2}
|\rho_g(1)|^2 = \frac{\left\langle g, g\right\rangle_q}{2 \nu(q_2) \Lambda(1, \ad g)} = \frac{q_2 \cosh \pi t_g \left\langle g, g\right\rangle_q}{2 q \nu(q_2) L(1, \ad g)}.
\end{equation}
\end{lemma}

\begin{proof}
We recall that
\[\Lambda(s,\ad g) = q_1^s \pi^{-\frac{3s}{2}} \Gamma\left(\frac{s}{2} + it_g\right) \Gamma\left(\frac{s}{2}\right) \Gamma\left(\frac{s}{2} - it_g\right) \prod_p L_p(s, \ad g)\]
with
\[L_p(s,\ad g)^{-1} = \begin{dcases*}
1 - p^{-s} & if $p \mid q_{\chi}$,	\\
1 - p^{-1 - s} & if $p \mid \frac{q_1}{q_{\chi}}$,	\\
1 - \overline{\chi}(p) \lambda_g(p^2) p^{-s} + \overline{\chi}(p) \lambda_g(p^2) p^{-2s} - p^{-3s} & if $p \nmid q_1$.
\end{dcases*}\]
Using \eqref{cuspmult} and \eqref{cuspconj} together with the fact that
\[\left|\lambda_f(p)\right|^2 = \begin{dcases*}
1 & if $p \mid q_{\chi}$,	\\
\frac{1}{p} & if $p \mid \frac{q_1}{q_{\chi}}$,
\end{dcases*}\]
we obtain \eqref{|lambda_g|^2}. Next, we take the residue of \eqref{|g|^2Einfty} at $\overline{s} = 1$, noting that $E_{\infty}(z,s)$ has residue \[\frac{1}{\vol(\Gamma_0(q) \backslash \Hb)} = \frac{3}{\pi \nu(q)}\]
at $s = 1$ independently of $z \in \Gamma_0(q) \backslash \Hb$. This yields the desired identity \eqref{rho_g(1)^2}.
\end{proof}

\begin{corollary}
\label{g^2Eiscor}
Let $q$ be squarefree, and let $g \in \BB_0^{\ast}(q_1,\chi)$ with $q_1 q_2 = q$ and $q_1 \equiv 0 \pmod{q_{\chi}}$, where $g$ is normalised such that $\langle g, g \rangle_q = 1$. We have that
\[\left\langle |g|^2, E_{\infty}(\cdot,s)\right\rangle_q = \frac{1}{2 q_1^{\overline{s}} \nu(q_2)} \frac{\Lambda^{q_1}(\overline{s}) \Lambda(\overline{s}, \ad g)}{\Lambda(1, \ad g) \Lambda^{q_1}(2\overline{s})}\]
for $\Re(s) \geq 1/2$ with $s \neq 1$, so that
\begin{equation}
\label{g^2Eisinnerproduct}
\left|\left\langle |g|^2, E_{\infty}\left(\cdot, \frac{1}{2} + it\right)\right\rangle_q\right|^2 = \frac{1}{4q_1 \nu(q_2)^2} \left|\frac{\Lambda^{q_1}\left(\frac{1}{2} + it\right) \Lambda\left(\frac{1}{2} + it, \ad g\right)}{\Lambda(1, \ad g) \Lambda^{q_1}(1 + 2it)}\right|^2.\
\end{equation}
\end{corollary}

Note that \hyperref[g^2Eiscor]{Corollary \ref*{g^2Eiscor}} remains valid when $g$ is replaced by $g \otimes \chi_v$ for $v \mid q_{\chi}$, since the level is unchanged and $\ad (g \otimes \chi_v) = \ad g$.

\begin{remark}
One can also prove \eqref{g^2Eisinnerproduct} ad\`{e}lically; see, for example, \cite[(4.21)]{MV10}.
\end{remark}

\subsection{The Ad\`{e}lic Watson--Ichino Formula for Maa\ss{} Newforms}

Now we consider the inner product $|\langle |g|^2, f\rangle_q|^2$. The Watson--Ichino formula is an ad\`{e}lic statement: the integral over $\Gamma_0(q) \backslash \Hb$ is replaced by an integral over $\Zgp(\A_{\Q}) \GL_2(\Q) \backslash \GL_2(\A_{\Q})$, and $g$ and $f$ are replaced by functions on $\GL_2(\Q) \backslash \GL_2(\A_{\Q})$ that are square integrable modulo the centre $\Zgp(\A_{\Q})$ and are elements of cuspidal automorphic representations of $\GL_2(\A_{\Q})$. In \hyperref[ClassicalWatsonIchinosect]{Section \ref*{ClassicalWatsonIchinosect}}, we translate this ad\`{e}lic statement into a statement in the classical language of automorphic forms.

Let $F$ be a number field, and let $\varphi_1 = \bigotimes_v \varphi_{1,v}$, $\varphi_2 = \bigotimes_v \varphi_{2,v}$, $\varphi_3 = \bigotimes_v \varphi_{3,v}$ be pure tensors in unitary cuspidal automorphic representations $\pi_1 = \bigotimes_v \pi_{1,v}$, $\pi_2 = \bigotimes_v \pi_{2,v}$, $\pi_3 = \bigotimes_v \pi_{3,v}$ of $\GL_2(\A_F)$ with central characters $\omega_{\pi_1}$, $\omega_{\pi_2}$, $\omega_{\pi_3}$ satisfying $\omega_{\pi_1} \omega_{\pi_2} \omega_{\pi_3} = 1$, and let $\widetilde{\varphi}_1 = \bigotimes_v \widetilde{\varphi}_{1,v}$, $\widetilde{\varphi}_2 = \bigotimes_v \widetilde{\varphi}_{2,v}$, $\widetilde{\varphi}_3 = \bigotimes_v \widetilde{\varphi}_{3,v}$ be pure tensors in the contragredient representations $\widetilde{\pi}_1 = \bigotimes_v \widetilde{\pi}_{1,v}$, $\widetilde{\pi}_2 = \bigotimes_v \widetilde{\pi}_{2,v}$, $\widetilde{\pi}_3 = \bigotimes_v \pi_{3,v}$. Let
\begin{align*}
\varphi & \defeq \varphi_1 \otimes \varphi_2 \otimes \varphi_3,	\\
\widetilde{\varphi} & \defeq \widetilde{\varphi}_1 \otimes \widetilde{\varphi}_2 \otimes \widetilde{\varphi}_3,	\\
I(\varphi \otimes \widetilde{\varphi}) & \defeq \int\limits_{\Zgp(\A_F) \GL_2(F) \backslash \GL_2(\A_F)} \varphi_1(g) \varphi_2(g) \varphi_3(g) \, dg \int\limits_{\Zgp(\A_F) \GL_2(F) \backslash \GL_2(\A_F)} \widetilde{\varphi}_1(g) \widetilde{\varphi}_2(g) \widetilde{\varphi}_3(g) \, dg,	\\
\langle \varphi, \widetilde{\varphi} \rangle & \defeq \prod_{\ell = 1}^{3} \left(\int\limits_{\Zgp(\A_F) \GL_2(F) \backslash \GL_2(\A_F)} \left|\varphi_{\ell}(g)\right|^2 \, dg \int\limits_{\Zgp(\A_F) \GL_2(F) \backslash \GL_2(\A_F)} \left|\widetilde{\varphi}_{\ell}(g)\right|^2 \, dg\right)^{1/2},
\end{align*}
with $dg$ the Tamagawa measure on $\Zgp(\A_F) \GL_2(F) \backslash \GL_2(\A_F)$. For each place $v$ of $F$ with corresponding local field $F_v$, we also let
\begin{align}
\notag
\varphi_v & \defeq \varphi_{1,v} \otimes \varphi_{2,v} \otimes \varphi_{3,v},	\\
\label{I_v}
I_v(\varphi_v \otimes \widetilde{\varphi}_v) & \defeq \int\limits_{\Zgp(F_v) \backslash \GL_2(F_v)}\prod_{\ell = 1}^{3} \left\langle \pi_{\ell,v}(g_v) \cdot \varphi_{\ell,v}, \widetilde{\varphi}_{\ell,v} \right\rangle \, dg_v,	\\
\label{I_v'}
I_v'(\varphi_v \otimes \widetilde{\varphi}_v) & \defeq \frac{L_v(1, \ad \pi_{1,v}) L_v(1, \ad \pi_{2,v}) L_v(1, \ad \pi_{3,v})}{\zeta_v(2)^2 L_v\left(\frac{1}{2}, \pi_{1,v} \otimes \pi_{2,v} \otimes \pi_{3,v}\right)} \frac{I_v(\varphi_v \otimes \widetilde{\varphi}_v)}{\langle \varphi_v, \widetilde{\varphi}_v\rangle_v},	\\
\notag
\langle \varphi_v, \widetilde{\varphi}_v\rangle_v & \defeq \prod_{j = 1}^{3} \left(\int_{K_v} \left|\varphi_{\ell,v}(k_v)\right|^2 \, dk_v \int_{K_v} \left|\widetilde{\varphi}_{\ell,v}(k_v)\right|^2 \, dk_v\right)^{1/2}.
\end{align}
The Haar measure $dg_v$ on $\Zgp(F_v) \backslash \GL_2(F_v)$ is normalised as follows:
\begin{itemize}
\item For $v$ nonarchimedean and $x_v \in \Zgp(F_v) \backslash \GL_2(F_v)$, we may use the Iwasawa decomposition to write $g_v = \left(\begin{smallmatrix} a_v & x_v \\ 0 & 1 \end{smallmatrix}\right) k_v$ with $x_v \in F_v$, $a_v \in F_v^{\times}$, and $k_v \in \GL_2(\OO_v)$. Then $dg_v = dx_v \, |a_v|_v^{-1} \, d^{\times} a_v \, dk_v$. Here the additive Haar measure $dx_v$ on $F_v$ is normalised to give $\OO_v$ volume $1$, the multiplicative Haar measure $d^{\times} a_v = \zeta_v(1) |a_v|_v^{-1} \, da_v$ on $F_v^{\times}$ is normalised to give $\OO_v^{\times} = \GL_1(\OO_v)$ volume $1$, and $dk_v$ is the Haar probability measure on the compact group $\GL_2(\OO_v)$.
\item For $F_v \cong \R$ and $x_v \in \Zgp(F_v) \backslash \GL_2(F_v)$, we may use the Iwasawa decomposition to write $g_v = \left(\begin{smallmatrix} a_v & x_v \\ 0 & 1 \end{smallmatrix}\right) k_v$ with $x_v \in \R$, $a_v \in \R^{\times}$, and $k_v = \left(\begin{smallmatrix} \cos \theta & \sin \theta \\ -\sin \theta & \cos \theta \end{smallmatrix}\right) \in \SO(2)$ with $\theta \in [0,2\pi)$. Then $dg_v = dx_v \, |a_v|_v^{-1} \, d^{\times} a_v \, dk_v$, where the additive Haar measure $dx_v$ on $\R$ is the usual Lebesgue measure normalised to give $[0,1]$ volume $1$, the multiplicative Haar measure $d^{\times} a_v$ on $\R^{\times}$ is $|a_v|_v^{-1} \, da_v$, and $dk_v = (2\pi)^{-1} \, d\theta$ is the Haar probability measure on the compact group $\SO(2)$.
\item A similar definition can also be given for $F_v \cong \C$, though we do not need this, since we will eventually take $F = \Q$.
\end{itemize}
The Tamagawa measure $dg$ on $\Zgp(\A_F) \GL_2(F) \backslash \GL_2(\A_F)$ is such that
\[dg = C_F \prod_v dg_v,\]
where
\[C_F = |d_F|^{-3/2} \prod_v \zeta_v(2)^{-1} = |d_F|^{-1/2} \Lambda_F(2)^{-1}.\]
Here $d_F$ denotes the discriminant of $F$, and we recall that the conductor of the Dedekind zeta function is $|d_F|$, so that the completed Dedekind zeta function is $\Lambda_F(s) = |d_F|^{s/2} \prod_v \zeta_v(s)$.

\begin{theorem}[{\cite[Theorem 1.1]{Ich08}}]
The period integral $I(\varphi \otimes \widetilde{\varphi}) / \langle \varphi, \widetilde{\varphi} \rangle$ is equal to
\[\frac{C_F}{8} \left(\frac{q(\pi_1 \otimes \pi_2 \otimes \pi_3)^{1/2}}{q(\ad \pi_1) q(\ad \pi_2) q(\ad \pi_3)}\right)^{-1/2} \frac{\Lambda\left(\frac{1}{2}, \pi_1 \otimes \pi_2 \otimes \pi_3\right)}{\Lambda(1, \ad \pi_1) \Lambda(1, \ad \pi_2) \Lambda(1, \ad \pi_3)} \prod_v I_v'(\varphi_v \otimes \widetilde{\varphi}_v),\]
with $I_v'(\varphi_v \otimes \widetilde{\varphi}_v)$ equal to $1$ whenever $\varphi_{1,v}$, $\varphi_{2,v}$, $\varphi_{3,v}$ and $\widetilde{\varphi}_{1,v}$, $\widetilde{\varphi}_{2,v}$, $\widetilde{\varphi}_{3,v}$ are spherical vectors at a nonarchimedean place $v$.
\end{theorem}

The quantity $I_v'(\varphi_v \otimes \widetilde{\varphi}_v)$ is often called the local constant. When $\varphi_1$, $\varphi_2$, $\varphi_3$ are pure tensors consisting of local newforms in the sense of Casselman (or in some cases translates of local newforms; see \cite{Hu17} and \cite[Section 2.1]{Col19}), then these local constants depend only (but sensitively!) on the representations $\pi_{1,v}$, $\pi_{2,v}$, $\pi_{3,v}$. The local constants have been explicitly determined for many different combinations of representations $\pi_{1,v}$, $\pi_{2,v}$, $\pi_{3,v}$ of $\GL_2(F_v)$ (cf.~\cite[Sections 2.2 and 2.3]{Col19}). We require several particular combinations of representations for our applications.

For $F_v \cong \R$, let $k(\pi_v) \in \Z$ denote the weight of $\pi_v$ and let $\epsilon_v \in \{1,i,-1,-i\}$ denote the local root number, so that $\epsilon_v = (-1)^{m_v}$ for $\pi_v$ a weight zero principal series representation $\sgn^{m_v} |\cdot|_v^{s_{1,v}} \boxplus \sgn^{m_v} |\cdot|_v^{s_{2,v}}$ with $m_v \in \{0,1\}$.

\begin{proposition}[{\cite[Theorem 3]{Wat08}}]
For $F_v \cong \R$,
\[I_v'(\varphi_v \otimes \widetilde{\varphi}_v) = \frac{1 + \epsilon_{1,v} \epsilon_{2,v} \epsilon_{3,v}}{2}\]
 if $k(\pi_{1,v}) = k(\pi_{2,v}) = k(\pi_{3,v}) = 0$.
\end{proposition}

Now let $F_v$ be a nonarchimedean local field with uniformiser $\varpi_v$ and cardinality $q_v$ of the residue field. In \hyperref[localconstantsect]{Section \ref*{localconstantsect}}, we prove the following. 

\begin{proposition}
\label{dihedral2Stprop}
Let $\pi_{1,v} = \omega_{1,v} \boxplus \omega_{1,v}'$ and $\pi_{2,v} = \widetilde{\pi}_{1,v} = \omega_{1,v}'^{-1} \boxplus \omega_{1,v}^{-1}$ be principal series representations of $\GL_2(F_v)$ for which the characters $\omega_{1,v}$, $\omega_{1,v'}$ of $F_v^{\times}$ have conductor exponents $c(\omega_{1,v}) = 1$ and $c(\omega_{1,v}') = 0$, and let $\pi_{3,v} = \omega_{3,v} \St_v$ be a special representation with $c(\omega_{3,v}) = 0$ and $\omega_{3,v}^2 = 1$. Suppose that $\pi_{1,v}$, $\pi_{2,v}$, $\pi_{3,v}$ are irreducible and unitarisable, so that $\omega_{1,v}$, $\omega_{1,v}'$, $\omega_{3,v}$ are unitary. Then if $\varphi_{1,v}$, $\varphi_{2,v}$, $\varphi_{3,v}$, $\widetilde{\varphi}_{1,v}$, $\widetilde{\varphi}_{2,v}$, $\widetilde{\varphi}_{3,v}$ are all local newforms,
\[I_v'(\varphi_v \otimes \widetilde{\varphi}_v) = \frac{1}{q_v} \left(1 + \frac{1}{q_v}\right).\]
\end{proposition}

\begin{proposition}
\label{Collinsprop}
Let $\pi_{1,v} = \omega_{1,v} \boxplus \omega_{1,v}'$, $\pi_{2,v} = \widetilde{\pi}_{1,v} = \omega_{1,v}'^{-1} \boxplus \omega_{1,v}^{-1}$, and $\pi_{3,v} = \omega_{3,v} \boxplus \omega_{3,v}^{-1}$ be principal series representations of $\GL_2(F_v)$ with $c(\omega_{1,v}) = 1$ and $c(\omega_{1,v}') = 0$, with $c(\omega_{3,v}) = 0$. Suppose that $\pi_{1,v}$, $\pi_{2,v}$, $\pi_{3,v}$ are irreducible and unitarisable, so that $\omega_{1,v}$, $\omega_{1,v}'$ are unitary while $q^{-1/2} < |\omega_{3,v}(\varpi_v)| < q^{1/2}$. Then if $\varphi_{1,v}$, $\varphi_{2,v}$, $\varphi_{3,v}$, $\widetilde{\varphi}_{1,v}$, $\widetilde{\varphi}_{2,v}$, $\widetilde{\varphi}_{3,v}$ are all local newforms,
\[I_v'(\varphi_v \otimes \widetilde{\varphi}_v) = \frac{1}{q_v}.\]
This also holds if either or both $\varphi_{3,v}$ and $\widetilde{\varphi}_{3,v}$ are translates of local newforms by $\pi_{3,v} \left(\begin{smallmatrix}\varpi_v^{-1} & 0 \\ 0 & 1 \end{smallmatrix}\right)$ and $\widetilde{\pi}_{3,v} \left(\begin{smallmatrix}\varpi_v^{-1} & 0 \\ 0 & 1 \end{smallmatrix}\right)$ respectively.
\end{proposition}

\begin{remark}
The latter local constant has also been determined by Collins \cite[Proposition 2.2.3]{Col19}. Moreover, Collins \cite[Section 5.2]{Col18} has numerically verified both of these local constants, as well as the local constant in \hyperref[localconstrem]{Remark \ref*{localconstrem}}.
\end{remark}

\subsection{The Classical Watson--Ichino Formula for Maa\ss{} Newforms}
\label{ClassicalWatsonIchinosect}

Now we restate the Watson--Ichino formula in the classical setting. For $\ell \in \{1,2,3\}$, let $f_{\ell} \in \BB_0(q,\chi_{\ell})$ be a Hecke--Maa\ss{} eigenform of level $q$, nebentypus $\chi_{\ell}$, and parity $\epsilon_{f_{\ell}}$, and similarly let $\widetilde{f}_{\ell} \in \BB_0(q,\overline{\chi_{\ell}})$ be a Hecke--Maa\ss{} eigenform such that $f_{\ell}$ and $\overline{\widetilde{f}}_{\ell}$ are both associated to the same newform. We assume additionally that $\chi_1 \chi_2 \chi_3 = \chi_{0(q)}$, the principal character modulo $q$. Letting $\varphi_1$, $\varphi_2$, $\varphi_3$ and $\widetilde{\varphi}_1$, $\widetilde{\varphi}_2$, $\widetilde{\varphi}_3$ denote the ad\`{e}lic lifts of the Hecke--Maa\ss{} eigenforms $f_1$, $f_2$, $f_3$ and $\widetilde{f}_1$, $\widetilde{f}_2$, $\widetilde{f}_3$, we have that
\begin{multline*}
\int_{\Gamma_0(q) \backslash \Hb} f_1(z) f_2(z) f_3(z) \, d\mu(z) \int_{\Gamma_0(q) \backslash \Hb} \widetilde{f}_1(z) \widetilde{f}_2(z) \widetilde{f}_3(z) \, d\mu(z)	\\
= \frac{1 + \epsilon_{f_1} \epsilon_{f_2} \epsilon_{f_3}}{16\nu(q)} \left(\frac{q(f_1 \otimes f_2 \otimes f_3)^{1/2}}{q(\ad f_1) q(\ad f_2) q(\ad f_3)}\right)^{-1/2} \frac{\Lambda\left(\frac{1}{2}, f_1 \otimes f_2 \otimes f_3\right)}{\Lambda(1, \ad f_1) \Lambda(1, \ad f_2) \Lambda(1, \ad f_3)}	\\
\times \prod_{p \mid q} I_p'(\varphi_p \otimes \widetilde{\varphi}_p) \prod_{\ell = 1}^{3} \left(\int_{\Gamma_0(q) \backslash \Hb} \left|f_{\ell}(z)\right|^2 \, d\mu(z) \int_{\Gamma_0(q) \backslash \Hb} \left|\widetilde{f}_{\ell}(z)\right|^2 \, d\mu(z)\right)^{1/2}.
\end{multline*}
This ad\`{e}lic-to-classical interpretation of the Watson--Ichino formula uses the fact that $\Lambda(2) = \pi / 6$ and $\vol(\Gamma_0(q) \backslash \Hb) = \pi \nu(q) / 3$, as well as the identity
\[\int\limits_{\Zgp(\A_{\Q}) \GL_2(\Q) \backslash \GL_2(\A_{\Q})} \phi(g) \, dg = \frac{2}{\vol(\Gamma_0(q) \backslash \Hb)} \int_{\Gamma_0(q) \backslash \Hb} f(z) \, d\mu(z)\]
for $f \in L^1(\Gamma_0(q) \backslash \Hb)$ with corresponding ad\`{e}lic lift $\phi \in L^1(\Zgp(\A_{\Q}) \GL_2(\Q) \backslash \GL_2(\A_{\Q}))$; the factor $2$ is present for this is the Tamagawa number of $\Zgp(\A_{\Q}) \GL_2(\Q) \backslash \GL_2(\A_{\Q})$.

\begin{corollary}
\label{sqfreetripleproductcor}
For squarefree $q = q_1 q_2$, $g \in \BB_0^{\ast}(q,\chi)$ with $\chi$ primitive, $f \in \BB_0^{\ast}(q_1)$ normalised such that $\langle g,g\rangle_q = \langle f,f\rangle_q = 1$, and $w_1, w_2 \mid q_2$, we have that
\begin{multline*}
\int_{\Gamma_0(q) \backslash \Hb} |g(z)|^2 (\iota_{w_1} f)(z) \, d\mu(z) \int_{\Gamma_0(q) \backslash \Hb} |g(z)|^2 (\iota_{w_2} f)(z) \, d\mu(z)	\\
= \frac{1 + \epsilon_f}{16 \sqrt{q_1} \nu(q_2)} \frac{\Lambda\left(\frac{1}{2}, f\right) \Lambda\left(\frac{1}{2}, f \otimes \ad g\right)}{\Lambda(1, \ad g)^2 \Lambda(1,\sym^2 f)}.
\end{multline*}
\end{corollary}

\begin{proof}
We have the isobaric decomposition $g \otimes \overline{g} = 1 \boxplus \ad g$, so that $g \otimes \overline{g} \otimes f = f \boxplus f \otimes \ad g$, while $f = \overline{f}$ implies that $\ad f = \sym^2 f$, and $\ad \overline{g} = \ad g$. Consequently, the conductor $q(g \otimes \overline{g} \otimes f)$ also factorises as $q(f) q(f \otimes \ad g)$. The conductors of $f$, $f \otimes \ad g$, $\ad g$, and $\sym^2 f$ are $q_1$, $q^4 q_1$, $q^2$, and $q_1^2$ respectively (cf.~\hyperref[rootnumberlemma]{Lemma \ref*{rootnumberlemma}}).

We denote by $\pi_g$, $\pi_{\overline{g}}$, $\pi_f$ the cuspidal automorphic representations of $\GL_2(\A_{\Q})$ associated to $g$, $\overline{g}$, $f$ respectively; note that $\pi_{\overline{g}} = \widetilde{\pi}_g$. The Watson--Ichino formula gives
\begin{multline*}
\int_{\Gamma_0(q) \backslash \Hb} |g(z)|^2 (\iota_{w_1} f)(z) \, d\mu(z) \int_{\Gamma_0(q) \backslash \Hb} |g(z)|^2 (\iota_{w_2} f)(z) \, d\mu(z)	\\
= \frac{(1 + \epsilon_f) q \sqrt{q_1}}{16\nu(q)} \frac{\Lambda\left(\frac{1}{2}, f\right) \Lambda\left(\frac{1}{2}, f \otimes \ad g\right)}{\Lambda(1, \ad g)^2 \Lambda(1,\sym^2 f)} \prod_{p \mid q} I_p'(\varphi_p \otimes \widetilde{\varphi}_p).
\end{multline*}
It remains to determine the local constants $I_p'(\varphi_p \otimes \widetilde{\varphi}_p)$. We observe the following:
\begin{itemize}
\item When $p \mid q_1$, the local component $\pi_{g,p}$ of $g$ is a unitarisable ramified principal series representation $\omega_{1,p} \boxplus \omega_{1,p}'$, where the unitary characters $\omega_{1,p}, \omega_{1,p}'$ of $\Q_p^{\times}$ have conductor exponents $c(\omega_{1,p}) = 1$ and $c(\omega_{1,p}') = 0$. The local component $\pi_{f,p}$ of $f$ is a special representation $\omega_{3,p} \St$, where $\omega_{3,p}$ is either the trivial character or the unramified quadratic character of $\Q_p^{\times}$. Finally, $\varphi_{1,p}$, $\varphi_{2,p}$, $\varphi_{3,p}$, $\widetilde{\varphi}_{1,p}$, $\widetilde{\varphi}_{2,p}$, $\widetilde{\varphi}_{3,p}$ are all local newforms.
\item When $p \mid q_2$ but $p \nmid [w_1,w_2]$, the local component $\pi_{g,p}$ of $g$ is of the same form as for $p \mid q_1$. The local component $\pi_{f,p}$ of $f$ is a unitarisable unramified principal series representation $\omega_{3,p} \boxplus \omega_{3,p}^{-1}$, where $c(\omega_{3,p}) = 0$ and $p^{-1/2} < |\omega_{3,p}(p)| < p^{1/2}$. Once again, all local forms are newforms.
\item When $p \mid (w_1,w_2)$, the setting is as above except both $\varphi_{3,p}$ and $\widetilde{\varphi}_{3,p}$ are translates of local newforms by $\pi_{3,p}\left(\begin{smallmatrix} p^{-1} & 0 \\ 0 & 1 \end{smallmatrix}\right)$ and $\widetilde{\pi}_{3,p}\left(\begin{smallmatrix} p^{-1} & 0 \\ 0 & 1 \end{smallmatrix}\right)$ respectively.
\item When $p \mid w_1$ but $p \nmid w_2$, the setting is as above except only $\varphi_{3,p}$ is the translate of the local newform.
\item Finally, when $p \mid w_2$ but $p \nmid w_1$, the setting is as above except instead only $\widetilde{\varphi}_{3,p}$ is the translate of the local newform.
\end{itemize}
For the former case, we apply \hyperref[dihedral2Stprop]{Proposition \ref*{dihedral2Stprop}} with $F_v = \Q_p$ and $q_v = p$, while \hyperref[Collinsprop]{Proposition \ref*{Collinsprop}} is applied to the remaining cases. This gives the result.
\end{proof}

\subsection{Proof of \texorpdfstring{\hyperref[dihedraltripleproductprop]{Proposition \ref*{dihedraltripleproductprop}}}{Proposition \ref{dihedraltripleproductprop}}}

\begin{proof}[Proof of {\hyperref[dihedraltripleproductprop]{Proposition \ref*{dihedraltripleproductprop}}}]
The identity \eqref{gpsi^2EWatsonIchino} for $|\langle |g|^2, E_{\infty}(\cdot,1/2 + it)\rangle_q|^2$ follows from \hyperref[g^2Eiscor]{Corollary \ref*{g^2Eiscor}}, while \hyperref[sqfreetripleproductcor]{Corollary \ref*{sqfreetripleproductcor}} gives the identity \eqref{gpsi^2fWatsonIchino} for $|\langle |g|^2, f\rangle_q|^2$.
\end{proof}

\begin{remark}
It behoves us to mention that both \cite[Section 4]{Luo14} and \cite[Section 2]{Liu15} mistakenly apply identities of Watson \cite{Wat08} that are only valid when all three automorphic forms $f_1,f_2,f_3$ have principal nebentypen; the correct identities are given in \hyperref[dihedraltripleproductprop]{Proposition \ref*{dihedraltripleproductprop}} and rely on \hyperref[dihedral2Stprop]{Propositions \ref*{dihedral2Stprop}} and \ref{Collinsprop}. Ultimately, this does not affect the validity of \cite[Theorem]{Luo14}. For \cite{Liu15}, there are two additional errata: the factorisations of $L(s,f \otimes f \otimes g)$ in \cite[(2.3) and (2.4)]{Liu15} are interchanged (with the same issue also being present in \cite[p.~422]{Sar01}), for the isobaric decompositions $f \otimes f = \chi_{-q} \boxplus \sym^2 f$ and $\sym^2 f = F \boxplus 1$ imply the correct factorisations
\begin{align*}
L(s,f \otimes f \otimes g) & = L(s,g \otimes \chi_{-q}) L(s,\sym^2 f \otimes g),	\\
L(s,\sym^2 f \otimes g) & = L(s,F \otimes g) L(s,g),
\end{align*}
and finally the approximate functional equation for $L(1/2,F \otimes g)$ given in \cite[Proof of Lemma 3.2]{Liu15} ought to involve a sum over $n \leq q^{3/2 + \e}$, not $q^{1 + \e}$ (which is to say that the conductor of $F \otimes g$ is $q^3$, not $q^2$; see \hyperref[rootnumberlemma]{Lemma \ref*{rootnumberlemma}}). The first of these two errata is readily rectified; the second, however, means that the exponent in \cite[Theorem 1.1]{Liu15} is subsequently weakened to $-2/3 - \delta/3 + \e$ rather than $-11/12 - \delta/3 + \e$.
\end{remark}

\section{Local Constants in the Watson--Ichino Formula}
\label{localconstantsect}

This section is devoted to the proofs of \hyperref[dihedral2Stprop]{Propositions \ref*{dihedral2Stprop}} and \ref{Collinsprop}. Since every calculation is purely local, we drop the subscripts $v$. Let $F$ be a nonarchimedean local field with ring of integers $\OO_F$, uniformiser $\varpi$, and maximal ideal $\pp = \varpi \OO_F$. Let $q = N(\pp) = \# \OO_F / \pp = |\varpi|^{-1}$, where the norm $|\cdot|$ is such that $|x| = q^{-v(x)}$ for $x \in \varpi^{v(x)} \OO_F^{\times}$. We set $K \defeq \GL_2(\OO_F)$ and define the congruence subgroup
\[K_1(\pp^m) \defeq \left\{\begin{pmatrix} a & b \\ c & d \end{pmatrix} \in K : c, d - 1 \in \pp^m\right\}\]
for any nonnegative integer $m$. We normalise the additive Haar measure $da$ on $F$ to give $\OO_F$ volume $1$, while the multiplicative Haar measure $d^{\times} a = \zeta_F(1) |a|^{-1} \, da$ on $F^{\times}$ is normalised to give $\OO_F^{\times}$ volume $1$, where $\zeta_F(s) = (1 - q^{-s})^{-1}$.

\subsection{Reduction to Formul\ae{} for Whittaker Functions}

For $\pi$ equal to a principal series representation $\omega \boxplus \omega'$ or a special representation $\omega \St$, and given a vector $\varphi_{\pi}$ in the induced model of $\pi$, we let
\begin{equation}
\label{Whittaker}
W_{\pi}(g) = \frac{\zeta_F(2)^{1/2}}{\zeta_F(1)} \int_F \varphi_{\pi}\left(w \begin{pmatrix} 1 & x \\ 0 & 1 \end{pmatrix} g\right) \psi(x) \, dx,
\end{equation}
denote the corresponding element of the Whittaker model $\WW(\pi,\psi)$, where $w = \left(\begin{smallmatrix} 0 & -1 \\ 1 & 0 \end{smallmatrix}\right)$ and $\psi$ is an unramified additive character of $F$; the normalisation of the Whittaker functional follows \cite[Section 3.2.1]{MV10}.

For generic irreducible unitarisable representations $\pi_1$, $\pi_2$, $\pi_3$ with $\pi_1$ a principal series representation, and for $\varphi_1$ in the induced model of $\pi_1$, $W_2 \in \WW(\pi_2,\psi)$, and $W_3 \in \WW(\pi_3,\psi^{-1})$, we define the local Rankin--Selberg integral $\ell_{\RS}(\varphi_1, W_2, W_3)$ to be
\[\zeta_F(1)^{1/2} \int_K \int_{F^{\times}} \varphi_1\left(\begin{pmatrix} a & 0 \\ 0 & 1 \end{pmatrix} k\right) W_2\left(\begin{pmatrix} a & 0 \\ 0 & 1 \end{pmatrix} k\right) W_3\left(\begin{pmatrix} a & 0 \\ 0 & 1 \end{pmatrix} k\right) \, \frac{d^{\times} a}{|a|} \, dk\]
(see \cite[(3.28)]{MV10}). The importance of this quantity is the following identity of Michel and Venkatesh.

\begin{lemma}[{\cite[Lemma 3.4.2]{MV10}}]
\label{MVlemma}
For $g,h \in \GL_2(F)$, $\varphi = \varphi_{\pi_1} \otimes \varphi_{\pi_2} \otimes \pi_3(g) \cdot \varphi_{\pi_3}$, and $\widetilde{\varphi} = \widetilde{\varphi}_{\pi_1} \otimes \widetilde{\varphi}_{\pi_2} \otimes \widetilde{\pi}_3(h) \cdot \widetilde{\varphi}_{\pi_3}$ with $\varphi_{\pi_1}$, $\varphi_{\pi_2}$, $\varphi_{\pi_3}$, $\widetilde{\varphi}_{\pi_1}$, $\widetilde{\varphi}_{\pi_2}$, $\widetilde{\varphi}_{\pi_3}$ newforms, we have the identity
\[\ell_{\RS}\left(\varphi_{\pi_1}, W_{\pi_2}, \pi_3(g) \cdot W_{\pi_3}\right) \ell_{\RS}\left(\widetilde{\varphi}_{\pi_1}, \widetilde{W}_{\pi_2}, \widetilde{\pi}_3(h) \cdot \widetilde{W}_{\pi_3}\right) = I(\varphi \otimes \widetilde{\varphi})\]
whenever $\pi_2$ is tempered.
\end{lemma}

\begin{remark}
\cite[Lemma 3.4.2]{MV10} only covers the case $g = h$, but the proof generalises via the polarisation identity.
\end{remark}

\subsection{Formul\ae{} for Whittaker Functions}
\label{Whittakerformulaesect}

\hyperref[MVlemma]{Lemma \ref*{MVlemma}} reduces the determination of local constants to evaluating integrals involving $\varphi_{\pi_1}$, $\varphi_{\pi_2}$, and $W_{\pi_3}$. Thus we must determine the values of these functions at certain values of $g \in \GL_2(F)$. We observe that both $\varphi_{\pi}$ and $W_{\pi}$ are right $K_1(\pp^{c(\pi)})$-invariant, where $c(\pi)$ denotes the conductor exponent of $\pi$; we will use this fact to limit ourselves to determining the values of these functions at $g = \left(\begin{smallmatrix} a & 0 \\ 0 & 1 \end{smallmatrix}\right)$ and $g = \left(\begin{smallmatrix} a & 0 \\ 0 & 1 \end{smallmatrix}\right) \left(\begin{smallmatrix} 1 & 0 \\ 1 & 1 \end{smallmatrix}\right)$.

We are interested in two cases, namely $\pi_1 = \omega_1 \boxplus \omega_1'$, $\pi_2 = \omega_1'^{-1} \boxplus \omega_1^{-1}$ with $\omega_1$, $\omega_1'$ both unitary, $c(\omega_1) = 1$ and $c(\omega_1') = 0$, so that $c(\pi_1) = c(\pi_2) = 1$, and $\pi_3$ either $\omega_3 \St$ with $\omega_3$ unitary and $c(\omega_3) = 0$, so that $c(\pi_3) = 1$, or $\omega_3 \boxplus \omega_3^{-1}$ with $q^{-1/2} < |\omega_3(\varpi)| < q^{1/2}$ and $c(\omega_3) = 0$, so that $c(\pi_3) = 0$. Moreover, we require that the product of the central characters of $\pi_1$, $\pi_2$, $\pi_3$ be trivial: in the former case, as the central character of $\pi_3$ is $\omega_3^2$, this means that $\omega_3^{-1} = \omega_3$, so that $\omega_3$ is either the trivial character or the unramified quadratic character of $F^{\times}$.

\subsubsection{The case $\pi_3 = \omega_3 \St$}

In this section, we deal with the first case, so that $\pi_3 = \omega_3 \St$.

\begin{lemma}[{\cite[Lemma 1.1.1]{Sch02}}]
We have that
\begin{equation}
\label{psizeta}
\int_{\varpi^m \OO_F^{\times}} \psi(x) \, dx = \begin{dcases*}
\frac{1}{q^m} \frac{1}{\zeta_F(1)} & if $m \geq 0$,	\\
-1 & if $m = -1$,	\\
0 & if $m \leq -2$.
\end{dcases*}
\end{equation}
\end{lemma}

\begin{lemma}[{\cite[Proposition 2.1.2]{Sch02}}]
The newform for $\pi_1$ in the induced model, normalised such that $W_{\pi_1} \left(\begin{smallmatrix} 1 & 0 \\ 0 & 1 \end{smallmatrix}\right) = 1$, is given by
\begin{equation}
\label{varphi1newform}
\varphi_{\pi_1}(g) = \begin{dcases*}
\frac{\zeta_F(1)}{\zeta_F(2)^{1/2}} \omega_1(a) \omega_1'(d) \left|\frac{a}{d}\right|^{1/2} & if $g = \begin{pmatrix} a & b \\ 0 & d \end{pmatrix} \begin{pmatrix} 1 & 0 \\ 1 & 1 \end{pmatrix} k$, $k \in K_1(\pp)$,	\\
0 & if $g = \begin{pmatrix} a & b \\ 0 & d \end{pmatrix} k$, $k \in K_1(\pp)$.
\end{dcases*}
\end{equation}
The newform for $\pi_3$ is equal to
\begin{equation}
\label{varphi3newform}
\varphi_{\pi_3}(g) = \begin{dcases*}
\zeta_F(2)^{1/2} \omega_3(ad) \left|\frac{a}{d}\right| & if $g = \begin{pmatrix} a & b \\ 0 & d \end{pmatrix} \begin{pmatrix} 1 & 0 \\ 1 & 1 \end{pmatrix} k$, $k \in K_1(\pp)$,	\\
-q \zeta_F(2)^{1/2} \omega_3(ad) \left|\frac{a}{d}\right| & if $g = \begin{pmatrix} a & b \\ 0 & d \end{pmatrix} k$, $k \in K_1(\pp)$.
\end{dcases*}
\end{equation}
\end{lemma}

Note that the normalisation of these newforms differs slightly than the normalisation in \cite[Proposition 2.1.2]{Sch02}; it is such that $W_{\pi_3}\left(\begin{smallmatrix} 1 & 0 \\ 0 & 1 \end{smallmatrix}\right) = 1$.

\begin{lemma}[{\cite[\S 2.4]{Sch02}}]
\label{Whittaker10varpi1lemma}
For $a \in F^{\times}$, we have that
\begin{align*}
W_{\pi_1}\begin{pmatrix} a & 0 \\ 0 & 1 \end{pmatrix} & = \begin{dcases*}
\omega_1'(a) |a|^{1/2} & if $0 < |a| \leq 1$,	\\
0 & if $|a| \geq q$,
\end{dcases*}	\\
W_{\pi_2}\begin{pmatrix} a & 0 \\ 0 & 1 \end{pmatrix} & = \begin{dcases*}
\omega_1'^{-1}(a) |a|^{1/2} & if $0 < |a| \leq 1$,	\\
0 & if $|a| \geq q$,
\end{dcases*}	\\
W_{\pi_3}\begin{pmatrix} a & 0 \\ 0 & 1 \end{pmatrix} & = \begin{dcases*}
\omega_3(a) |a| & if $0 < |a| \leq 1$,	\\
0 & if $|a| \geq q$.
\end{dcases*}
\end{align*}
\end{lemma}

\begin{proof}
Let
\[g = w \begin{pmatrix} 1 & x \\ 0 & 1 \end{pmatrix} \begin{pmatrix} a & 0 \\ 0 & 1 \end{pmatrix} = \begin{pmatrix} 0 & -1 \\ a & x \end{pmatrix}.\]
Then
\[g = \begin{dcases*}
\begin{pmatrix} \frac{a}{a + x \varpi} & -\varpi - \frac{a}{a + x \varpi} \\ 0 & a + x \varpi \end{pmatrix} \begin{pmatrix} 1 & 0 \\ 1 & 1 \end{pmatrix} \begin{pmatrix} \varpi + \frac{a}{a + x \varpi} & -1 + \frac{x}{a + x \varpi} \\ -\varpi & 1 \end{pmatrix} & if $|x| \leq |a|$,	\\
\begin{pmatrix} \frac{a}{x} & -1 \\ 0 & x \end{pmatrix} \begin{pmatrix} 1 & 0 \\ \frac{a}{x} & 1 \end{pmatrix} & if $|x| \geq q|a|$,
\end{dcases*}\]
and so upon combining \eqref{Whittaker} and \eqref{varphi1newform},
\[W_{\pi_1}\begin{pmatrix} a & 0 \\ 0 & 1 \end{pmatrix} = \omega_1(a) |a|^{1/2} \int\limits_{|x| \leq |a|} \omega_1^{-1} \omega_1'(a + x \varpi) \psi(x) |a + x \varpi|^{-1} \, dx.\]
Since $|x| \leq |a|$, $|a + x \varpi| = |a|$, 
while $\omega_1^{-1} \omega_1'(a + x \varpi) = \omega_1^{-1} \omega_1'(a)$ as $\omega_1^{-1} \omega_1'$ has conductor exponent $1$. So
\[W_{\pi_1}\begin{pmatrix} a & 0 \\ 0 & 1 \end{pmatrix} = \omega_1'(a) |a|^{-1/2} \int\limits_{|x| \leq |a|} \psi(x) \, dx ,\]
from which the desired identity for $W_{\pi_1} \left(\begin{smallmatrix} a & 0 \\ 0 & 1 \end{smallmatrix}\right)$ follows via \eqref{psizeta}. The identity for $W_{\pi_2} \left(\begin{smallmatrix} a & 0 \\ 0 & 1 \end{smallmatrix}\right)$ follows by taking complex conjugates. Finally, we insert \eqref{varphi3newform} into \eqref{Whittaker} in order to see that $W_{\pi_3} \left(\begin{smallmatrix} a & 0 \\ 0 & 1 \end{smallmatrix}\right)$ is equal to
\[\frac{\zeta_F(2)}{\zeta_F(1)} \omega_3(a) \left(|a|^{-1} \int\limits_{|x| \leq |a|} \psi(x) \, dx - q|a| \int\limits_{|x| \geq q|a|} \psi(x) |x|^{-2} \, dx\right).\]
The result then follows once again via \eqref{psizeta}.
\end{proof}

\begin{lemma}[{\cite[Lemma 1.1.1]{Sch02}}]
For any ramified character $\omega$ of $F^{\times}$ of conductor exponent $c(\omega) \geq 1$ and any $s \in \C$, we have that
\begin{equation}
\label{epsilon}
\int_{\varpi^m \OO_F^{\times}} \omega^{-1}(x) \psi(x) |x|^{-s} \, dx = \begin{dcases*}
\epsilon(s,\omega,\psi) & if $m = -c(\omega)$,	\\
0 & otherwise.
\end{dcases*}
\end{equation}
Here $\epsilon(s,\omega,\psi) = \epsilon(1/2,\omega,\psi) q^{- c(\omega) (s - 1/2)}$ and $|\epsilon(1/2,\omega,\psi)| = 1$.
\end{lemma}

\begin{lemma}[Cf.~{\cite[Lemma 2.13]{Hu17}}]
\label{Whittaker1011lemma}
We have that
\begin{align*}
W_{\pi_1}\left(\begin{pmatrix} a & 0 \\ 0 & 1 \end{pmatrix} \begin{pmatrix} 1 & 0 \\ 1 & 1 \end{pmatrix}\right) & = \begin{dcases*}
\epsilon(1, \omega_1 \omega_1'^{-1}, \psi) \omega_1(a) \psi(-a) |a|^{1/2} & if $0 < |a| \leq q$,	\\
0 & if $|a| \geq q^2$,
\end{dcases*}	\\
W_{\pi_2}\left(\begin{pmatrix} a & 0 \\ 0 & 1 \end{pmatrix} \begin{pmatrix} 1 & 0 \\ 1 & 1 \end{pmatrix}\right) & = \begin{dcases*}
\epsilon(1, \omega_1^{-1} \omega_1', \psi^{-1}) \omega_1^{-1}(a) \psi(a) |a|^{1/2} & if $0 < |a| \leq q$,	\\
0 & if $|a| \geq q^2$,
\end{dcases*}	\\
W_{\pi_3}\left(\begin{pmatrix} a & 0 \\ 0 & 1 \end{pmatrix} \begin{pmatrix} 1 & 0 \\ 1 & 1 \end{pmatrix} \right) & = \begin{dcases*}
-\frac{1}{q} \omega_3(a) \psi(-a) |a| & if $0 < |a| \leq q$,	\\
0 & if $|a| \geq q^2$.
\end{dcases*}
\end{align*}
\end{lemma}

\begin{proof}
Let
\[g = w \begin{pmatrix} 1 & x \\ 0 & 1 \end{pmatrix} \begin{pmatrix} a & 0 \\ 0 & 1 \end{pmatrix} \begin{pmatrix} 1 & 0 \\ 1 & 1 \end{pmatrix} = \begin{pmatrix} -1 & -1 \\ a + x & x \end{pmatrix}.\]
Then
\[g = \begin{dcases*}
\begin{pmatrix} \frac{a}{x} & -1 \\ 0 & x \end{pmatrix} \begin{pmatrix} 1 & 0 \\ \frac{a}{x} + 1 & 1 \end{pmatrix} & if $|x + a| \leq \dfrac{|a|}{q}$,	\\
\begin{pmatrix} \frac{a}{a + x} & -\frac{2a + x}{a + x} \\ 0 & a + x \end{pmatrix} \begin{pmatrix} 1 & 0 \\ 1 & 1 \end{pmatrix} \begin{pmatrix} 1 & -\frac{a}{a + x} \\ 0 & 1 \end{pmatrix} & if $|x + a| \geq |a|$.
\end{dcases*}\]
Combining \eqref{Whittaker} and \eqref{varphi1newform} yields
\[W_{\pi_1}\left(\begin{pmatrix} a & 0 \\ 0 & 1 \end{pmatrix} \begin{pmatrix} 1 & 0 \\ 1 & 1 \end{pmatrix}\right) = \omega_1(a) |a|^{1/2} \int\limits_{|x + a| \geq |a|} \omega_1^{-1}\omega_1'(x + a) \psi(x) |x + a|^{-1} \, dx.\]
Upon making the change of variables $x \mapsto x - a$ and using \eqref{epsilon}, the identity for $W_{\pi_1}$ is derived. The identity for $W_{\pi_2}$ follows by taking complex conjugates. Finally, combining \eqref{Whittaker} and \eqref{varphi3newform} shows that
\[W_{\pi_3}\left(\begin{pmatrix} a & 0 \\ 0 & 1 \end{pmatrix} \begin{pmatrix} 1 & 0 \\ 1 & 1 \end{pmatrix} \right) = \frac{\zeta_F(2)}{\zeta_F(1)} \omega_3(a) |a| \left(-q \int\limits_{|x + a| \leq \frac{|a|}{q}} \psi(x) |x|^{-2} \, dx + \int\limits_{|x + a| \geq |a|} \psi(x) |x + a|^{-2} \, dx\right).\]
The result then follows via \eqref{psizeta} after the change of variables $x \mapsto x - a$.
\end{proof}

\subsubsection{The case $\pi_3 = \omega_3 \boxplus \omega_3^{-1}$}

Finally, we deal with the case $\pi_3 = \omega_3 \boxplus \omega_3^{-1}$.

\begin{lemma}
The newform in the induced model is
\begin{equation}
\label{varphi3newformunram}
\varphi_{\pi_3}(g) = \frac{\zeta_F(1) L(1,\omega_3^2)}{\zeta_F(2)^{1/2}} \omega_3\left(\frac{a}{d}\right) \left|\frac{a}{d}\right|^{1/2} \text{ for $g = \begin{pmatrix} a & b \\ 0 & d \end{pmatrix} k$, $k \in K$.}
\end{equation}
\end{lemma}

Again, the normalisation is such that $W_{\pi_3}\left(\begin{smallmatrix} 1 & 0 \\ 0 & 1 \end{smallmatrix}\right) = 1$.

\begin{lemma}[{\cite[\S 2.4]{Sch02}}]
\label{Whittakerunram10varpi1lemma}
We have that
\[W_{\pi_3}\begin{pmatrix} a & 0 \\ 0 & 1 \end{pmatrix} = \begin{dcases*}
\sum_{m = 0}^{v(a)} \omega_3(\varpi)^{m} \omega_3^{-1}(\varpi)^{v(a) - m} |a|^{1/2} & if $0 < |a| \leq 1$,	\\
0 & if $|a| \geq q$.
\end{dcases*}\]
\end{lemma}

\begin{lemma}
\label{Whittakershiftunramlemma}
We have that
\[\pi_3\begin{pmatrix} \varpi^{-1} & 0 \\ 0 & 1 \end{pmatrix} \cdot W_{\pi_3}\left(\begin{pmatrix} a & 0 \\ 0 & 1 \end{pmatrix} \begin{pmatrix} 1 & 0 \\ \varpi & 1 \end{pmatrix}\right) = W_{\pi_3}\begin{pmatrix} \varpi^{-1} a & 0 \\ 0 & 1 \end{pmatrix}.\]
\end{lemma}

\begin{proof}
This follows from the fact that $\pi_3\left(\begin{smallmatrix} \varpi^{-1} & 0 \\ 0 & 1 \end{smallmatrix}\right) \cdot W_{\pi_3}$ is right $K_1(\pp)$-invariant.
\end{proof}

\begin{lemma}
\label{Whittakerunram1011lemma}
We have that
\begin{multline*}
\pi_3\begin{pmatrix} \varpi^{-1} & 0 \\ 0 & 1 \end{pmatrix} \cdot W_{\pi_3}\left(\begin{pmatrix} a & 0 \\ 0 & 1 \end{pmatrix} \begin{pmatrix} 1 & 0 \\ 1 & 1 \end{pmatrix}\right)	\\
= \begin{dcases*}
\sum_{m = 0}^{v(a) + 1} \omega_3(\varpi)^{m} \omega_3^{-1}(\varpi)^{v(a) + 1 - m} \psi(-a) \left(\frac{|a|}{q}\right)^{1/2} & if $0 < |a| \leq q$,	\\
0 & if $|a| \geq q^2$.
\end{dcases*}
\end{multline*}
\end{lemma}

\begin{proof}
For
\[g = w \begin{pmatrix} 1 & x \\ 0 & 1 \end{pmatrix} \begin{pmatrix} a & 0 \\ 0 & 1 \end{pmatrix} \begin{pmatrix} 1 & 0 \\ 1 & 1 \end{pmatrix} = \begin{pmatrix} -1 & -1 \\ a + x & x \end{pmatrix},\]
we have that
\[g \begin{pmatrix} \varpi^{-1} & 0 \\ 0 & 1 \end{pmatrix} = \begin{dcases*}
\begin{pmatrix} \frac{a}{\varpi x} & -1 \\ 0 & x \end{pmatrix} \begin{pmatrix} 1 & 0 \\ \varpi^{-1}\left(\frac{a}{x} + 1\right) & 1 \end{pmatrix} & if $|x + a| \leq \dfrac{|a|}{q}$,	\\
\begin{pmatrix} \frac{a}{a + x} & -\varpi^{-1} \\ 0 & \varpi^{-1}(a + x) \end{pmatrix} \begin{pmatrix} 0 & -1 \\ 1 & \frac{\varpi x}{a + x} \end{pmatrix} & if $|x + a| \geq |a|$.
\end{dcases*}\]
From this and \eqref{varphi3newformunram}, $\varphi_{\pi_3}\left(w \left(\begin{smallmatrix} 1 & x \\ 0 & 1 \end{smallmatrix}\right) g \left(\begin{smallmatrix} \varpi^{-1} & 0 \\ 0 & 1 \end{smallmatrix}\right)\right)$ is equal to
\[\begin{dcases*}
\frac{\zeta_F(1) L(1,\omega_3^2)}{\zeta_F(2)^{1/2}} \omega_3\left(\frac{a}{\varpi}\right) (q|a|)^{1/2} \omega_3(x)^{-2} |x|^{-1} & if $|x + a| \leq \dfrac{|a|}{q}$,	\\
\frac{\zeta_F(1) L(1,\omega_3^2)}{\zeta_F(2)^{1/2}} \omega_3(\varpi a) \left(\frac{|a|}{q}\right)^{1/2} \omega_3(x + a)^{-2} |x + a|^{-1} & if $|x + a| \geq |a|$.
\end{dcases*}\]
Coupled with \eqref{Whittaker}, $\pi_3\left(\begin{smallmatrix} \varpi^{-1} & 0 \\ 0 & 1 \end{smallmatrix}\right) \cdot W_{\pi_3}\left(\left(\begin{smallmatrix} a & 0 \\ 0 & 1 \end{smallmatrix}\right) \left(\begin{smallmatrix} 1 & 0 \\ 1 & 1 \end{smallmatrix}\right)\right)$ is thereby equal to
\begin{multline*}
L(1,\omega_3^2) \omega_3^{-1}(\varpi a) \psi(-a) \left(\frac{|a|}{q}\right)^{-1/2} \int\limits_{|x| \leq \frac{|a|}{q}} \psi(x) \, dx	\\
+ L(1,\omega_3^2) \omega_3(\varpi a) \psi(-a) \left(\frac{|a|}{q}\right)^{1/2} \int\limits_{|x| \geq |a|} \omega_3(x)^{-2} \psi(x) |x|^{-1} \, dx
\end{multline*}
after making the change of variables $x \mapsto x - a$, which gives the result via \eqref{psizeta}.
\end{proof}

\subsection{Proofs of \texorpdfstring{\hyperref[dihedral2Stprop]{Propositions \ref*{dihedral2Stprop}}}{Propositions \ref{dihedral2Stprop}} and \ref{Collinsprop}}

To prove \hyperref[dihedral2Stprop]{Propositions \ref*{dihedral2Stprop}} and \ref{Collinsprop}, we use \hyperref[MVlemma]{Lemma \ref*{MVlemma}} to reduce the problem to evaluating local Rankin--Selberg integrals. We then use the identities in \hyperref[Whittakerformulaesect]{Section \ref*{Whittakerformulaesect}} for values of $\varphi_{\pi}$ and $W_{\pi}$ together with the following lemma.

\begin{lemma}[{\cite[Lemma 2.2]{Hu16}}]
\label{Hulemma}
Suppose that $f : \GL_2(F) \to \C$ is right $K$-integrable and right $K_1(\pp^m)$-invariant for some $m \in \N$. Then
\[\int_K f(gk) \, dk = \sum_{j = 0}^{m} A_j f\left(g \begin{pmatrix} 1 & 0 \\ \varpi^j & 1 \end{pmatrix}\right), \qquad A_j \defeq \begin{dcases*}
\frac{\zeta_F(2)}{\zeta_F(1)} & if $j = 0$,	\\
\frac{1}{q^j} \frac{\zeta_F(2)}{\zeta_F(1)^2} & if $1 \leq j \leq m - 1$,	\\
\frac{1}{q^m} \frac{\zeta_F(2)}{\zeta_F(1)} & if $j = m$.
\end{dcases*}\]
\end{lemma}

\begin{proof}[Proof of {\hyperref[dihedral2Stprop]{Proposition \ref*{dihedral2Stprop}}}]
\hyperref[Whittaker10varpi1lemma]{Lemmata \ref*{Whittaker10varpi1lemma}}, \ref{Whittaker1011lemma}, and \ref{Hulemma} imply that
\[\ell_{\RS}(\varphi_{\pi_1},W_{\pi_2},W_{\pi_3}) = - \frac{1}{q} \left(\zeta_F(1) \zeta_F(2)\right)^{1/2} \epsilon(1, \omega_1^{-1} \omega_1', \psi^{-1}) \int\limits_{0 < |a| \leq q} \omega_3(a) |a| \, d^{\times} a.\]
The integral is readily seen to be equal to $q \omega_3^{-1}(\varpi) L(1,\omega_3)$ via the change of variables $a \mapsto \varpi^{-1} a$; \hyperref[MVlemma]{Lemma \ref*{MVlemma}} then gives the identity
\[I(\varphi \otimes \widetilde{\varphi}) = \frac{1}{q} \zeta_F(1) \zeta_F(2) L(1, \omega_3)^2.\]
Now
\[\langle \varphi, \widetilde{\varphi} \rangle = \langle W_{\pi_1}, \widetilde{W}_{\pi_1} \rangle \langle W_{\pi_2}, \widetilde{W}_{\pi_2} \rangle \langle W_{\pi_3}, \widetilde{W}_{\pi_3} \rangle,\]
where
\[\langle W_{\pi},\widetilde{W}_{\pi} \rangle \defeq \int_{F^{\times}} \left| W_{\pi} \begin{pmatrix} a & 0 \\ 0 & 1 \end{pmatrix}\right|^2 \, d^{\times} a,\]
and \hyperref[Whittaker10varpi1lemma]{Lemma \ref*{Whittaker10varpi1lemma}} implies that
\[\langle W_{\pi_1}, \widetilde{W}_{\pi_1} \rangle = \zeta_F(1), \qquad \langle W_{\pi_2}, \widetilde{W}_{\pi_2} \rangle = \zeta_F(1), \qquad \langle W_{\pi_3}, \widetilde{W}_{\pi_3} \rangle = \zeta_F(2).\]
We conclude that
\[\frac{I(\varphi \otimes \widetilde{\varphi})}{\langle \varphi, \widetilde{\varphi} \rangle} = \frac{1}{q} \frac{L(1,\omega_3)^2}{\zeta_F(1)}.\]
On the other hand, we have the isobaric decomposition
\[\pi_1 \otimes \pi_2 \otimes \pi_3 = \omega_1 \omega_1'^{-1} \omega_3 \St \boxplus \omega_1^{-1} \omega_1' \omega_3 \St \boxplus \omega_3 \St \boxplus \omega_3 \St,\]
so that
\[L(s,\pi_1 \otimes \pi_2 \otimes \pi_3) = L\left(s + \frac{1}{2}, \omega_3\right)^2.\]
Moreover,
\[\ad \pi_1 = \ad \pi_2 = \omega_1 \omega_1'^{-1} \boxplus \omega_1^{-1} \omega_1' \boxplus 1,\]
so that
\[L(s, \ad \pi_1) = L(s, \ad \pi_2) = \zeta_F(s),\]
while $\ad \pi_3$ is the special representation of $\GL_3(F)$ associated to the trivial character, so that
\[L(s, \ad \pi_3) = \zeta_F(s + 1).\]
So
\[\frac{\zeta_F(2)^2 L\left(\frac{1}{2}, \pi_1 \otimes \pi_2 \otimes \pi_3\right)}{L(1, \ad \pi_1) L(1, \ad \pi_2) L(1, \ad \pi_3)} = \frac{\zeta_F(2) L(1, \omega_3)^2}{\zeta_F(1)^2},\]
and consequently, upon recalling \eqref{I_v'},
\[I'(\varphi \otimes \widetilde{\varphi}) = \frac{1}{q} \frac{\zeta_F(1)}{\zeta_F(2)} = \frac{1}{q} \left(1 + \frac{1}{q}\right).\qedhere\]
\end{proof}

\begin{remark}
\label{localconstrem}
A similar calculation shows that $I'(\varphi \otimes \widetilde{\varphi})$ is again equal to $q^{-1} (1 + q^{-1})$ when $\pi_1$, $\pi_2$, $\pi_3$ are all irreducible unitarisable principal series representations of conductor exponent one for which $\omega_{\pi_1} \omega_{\pi_2} \omega_{\pi_3} = 1$.
\end{remark}

\begin{proof}[Proof of {\hyperref[Collinsprop]{Proposition \ref*{Collinsprop}}}]
For $\pi_3 = \omega_3 \boxplus \omega_3^{-1}$ with $c(\omega_3) = c(\omega_3^{-1}) = 0$, the right $K$-invariance of $W_{\pi_3}$ allow us to see that $\ell_{\RS}(\varphi_{\pi_1}, W_{\pi_2}, W_{\pi_3})$ is equal to
\[\left(\zeta_F(1) \zeta_F(2)\right)^{1/2} \epsilon(1, \omega_1^{-1} \omega_1', \psi^{-1}) \int\limits_{0 < |a| \leq 1} \sum_{m = 0}^{v(a)} \omega_3(\varpi)^m \omega_3^{-1}(\varpi)^{v(a) - m} \psi(a) |a|^{1/2} \, d^{\times} a\]
via \hyperref[Whittaker10varpi1lemma]{Lemmata \ref*{Whittaker10varpi1lemma}}, \ref{Whittaker1011lemma}, \ref{Whittakerunram10varpi1lemma}, and \ref{Hulemma}. The integral simplifies to $L(1/2, \omega_3) L(1/2, \omega_3^{-1})$. Similarly, the Rankin--Selberg integral $\ell_{\RS}\left(\varphi_{\pi_1}, W_{\pi_2}, \pi_3\left(\begin{smallmatrix} \varpi^{-1} & 0 \\ 0 & 1 \end{smallmatrix}\right) \cdot W_{\pi_3}\right)$ is equal to
\[\left(\frac{\zeta_F(1) \zeta_F(2)}{q}\right)^{1/2} \epsilon(1, \omega_1^{-1} \omega_1', \psi^{-1}) \int\limits_{0 < |a| \leq q} \sum_{m = 0}^{v(a) + 1} \omega_3(\varpi)^{m} \omega_3^{-1}(\varpi)^{v(a) + 1 - m} |a|^{1/2} \, d^{\times} a,\]
additionally using \hyperref[Whittakershiftunramlemma]{Lemmata \ref*{Whittakershiftunramlemma}} and \ref{Whittakerunram1011lemma}. After making the change of variables $a \mapsto \varpi^{-1}a$, we see that this is equal to $\ell_{\RS}(\varphi_{\pi_1}, W_{\pi_2}, W_{\pi_3})$. So by \hyperref[MVlemma]{Lemma \ref*{MVlemma}},
\[I(\varphi \otimes \widetilde{\varphi}) = \frac{1}{q} \zeta_F(1) \zeta_F(2) L\left(\frac{1}{2},\omega_3\right)^2 L\left(\frac{1}{2},\omega_3^{-1}\right)^2.\]
As
\[\langle W_{\pi_3}, \widetilde{W}_{\pi_3} \rangle = \left\langle \pi_3\begin{pmatrix} \varpi^{-1} & 0 \\ 0 & 1 \end{pmatrix} \cdot W_{\pi_3}, \widetilde{\pi}_3\begin{pmatrix} \varpi^{-1} & 0 \\ 0 & 1 \end{pmatrix} \cdot \widetilde{W}_{\pi_3} \right\rangle = \frac{\zeta_F(1) L(1, \ad \pi_3)}{\zeta_F(2)}\]
(see, for example, \cite[Section 3.4.1]{MV10}), we find that
\[\frac{I(\varphi \otimes \widetilde{\varphi})}{\langle \varphi, \widetilde{\varphi} \rangle} = \frac{1}{q} \frac{\zeta_F(2)^2 L\left(\frac{1}{2},\omega_3\right)^2 L\left(\frac{1}{2},\omega_3^{-1}\right)^2}{\zeta_F(1)^2 L(1, \ad \pi_3)}.\]
On the other hand,
\[\pi_1 \otimes \pi_2 \otimes \pi_3 = \omega_1 \omega_1'^{-1} \omega_3 \boxplus \omega_1^{-1} \omega_1' \omega_3 \boxplus \omega_1 \omega_1'^{-1} \omega_3^{-1} \boxplus \omega_1^{-1} \omega_1' \omega_3^{-1} \boxplus \omega_3 \boxplus \omega_3 \boxplus \omega_3^{-1} \boxplus \omega_3^{-1},\]
so that
\[L(s, \pi_1 \otimes \pi_2 \otimes \pi_3) = L(s,\omega_3)^2 L(s,\omega_3^{-1})^2,\]
and so
\[I'(\varphi \otimes \widetilde{\varphi}) = \frac{1}{q}.\qedhere\]
\end{proof}

\begin{remark}
\label{matrixcoeffremark}
One can also prove \hyperref[dihedral2Stprop]{Propositions \ref*{dihedral2Stprop}} and \ref{Collinsprop} by the methods used in \cite{Hu17}: in place of \hyperref[MVlemma]{Lemma \ref*{MVlemma}}, we instead calculate $I(\varphi \otimes \widetilde{\varphi})$ via the fact that
\[\frac{I(\varphi \otimes \widetilde{\varphi})}{\langle \varphi, \widetilde{\varphi} \rangle} = \int\limits_{\Zgp(F) \backslash \GL_2(F)} \Phi_{\pi_1}(g) \Phi_{\pi_2}(g) \Phi_{\pi_3}(g) \, dg,\]
recalling \eqref{I_v}, where $\Phi_{\pi}$ denotes the normalised matrix coefficient
\[\Phi_{\pi}(g) \defeq \frac{\langle \pi(g) \cdot W_{\pi}, \widetilde{W}_{\pi} \rangle}{\langle W_{\pi}, \widetilde{W}_{\pi} \rangle} = \frac{1}{\langle W_{\pi}, \widetilde{W}_{\pi} \rangle} \int_{F^{\times}} W_{\pi}\left(\begin{pmatrix} a & 0 \\ 0 & 1 \end{pmatrix} g\right) \widetilde{W}_{\pi} \begin{pmatrix} a & 0 \\ 0 & 1 \end{pmatrix} \, d^{\times} a.\]
Since $W_{\pi}$ is right $K_1(\pp)$-invariant, \hyperref[Hulemma]{Lemma \ref*{Hulemma}} together with the Iwasawa decomposition imply that $I(\varphi \otimes \widetilde{\varphi}) / \langle \varphi, \widetilde{\varphi} \rangle$ is equal to
\begin{equation}
\label{Iviamatrixcoeffs}
\frac{\zeta_F(2)}{\zeta_F(1)} \int\limits_{\Zgp(F) \backslash \Bgp(F)} \prod_{j = 1}^{3} \Phi_{\pi_j}\left(b \begin{pmatrix} 1 & 0 \\ 1 & 1 \end{pmatrix}\right) \, db + \frac{1}{q} \frac{\zeta_F(2)}{\zeta_F(1)} \int\limits_{\Zgp(F) \backslash \Bgp(F)} \prod_{j = 1}^{3} \Phi_{\pi_j}(b) \, db,
\end{equation}
where $b = \left(\begin{smallmatrix} a & x \\ 0 & 1 \end{smallmatrix}\right)$ with $a \in F^{\times}$, $x \in F$, and $db = |a|^{-1} \, d^{\times} a \, dx$. One can then use \hyperref[Whittaker10varpi1lemma]{Lemmata \ref*{Whittaker10varpi1lemma}} and \ref{Whittaker1011lemma} and the fact that $W_{\pi}\left(\left(\begin{smallmatrix} 1 & x \\ 0 & 1 \end{smallmatrix}\right) g\right) = \psi(x) W_{\pi}(g)$ to show that
\begin{align*}
\Phi_{\pi_1}(b) & = \begin{dcases*}
\omega_1'(a) |a|^{-1/2} & if $|a| \geq \max\{|x|,q\}$,	\\
\omega_1'(a) |a|^{1/2} & if $\max\{|a|,|x|\} \leq 1$,	\\
0 & otherwise,
\end{dcases*}	\\
\Phi_{\pi_2}(b) & = \begin{dcases*}
\omega_1'^{-1}(a) |a|^{-1/2} & if $|a| \geq \max\{|x|,q\}$,	\\
\omega_1'^{-1}(a) |a|^{1/2} & if $\max\{|a|,|x|\} \leq 1$,	\\
0 & otherwise,
\end{dcases*}	\\
\Phi_{\pi_3}(b) & = \begin{dcases*}
-q \omega_3(a) |a| |x|^{-2} & if $|x| \geq \max\{q|a|, q\}$,	\\
\omega_3(a) |a|^{-1} & if $|a| \geq \max\{|x|,q\}$,	\\
\omega_3(a) |a| & if $\max\{|a|,|x|\} \leq 1$,
\end{dcases*}	\\
\Phi_{\pi_1}\left(b \begin{pmatrix} 1 & 0 \\ 1 & 1 \end{pmatrix}\right) & = \begin{dcases*}
\omega_1(a) \omega_1^{-1}\omega_1'(x - a) |a|^{1/2} |x - a|^{-1} & if $|x - a| \geq \max\{|a|,q\}$,	\\
0 & otherwise,
\end{dcases*}	\\
\Phi_{\pi_2}\left(b \begin{pmatrix} 1 & 0 \\ 1 & 1 \end{pmatrix}\right) & = \begin{dcases*}
\omega_1^{-1}(a) \omega_1\omega_1'^{-1}(x - a) |a|^{1/2} |x - a|^{-1} & if $|x - a| \geq \max\{|a|,q\}$,	\\
0 & otherwise,
\end{dcases*}	\\
\Phi_{\pi_3}\left(b \begin{pmatrix} 1 & 0 \\ 1 & 1 \end{pmatrix}\right) & = \begin{dcases*}
\omega_3(a) |a| |x - a|^{-2} & if $|x - a| \geq \max\{|a|,q\}$,	\\
-q \omega_3(a) |a|^{-1} & if $|a| \geq \max\{q|x - a|,q\}$,	\\
-\frac{1}{q} \omega_3(a) |a| & if $\max\{|x - a|,|a|\} \leq 1$,
\end{dcases*}
\end{align*}
where $\pi_1$, $\pi_2$, $\pi_3$ are as in \hyperref[dihedral2Stprop]{Proposition \ref*{dihedral2Stprop}}. Inserting these identities into \eqref{Iviamatrixcoeffs} and evaluating the resulting integrals thereby reproves \hyperref[dihedral2Stprop]{Proposition \ref*{dihedral2Stprop}}; similar calculations yield \hyperref[Collinsprop]{Proposition \ref*{Collinsprop}}.
\end{remark}

\section{The First Moment in the Short Initial Range}
\label{firstmomentsect}

The main results of this section are bounds for the first moments
\begin{align*}
\widetilde{\MM}^{\Maass}(h) & \defeq \sum_{d_1 d_2 = D} \frac{\varphi(d_2)}{d_2} \sum_{\substack{f \in \BB_0^{\ast}(\Gamma_0(d_1)) \\ \epsilon_f = 1}} \frac{L\left(\frac{1}{2},f \otimes g_{\psi^2}\right)}{L_{d_2}\left(\frac{1}{2}, f\right) L^{d_2}(1,\sym^2 f)} h(t_f),	\\
\widetilde{\MM}^{\Eis}(h) & \defeq \frac{1}{2\pi} \int_{-\infty}^{\infty} \left|\frac{L\left(\frac{1}{2} + it, g_{\psi^2}\right)}{\zeta_D\left(\frac{1}{2} + it\right) \zeta^D(1 + 2it)}\right|^2 h(t) \, dt,	\\
\widetilde{\MM}^{\hol}(h^{\hol}) & \defeq \sum_{d_1 d_2 = D}\frac{\varphi(d_2)}{d_2} \sum_{\substack{f \in \BB_{\hol}^{\ast}(\Gamma_0(d_1)) \\ k_f \equiv 0 \hspace{-.25cm} \pmod{4}}} \frac{L\left(\frac{1}{2},f \otimes g_{\psi^2}\right)}{L_{d_2}\left(\frac{1}{2}, f\right) L^{d_2}(1,\sym^2 f)} h^{\hol}(k_f),
\end{align*}
which will be required in the course of the proof of \hyperref[dihedralmomentsprop]{Proposition \ref*{dihedralmomentsprop} (1)}.

\begin{proposition}
\label{firstmomentprop}
Fix $\beta > 0$, and suppose that $t_g^{\beta} \leq T \leq t_g^{1 - \beta}$.
\begin{enumerate}[leftmargin=*]
\item[\emph{(1)}] For $h(t) = 1_{E \cup -E}(t)$ with $E = [T,2T]$,
\[\widetilde{\MM}^{\Maass}(h) + \widetilde{\MM}^{\Eis}(h) \ll_{\e} T^{2 + \e} + t_g^{1 + \e}.\]
\item[\emph{(2)}] For $h^{\hol}(k) = 1_E(k)$ with $E = [T,2T]$,
\[\widetilde{\MM}^{\hol}(h^{\hol}) \ll_{\e} T^{2 + \e} + t_g^{1 + \e}.\]
\end{enumerate}
\end{proposition}

Were we to replace $g_{\psi^2}$ with an Eisenstein series $E(z,1/2 + 2it_g)$, so that $L(1/2,f \otimes g_{\psi^2})$ would be replaced by $|L(1/2 + 2it_g,f)|^2$, then we would immediately obtain the desired bound via the large sieve, \hyperref[largesievethm]{Theorem \ref*{largesievethm}}. Thus this result is of similar strength to the large sieve; in particular, dropping all but one term returns the convexity bounds for $L(1/2,f \otimes g_{\psi^2})$ and $|L(1/2 + it,g_{\psi^2})|^2$ for $T \ll t_g^{1/2}$. However, we cannot proceed via the large sieve as in the Eisenstein case because we do not know how to bound $L(1/2,f \otimes g_{\psi^2})$ by the square of a Dirichlet polynomial of length $t_g^2$, and if we were to instead first apply the Cauchy--Schwarz inequality and then use the large sieve, we would only obtain the bound $O_{\e}(T^{2 + \e} + t_g^{2 + \e})$, which is insufficient for our requirements.

Our approach to prove \hyperref[firstmomentprop]{Proposition \ref*{firstmomentprop}} is to first use the approximate functional equation to write the $L$-functions involved as Dirichlet polynomials and then apply the Kuznetsov and Petersson formul\ae{} in order to express $\widetilde{\MM}^{\Maass}(h) + \widetilde{\MM}^{\Eis}(h)$ and $\widetilde{\MM}^{\hol}(h^{\hol})$ in terms of a delta term, which is trivially bounded, and sums of Kloosterman sums. We then open up the Kloosterman sums and apply the Vorono\u{\i} summation formula. The proof is completed via employing a stationary phase-type argument to the ensuing expression.

\begin{remark}
This strategy is used elsewhere to obtain results that are similar to \hyperref[firstmomentprop]{Proposition \ref*{firstmomentprop}}. Holowinsky and Templier use this approach in order to prove \cite[Theorem 5]{HT14}, which gives a hybrid level aspect bound for a first moment of Rankin--Selberg $L$-functions involving holomorphic forms of fixed weight; the moment involves a sum over holomorphic newforms $f$ of level $N$, while $g_{\psi}$ is of level $M$, and the bound for this moment is a hybrid bound in terms of $N$ and $M$ (with unspecified polynomial dependence on the weights of $f$ and $g_{\psi}$). The first author and Radziwi\l{}\l{} have recently proven a hybrid bound \cite[Proposition 2.28]{HR19} akin to \hyperref[firstmomentprop]{Proposition \ref*{firstmomentprop}} where $g_{\psi}$ is replaced by the Eisenstein newform $E_{\chi,1}(z) \defeq E_{\infty}(z,1/2,\chi_D)$ of level $D$ and nebentypus $\chi_D$; the bound for this moment is a hybrid bound in terms of $T$ and $D$, and the method is also valid for cuspidal dihedral forms $g_{\psi}$ (with unspecified polynomial dependence on the weight or spectral parameter of $g_{\psi}$).
\end{remark}

In applying the approximate functional equation in order to prove \hyperref[firstmomentprop]{Proposition \ref*{firstmomentprop}}, we immediately run into difficulties because the length of the approximate functional equation depends on the level, and the Kuznetsov and Petersson formul\ae{} involve cusp forms of \emph{all} levels dividing $D$. Since we are evaluating a first moment rather than a second moment, we cannot merely use positivity and oversum the Dirichlet polynomial coming from the approximate functional equation.

One possible approach to overcome this obstacle would be to use the Kuznetsov and Petersson formul\ae{} for newforms; see \cite[Lemma 5]{HT14} and \cite[Section 10.2]{You19}. Instead, we work around this issue by using the Kuznetsov and Petersson formul\ae{} associated to the pair of cusps $(\aa,\bb)$ with $\aa \sim \infty$ and $\bb \sim 1$. As shall be seen, this introduces the root number of $f \otimes g_{\psi^2}$ in such a way to give approximate functional equations of the correct length for each level dividing $D$.

We will give the proof of \hyperref[firstmomentprop]{Proposition \ref*{firstmomentprop} (1)}, then describe the minor modifications needed for the proof of \hyperref[firstmomentprop]{Proposition \ref*{firstmomentprop} (2)}. Via the positivity of $L(1/2,f \otimes g_{\psi^2})$, it suffices to prove the result with $h$ replaced by
\begin{equation}
\label{hTfirstmomenteq}
h_T(t) \defeq e^{-\left(\frac{t - T}{T^{1 - \e}}\right)^2} + e^{-\left(\frac{t + T}{T^{1 - \e}}\right)^2}.
\end{equation}
We remind the reader that from here onwards, we will make use of many standard automorphic tools that are detailed in \hyperref[toolboxappendix]{Appendix \ref*{toolboxappendix}}.

\begin{lemma}
\label{MMtildeidentitylemma}
The first moment $\widetilde{\MM}^{\Maass}(h_T) + \widetilde{\MM}^{\Eis}(h_T)$ is equal to
\begin{multline}
\label{firstmomenteq}
\frac{D}{2} \int_{-\infty}^{\infty} \widetilde{V_2^1}\left(\frac{1}{D^{3/2}},r\right) h_T(r) \, d_{\spec}r	\\
+ \frac{D}{2} \sum_{\pm} \sum_{n = 1}^{\infty} \frac{\lambda_{g_{\psi^2}}(n)}{\sqrt{n}} \sum_{\substack{c = 1 \\ c \equiv 0 \hspace{-.25cm} \pmod{D}}}^{\infty} \frac{S(1,\pm n;c)}{c} \left(\Ks^{\pm} \widetilde{V_2^1}\left(\frac{n}{D^{3/2}},\cdot\right) h_T\right)\left(\frac{\sqrt{n}}{c}\right)	\\
+ \frac{D}{2} \sum_{\pm} \sum_{n = 1}^{\infty} \frac{\lambda_{g_{\psi^2}}(n)}{\sqrt{n}} \sum_{\substack{c = 1 \\ (c,D) = 1}}^{\infty} \frac{S(1,\pm n\overline{D};c)}{c\sqrt{D}} \left(\Ks^{\pm} \widetilde{V_2^1}\left(\frac{n}{D^{3/2}},\cdot\right) h_T\right)\left(\frac{\sqrt{n}}{c\sqrt{D}}\right),
\end{multline}
where
\begin{align*}
\widetilde{V_2^1}(x,r) & \defeq \sum_{\ell = 1}^{\infty} \frac{\chi_D(\ell)}{\ell} V_2^1(x\ell^2,r)	\\
& = \frac{1}{2\pi i} \int_{\sigma - i\infty}^{\sigma + i\infty} L(1 + 2s,\chi_D) e^{s^2} x^{-s} \prod_{\pm_1} \prod_{\pm_2} \frac{\Gamma_{\R}\left(\frac{1}{2} + s \pm_1 i(2t_g \pm_2 r)\right)}{\Gamma_{\R}\left(\frac{1}{2} \pm_1 i(2t_g \pm_2 r)\right)} \, \frac{ds}{s}.
\end{align*}
\end{lemma}

Here $d_{\spec}r$, $S(m,n;c)$, $\Ks^{\pm}$, and $V_2^1$ are as in \eqref{dspeceq}, \eqref{Kloosteq}, \eqref{NsKspmeq}, and \eqref{V2eq} respectively.

\begin{proof}
We take $m = 1$ and $h = V_2^1(n \ell^2/D^{3/2},\cdot) h_T$ in the Kuznetsov formula, \hyperref[Kuznetsovthm]{Theorem \ref*{Kuznetsovthm}}, using the explicit expressions in \hyperref[Kuznetsovlemma]{Lemma \ref*{Kuznetsovlemma}}, which we then multiply by $\chi_D(\ell)/2\sqrt{n} \ell$ and sum over $n,\ell \in \N$ and over both the same sign and opposite sign Kuznetsov formul\ae{}. After making the change of variables $n \mapsto w_2 n$, using the fact that $\lambda_{g_{\psi^2}}(w_2 n) = \lambda_{g_{\psi^2}}(n)$ for all $w_2 \mid D$ via \hyperref[dihedralHeckeeigenlemma]{Lemma \ref*{dihedralHeckeeigenlemma}}, and simplifying the resulting sum over $v_2 w_2 = \ell$ using the multiplicativity of the summands, the spectral sum ends up as
\begin{multline*}
\sum_{d_1 d_2 = D} \frac{\varphi(d_2)}{d_2} \sum_{\substack{f \in \BB_0^{\ast}(\Gamma_0(d_1)) \\ \epsilon_f = 1}} \frac{h_T(t_f)}{L^{d_2}(1,\sym^2 f)}	\\
\times \sum_{v_2 w_2 = d_2} \frac{\nu(v_2)}{v_2} \frac{\mu(w_2) \lambda_f(w_2)}{\sqrt{w_2}} \sum_{n = 1}^{\infty} \sum_{\ell = 1}^{\infty} \frac{\lambda_f(n) \lambda_{g_{\psi^2}}(n) \chi_D(\ell)}{\sqrt{n} \ell} V_2^1\left(\frac{w_2 n \ell^2}{D^{3/2}},t_f\right).
\end{multline*}
We do the same with the Kuznetsov formula associated to the $(\infty,1)$ pair of cusps, \hyperref[infty1Kuznetsovthm]{Theorem \ref*{infty1Kuznetsovthm}}, using the explicit expressions in \hyperref[infty1Kuznetsovlemma]{Lemma \ref*{infty1Kuznetsovlemma}}, obtaining
\begin{multline*}
\sum_{d_1 d_2 = D} \frac{\varphi(d_2)}{d_2} \sum_{\substack{f \in \BB_0^{\ast}(\Gamma_0(d_1)) \\ \epsilon_f = 1}} \frac{h_T(t_f)}{L^{d_2}(1,\sym^2 f)}	\\
\times \sum_{v_2 w_2 = d_2} \frac{\nu(v_2)}{v_2} \frac{\mu(w_2) \lambda_f(w_2)}{\sqrt{w_2}} \eta_f(d_1) \sum_{n = 1}^{\infty} \sum_{\ell = 1}^{\infty} \frac{\lambda_f(n) \lambda_{g_{\psi^2}}(n) \chi_D(\ell)}{\sqrt{n} \ell} V_2^1\left(\frac{v_2 n \ell^2}{D^{3/2}},t_f\right).
\end{multline*}
We add these two expressions together and use the approximate functional equation, \hyperref[approxfunclemma]{Lemma \ref*{approxfunclemma}}, with $X = \sqrt{d_2}/w_2$. Recalling \hyperref[EllMaassidentitylemma]{Lemma \ref*{EllMaassidentitylemma}}, this yields $\widetilde{\MM}^{\Maass}(h_T)$. Similarly, the sum of the Eisenstein terms is $\widetilde{\MM}^{\Eis}(h_T)$. Upon noting that the delta term only arises when we take $n = 1$ in the same sign Kuznetsov formula with the $(\infty,\infty)$ pair of cusps, the desired identity follows.
\end{proof}

\begin{lemma}
\label{Ks+lemma}
Both of the terms
\begin{gather*}
\sum_{n = 1}^{\infty} \frac{\lambda_{g_{\psi^2}}(n)}{\sqrt{n}} \sum_{\substack{c = 1 \\ c \equiv 0 \hspace{-.25cm} \pmod{D}}}^{\infty} \frac{S(1,n;c)}{c} \left(\Ks^{+} \widetilde{V_2^1}\left(\frac{n}{D^{3/2}},\cdot\right) h_T\right) \left(\frac{\sqrt{n}}{c}\right),	\\
\sum_{n = 1}^{\infty} \frac{\lambda_{g_{\psi^2}}(n)}{\sqrt{n}} \sum_{\substack{c = 1 \\ (c,D) = 1}}^{\infty} \frac{S(1,n\overline{D};c)}{c\sqrt{D}} \left(\Ks^{+} \widetilde{V_2^1}\left(\frac{n}{D^{3/2}},\cdot\right) h_T\right)\left(\frac{\sqrt{n}}{c\sqrt{D}}\right)
\end{gather*}
are $O_{\e}(t_g^{1 + \e})$.
\end{lemma}

\begin{proof}
The strategy is to apply the Vorono\u{\i} summation formula, \hyperref[Voronoilemma]{Lemma \ref*{Voronoilemma}}, to the sum over $n$, and then to bound carefully the resulting dual sum using a stationary phase-type argument (although this will be masked by integration by parts). We only cover the proof for the first term, since the second term follows by the exact same argument save for a slightly different formulation of the Vorono\u{\i} summation formula, which gives rise to Ramanujan sums in place of Gauss sums.

Dividing the $n$-sum and the $r$-integral in the definition of $\Ks^{+}$, \eqref{NsKspmeq}, into dyadic intervals, we consider the sum
\[\sum_{\substack{c = 1 \\ c \equiv 0 \hspace{-.25cm} \pmod{D}}}^{\infty} \sum_{n = 1}^{\infty} \frac{\lambda_{g_{\psi^2}}(n)}{\sqrt{n}} W\left(\frac{n}{N}\right) \frac{S(1,n;c)}{c} \left(\Ks^{+} \widetilde{V_2^1}\left(\frac{n}{D^{3/2}},\cdot\right) h\left(\frac{\cdot}{T}\right)\right) \left(\frac{\sqrt{n}}{c}\right)\]
for any $N < t_g^{2 + \e}$, where $W$ and $h$ are smooth functions compactly supported on $(1,2)$. Here the function $h_T$ has been absorbed into $h$. By Stirling's formula \eqref{Stirlingeq}, we have that
\begin{equation}
\label{vbound}
\frac{\partial^{j + k}}{\partial x^j \partial r^k} \widetilde{V_2^1}\left(\frac{Nx}{D^{3/2}}, rT\right) \ll_{j,k,\e} T^{\e}
\end{equation}
for $j,k \in \N_0$, where we follow the $\e$-convention. To understand the transform $\Ks^{+}$, we refer to \cite[Lemma 3.7]{BuK17a}. By \cite[(3.61)]{BuK17a}, we must bound
\begin{multline*}
\sum_{\substack{c = 1 \\ c \equiv 0 \hspace{-.25cm} \pmod{D}}}^{\infty} \sum_{n = 1}^{\infty} \frac{\lambda_{g_{\psi^2}}(n)}{\sqrt{n}} W\left(\frac{n}{N}\right) \frac{S(1,n;c)}{c} \\
\times \int_{-\infty}^{\infty} e\left(\frac{2\sqrt{n}}{c} \cosh\pi u\right) \int_{0}^{\infty} \widetilde{V_2^1}\left(\frac{n}{D^{3/2}}, r\right) h\left(\frac{r}{T}\right) r e(-ur) \tanh (\pi r)  \, dr \, du
\end{multline*}
by $O_{\e}(t_g^{1 + \e})$. We make the substitutions $r \mapsto rT$ and $u \mapsto u/T$. Repeated integration by parts with respect to $r$, recalling \eqref{vbound} and using $(d/dr)^k(\tanh \pi r T)\ll_k\ e^{-T}$ for $k\ge 1$, shows that we may restrict to $|u| < T^{\e}$, up to a negligible error. After making this restriction, using $\tanh(\pi r T)=1+O(e^{-T})$, and taking the Taylor expansion of $\cosh(\pi u / T)$, we need to show
\begin{multline*}
T \sum_{\substack{c = 1 \\ c \equiv 0 \hspace{-.25cm} \pmod{D}}}^{\infty} \sum_{n = 1}^{\infty} \frac{\lambda_{g_{\psi^2}}(n)}{\sqrt{n}} W\left(\frac{n}{N}\right) \frac{S(1,n;c)}{c} e\left(\frac{2\sqrt{n}}{c}\right) \\
\times \int_{-T^{\e}}^{T^{\e}} e\left(\frac{2\sqrt{n}}{c} \left(\frac{1}{2!} \left(\frac{\pi u}{T}\right)^2 + \frac{1}{4!} \left(\frac{\pi u}{T}\right)^4 + \cdots\right)\right) \int_{0}^{\infty} \widetilde{V_2^1}\left(\frac{n}{D^{3/2}}, rT \right) r h(r) e(-ur) \, dr \, du
\end{multline*}
is $O_{\e}(t_g^{1 + \e})$. Now we integrate by parts multiple times with respect to $u$, differentiating the exponential $e(\frac{2\sqrt{n}}{c} (\frac{1}{2!} (\frac{\pi u}{T})^2 + \frac{1}{4!} (\frac{\pi u}{T})^4 + \cdots))$ and integrating the exponential $e(-ur)$. This shows that we may restrict the summation over $c$ to $c < \sqrt{N} / T^{2 - \e}$, because the contribution of the terms not satisfying this condition will be negligible. In particular, we may assume that $N > T^{4 - \e}$, for otherwise the $c$-sum is empty. Also, the contribution of the endpoints $u=\pm T^\e$ after integration by parts is negligible by repeated integration by parts with respect to $r$ (the same argument which allowed us to truncate the $u$-integral in the first place). Thus we have shown that it suffices to prove that
\begin{equation}
\label{preVoronoifirstmomenteq}
T \sum_{\substack{c < \frac{\sqrt{N}}{T^{2 - \e}} \\ c \equiv 0 \hspace{-.25cm} \pmod{D}}} \sum_{n = 1}^{\infty} \frac{\lambda_{g_{\psi^2}}(n)}{\sqrt{n}} W\left(\frac{n}{N}\right) \frac{S(1,n;c)}{c} e\left(\frac{2\sqrt{n}}{c}\right) \Omega\left(\frac{\sqrt{n}}{cT^2}\right) \widetilde{V_2^1}\left(\frac{n}{D^{3/2}}, rT\right)
\end{equation}
is $O_{\e}(t_g^{1 + \e})$ for any smooth function $\Omega$ satisfying $\Omega^{(j)} \ll_j 1$ for $j \in \N_0$ and any $r \in (1,2)$.

We now open up the Kloosterman sum and apply the Vorono\u{\i} summation formula, \hyperref[Voronoilemma]{Lemma \ref*{Voronoilemma}}. Via Mellin inversion, \eqref{preVoronoifirstmomenteq} is equal to
\begin{multline}
\label{toprove}
\frac{T}{\pi i \sqrt{N}} \sum_{\pm} \sum_{\substack{c < \frac{\sqrt{N}}{T^{2 - \e}} \\ c \equiv 0 \hspace{-.25cm} \pmod{D}}} \sum_{n = 1}^{\infty} \frac{\lambda_{g_{\psi^2}}(n)}{n} \sum_{d \in (\Z/c\Z)^{\times}} \chi_D(d) e\left(\frac{d(n \mp 1)}{c}\right)	\\
\times \int_{\sigma - i\infty}^{\sigma + i\infty} \left(\frac{N n}{c^2}\right)^{-s} \widehat{\JJ_{2t_g}^{\pm}}(2(1 + s)) \int_{0}^{\infty} \frac{W(x)}{\sqrt{x}} e\left(\frac{2\sqrt{Nx}}{c}\right) \Omega\left(\frac{\sqrt{Nx}}{cT^2}\right) \widetilde{V_2^1}\left(\frac{Nx}{D^{3/2}}, rT \right) x^{-s - 1} \, dx \, ds
\end{multline}
for any $\sigma \geq 0$, where $\JJ_{2t_g}^{\pm}$ is as in \eqref{Jrpmeq} with Mellin transform $\widehat{\JJ_{2t_g}^{\pm}}$ given by \eqref{J_r+Mellin} and \eqref{J_r-Mellin}. Repeated integration by parts in the $x$ integral, integrating $x^{-s}$ and differentiating the rest and recalling \eqref{vbound}, shows that up to negligible error, we may restrict the $s$-integral to
\begin{equation}
\label{s-bound}
|\Im(s)| < \frac{\sqrt{N}}{c} t_g^{\e} < \frac{t_g^{1 + \e}}{c}.
\end{equation}
Moving the line of integration in \eqref{toprove} far to the right and using the bounds in \hyperref[JJrMellinasympcor]{Corollary \ref*{JJrMellinasympcor}} for the Mellin transform of $\JJ_{2t_g}^{\pm}$, we may crudely restrict to $n < t_g^{2 + \e}$. Upon fixing $\sigma = 0$ in \eqref{toprove}, so that the $s$-integral is on the line $s = it$ and $x^{-s} = e(-\frac{t\log x}{2\pi})$, and making the substitution $x \mapsto x^2$, it suffices to prove that
\begin{multline*}
\Xi \defeq \frac{T}{\sqrt{N}} \sum_{\pm} \sum_{\substack{c < \frac{\sqrt{N}}{T^{2 - \e}} \\ c \equiv 0 \hspace{-.25cm} \pmod{D}}} \sum_{n < t_g^{2 + \e}} \frac{\lambda_{g_{\psi^2}}(n)}{n} \sum_{a \mid \left(\frac{c}{D}, n \mp 1\right)} a \mu\left(\frac{c}{aD}\right) \chi_D\left(\frac{c}{aD}\right) \chi_D\left(\frac{n \mp 1}{a}\right) \\
\times \int_{|t| < \frac{\sqrt{N}}{c} t_g^\e} \left(\frac{N n}{c^2}\right)^{-it} \widehat{\JJ_{2t_g}^{\pm}}(2(1 + it)) I(t) \, dt
\end{multline*}
is $O_{\e}(t_g^{1 + \e})$, where we have used \hyperref[Miyakelemma]{Lemma \ref*{Miyakelemma}} to reexpress the sum over $d$ as a sum over $a \mid (c/D,n \mp 1)$, and
\[I(t) \defeq \int_{0}^{\infty} \frac{W(x^2)}{x^2} e\left(\frac{2\sqrt{N}x}{c} - \frac{t\log x}{\pi} \right) \Omega\left(\frac{\sqrt{N}x}{cT^2}\right) \widetilde{V_2^1}\left(\frac{Nx^2}{D^{3/2}}, rT \right) \, dx.\]
We write $\Xi = \Xi_1 + \Xi_2$, where $\Xi_1$ is the same expression as $\Xi$ but with the $t$-integral further restricted to
\[\left|t - \frac{2\pi \sqrt{N}}{c}\right| \leq \left(\frac{\sqrt{N}}{c}\right)^{\frac{1}{2} + \e}\]
and $\Xi_2$ is the same expression as $\Xi$ but with the $t$-integral further restricted to 
\begin{equation}
\label{restrict}
\left|t - \frac{2\pi \sqrt{N}}{c}\right| > \left(\frac{\sqrt{N}}{c}\right)^{\frac{1}{2} + \e}.
\end{equation}
Thus $\Xi_1$ keeps close to the stationary point of the $x$-integral in the definition of $I(t)$, while $\Xi_2$ keeps away.

We first bound $\Xi_1$. Using the bound $\widehat{\JJ_{2t_g}^{\pm}}(2(1 + it)) \ll_{\e} t_g^{1 + \e}$ in the range \eqref{s-bound} from \hyperref[JJrMellinasympcor]{Corollary \ref*{JJrMellinasympcor}} and the trivial bound $I(t) \ll 1$, we get
\[\Xi_1 \ll_{\e} T N^{\frac{1}{4}} t_g^{\e} \sum_{\pm} \sum_{\substack{c < \frac{\sqrt{N}}{T^{2 - \e}} \\ c \equiv 0 \hspace{-.25cm} \pmod{D}}} \frac{1}{\sqrt{c}} \sum_{n < t_g^{2 + \e}} \frac{|\lambda_{g_{\psi^2}}(n)|}{n} \sum_{a \mid \left(\frac{c}{D}, n \mp 1\right)} a \ll_{\e} t_g^{1 + \e}\]
upon making the change of variables $n \mapsto an \pm 1$ and recalling that $N < t_g^{2 + \e}$.

We now turn to bounding $\Xi_2$. The difference here is that we will not trivially bound the integral $I(t)$. Keeping in mind the restriction \eqref{restrict}, we write
\begin{multline*}
I(t) = \int_{0}^{\infty} \frac{W(x^2)}{x^2} \Omega\left(\frac{\sqrt{N}x}{cT^2}\right) \widetilde{V_2^1}\left(\frac{Nx^2}{D^{3/2}}, rT\right) \left(\frac{2\sqrt{N}}{c} - \frac{t}{\pi x}\right)^{-1}	\\
\times \left(\frac{2\sqrt{N}}{c} - \frac{t}{\pi x}\right) e\left(\frac{2\sqrt{N} x}{c} - \frac{t \log x}{\pi} \right) \, dx.
\end{multline*} 
We integrate by parts $k$-times with respect to $x$, differentiating the product of terms on the first line above and integrating the product of terms on the second line. This leads to the bound
\[I(t) \ll_k \left(\frac{\sqrt{N}}{cT^2}\right)^k \left|\frac{2\pi \sqrt{N}}{c} - t\right|^{-k} + (1 + |t|)^k \left|\frac{2\pi \sqrt{N}}{c} - t\right|^{-2k},\]
where the first term in the upper bound comes from the derivatives of $\Omega(\frac{\sqrt{N} x}{cT^2})$, while the second term comes from the derivatives of $(\frac{2\sqrt{N}}{c} - \frac{t}{\pi x})^{-1}$. By \eqref{s-bound} and \eqref{restrict}, the second term in this upper bound is negligible. The first term is negligible unless
\[\left|\frac{2\pi \sqrt{N}}{c} - t\right| \ll \left(\frac{\sqrt{N}}{cT^2}\right)^{1 + \e}.\]
But the contribution to $\Xi_2$ of $t$ in this range is
\begin{multline*}
\frac{T}{\sqrt{N}} \sum_{\pm} \sum_{\substack{c < \frac{\sqrt{N}}{T^{2 - \e}} \\ c \equiv 0 \hspace{-.25cm} \pmod{D}}} \sum_{n < t_g^{2 + \e}} \frac{\lambda_{g_{\psi^2}}(n)}{n} \sum_{a \mid \left(\frac{c}{D}, n \mp 1\right)} a \mu\left(\frac{c}{aD}\right) \chi_D\left(\frac{c}{aD}\right) \chi_D\left(\frac{n \mp 1}{a}\right) \\
\times \int_{\left|t - \frac{2\pi \sqrt{N}}{c}\right| \ll \left(\frac{\sqrt{N}}{cT^2}\right)^{1 + \e}} \left(\frac{Nn}{c^2}\right)^{-it} \widehat{\JJ_{2t_g}^{\pm}}(2(1 + it)) I(t) \, dt,
\end{multline*}
which is trivially bounded, using the fact that $\widehat{\JJ_{2t_g}^{\pm}}(2(1 + it)) \ll_{\e} t_g^{1 + \e}$, by
\[\frac{t_g^{1 + \e}}{T} \sum_{\pm} \sum_{\substack{c < \frac{\sqrt{N}}{T^{2 - \e}} \\ c \equiv 0 \hspace{-.25cm} \pmod{D}}} \frac{1}{c} \sum_{n < t_g^{2 + \e}} \frac{|\lambda_{g_{\psi^2}}(n)|}{n} \sum_{a \mid \left(\frac{c}{D}, n \mp 1\right)} a \ll_{\e} \frac{t_g^{1 + \e}}{T},\]
which is more than sufficient.
\end{proof}

\begin{lemma}
\label{Ks-lemma}
Both of the terms
\begin{gather*}
\sum_{n = 1}^{\infty} \frac{\lambda_{g_{\psi^2}}(n)}{\sqrt{n}} \sum_{\substack{c = 1 \\ c \equiv 0 \hspace{-.25cm} \pmod{D}}}^{\infty} \frac{S(1,-n;c)}{c} \left(\Ks^{-} \widetilde{V_2^1}\left(\frac{n}{D^{3/2}},\cdot\right) h_T\right) \left(\frac{\sqrt{n}}{c}\right),	\\
\sum_{n = 1}^{\infty} \frac{\lambda_{g_{\psi^2}}(n)}{\sqrt{n}} \sum_{\substack{c = 1 \\ (c,D) = 1}}^{\infty} \frac{S(1,-n\overline{D};c)}{c\sqrt{D}} \left(\Ks^{-} \widetilde{V_2^1}\left(\frac{n}{D^{3/2}},\cdot\right) h_T\right)\left(\frac{\sqrt{n}}{c\sqrt{D}}\right)
\end{gather*}
are $O_{\e}(t_g^{1 + \e})$.
\end{lemma}

\begin{proof}
The strategy is the same: to apply the Vorono\u{\i} summation formula to the sum over $n$, and then to bound trivially. This time, however, there will be no stationary phase analysis, so the proof is more straightforward. Again, we will only detail the proof of the bound for the first term.

Dividing as before the $n$-sum and the $r$-integral in the definition of $\mathcal{K}^-$ into dyadic intervals, we consider the sum
\[\sum_{\substack{c = 1 \\ c \equiv 0 \hspace{-.25cm} \pmod{D}}}^{\infty} \sum_{n = 1}^{\infty} \frac{\lambda_{g_{\psi^2}}(n)}{\sqrt{n}} W\left(\frac{n}{N}\right) \frac{S(1,-n;c)}{c} \left(\Ks^{-} \widetilde{V_2^1}\left(\frac{n}{D^{3/2}},\cdot\right) h\left(\frac{\cdot}{T}\right)\right) \left(\frac{\sqrt{n}}{c}\right)\]
for any $N < t_g^{2 + \e}$, where $W$ and $h$ are smooth functions compactly supported on $(1,2)$, with the function $h_T$ having been absorbed into $h$. To understand the transform $\Ks^{-}$, we refer to \cite[Lemma 3.8]{BuK17a}. By \cite[(3.68)]{BuK17a} and the fact that $\tanh \pi r = 1 + O(e^{-2\pi |r|})$, we must bound
\begin{multline*}
\sum_{\substack{c = 1 \\ c \equiv 0 \hspace{-.25cm} \pmod{D}}}^{\infty} \sum_{n = 1}^{\infty} \frac{\lambda_{g_{\psi^2}}(n)}{\sqrt{n}} W\left(\frac{n}{N}\right) \frac{S(1,-n;c)}{c} \\
\times \int_{-\infty}^{\infty} e\left(-\frac{2\sqrt{n}}{c} \sinh \pi u\right) \int_{0}^{\infty} \widetilde{V_2^1}\left(\frac{n}{D^{3/2}}, r\right) h\left(\frac{r}{T}\right) r e(-ur) \, dr \, du
\end{multline*}
by $O_{\e}(t_g^{1 + \e})$. We make the substitutions $r \mapsto Tr$ and $u \mapsto u/T$. Repeated integration by parts with respect to $r$ shows that we may restrict to $|u| < T^{\e}$, up to a negligible error. After making this restriction and taking the Taylor expansion of $\sinh(\pi u / T)$, we need to prove that
\begin{multline*}
T \sum_{\substack{c = 1 \\ c \equiv 0 \hspace{-.25cm} \pmod{D}}}^{\infty} \sum_{n = 1}^{\infty} \frac{\lambda_{g_{\psi^2}}(n)}{\sqrt{n}} W\left(\frac{n}{N}\right) \frac{S(1,-n;c)}{c} \\
\times \int_{-T^{\e}}^{T^{\e}} e\left(-\frac{2\sqrt{n}}{c} \left(\frac{\pi u}{T} + \frac{1}{3!} \left(\frac{\pi u}{T}\right)^3 + \cdots\right)\right) \int_{0}^{\infty} \widetilde{V_2^1}\left(\frac{n}{D^{3/2}}, rT \right) r h(r) e(-ur) \, dr \, du
\end{multline*}
is $O_{\e}(t_g^{1 + \e})$. We integrate by parts multiple times with respect to $u$, differentiating the exponential $e(-\frac{2\sqrt{n}}{c} (\frac{\pi u}{T} + \frac{1}{3!} (\frac{\pi u}{T})^3 + \cdots))$ and integrating the exponential $e(-ur)$. This shows that we may restrict the summation over $c$ to $c < \sqrt{N} / T^{1 - \e}$, because the contribution of the terms not satisfying this condition will be negligible. In particular, we may assume that $N > T^{2 - \e}$, for otherwise the $c$-sum is empty. Thus we have shown that it suffices to prove that
\begin{equation}
\label{preVoronoifirstmoment2eq}
T \sum_{\substack{c < \frac{\sqrt{N}}{T^{1 - \e}} \\ c \equiv 0 \hspace{-.25cm} \pmod{D}}} \sum_{n = 1}^{\infty} \frac{\lambda_{g_{\psi^2}}(n)}{\sqrt{n}} W\left(\frac{n}{N}\right) \frac{S(1,-n;c)}{c} \Omega\left(\frac{\sqrt{n}}{cT}\right) \widetilde{V_2^1}\left(\frac{n}{D^{3/2}}, rT\right)
\end{equation}
is $O_{\e}(t_g^{1 + \e})$ for any smooth function $\Omega$ satisfying $\Omega^{(j)} \ll_j 1$ for $j \in \N_0$ and any $r \in (1,2)$.

We now open up the Kloosterman sum and apply the Vorono\u{\i} summation formula, \hyperref[Voronoilemma]{Lemma \ref*{Voronoilemma}}. Via Mellin inversion, \eqref{preVoronoifirstmoment2eq} is equal to
\begin{multline}
\label{toprove2}
\frac{T}{\pi i \sqrt{N}} \sum_{\pm} \sum_{\substack{c < \frac{\sqrt{N}}{T^{2 - \e}} \\ c \equiv 0 \hspace{-.25cm} \pmod{D}}} \sum_{n = 1}^{\infty} \frac{\lambda_{g_{\psi^2}}(n)}{n} \sum_{d \in (\Z/c\Z)^{\times}} \chi_D(d) e\left(\frac{d(n \pm 1)}{c}\right)	\\
\times \int_{\sigma - i\infty}^{\sigma + i\infty} \left(\frac{N n}{c^2}\right)^{-s} \widehat{\JJ_{2t_g}^{\pm}}(2(1 + s)) \int_{0}^{\infty} \frac{W(x)}{\sqrt{x}} \Omega\left(\frac{\sqrt{Nx}}{cT}\right) \widetilde{V_2^1}\left(\frac{Nx}{D^{3/2}}, rT \right) x^{-s - 1} \, dx \, ds
\end{multline}
for any $\sigma \geq 0$. We again use \hyperref[Miyakelemma]{Lemma \ref*{Miyakelemma}} to write the Gauss sum over $d$ as a sum over $a \mid (c/D,n \pm 1)$. Repeated integration by parts in the $x$-integral shows that the $s$-integral may be restricted to
\[|\Im(s)| < \frac{\sqrt{N}}{cT} t_g^{\e} < \frac{t_g^{1 + \e}}{cT}.\]
Moving the line of integration in \eqref{toprove2} far to the right and using the bounds in \hyperref[JJrMellinasympcor]{Corollary \ref*{JJrMellinasympcor}} for $\widehat{\JJ_{2t_g}^{\pm}}$, we may once again restrict to $n < t_g^{2 + \e}$. Upon fixing $\sigma = 0$ in \eqref{toprove2} and bounding the resulting integral trivially by $\frac{\sqrt{N}}{cT} t_g^{1 + \e}$, since $\widehat{\JJ_{2t_g}^{\pm}}(2(1 + it)) \ll_{\e} t_g^{1 + \e}$, we arrive at the bound
\[t_g^{1 + \e} \sum_{\pm} \sum_{\substack{c < \frac{\sqrt{N}}{T^{2 - \e}} \\ c \equiv 0 \hspace{-.25cm} \pmod{D}}} \frac{1}{c} \sum_{n < t_g^{2 + \e}} \frac{|\lambda_{g_{\psi^2}}(n)|}{n} \sum_{a \mid \left(\frac{c}{D}, n \mp 1\right)} a \ll_{\e} t_g^{1 + \e}\]
upon making the change of variables $n \mapsto an \mp 1$ and recalling that $N < t_g^{2 + \e}$.
\end{proof}

\begin{proof}[Proof of {\hyperref[firstmomentprop]{Proposition \ref*{firstmomentprop} (1)}}]
It is clear that the first term in \eqref{firstmomenteq} is $O_{\e}(T^{2 + \e})$. \hyperref[Ks+lemma]{Lemmata \ref*{Ks+lemma}} and \ref{Ks-lemma} then bound the second and third terms by $O_{\e}(t_g^{1 + \e})$.
\end{proof}

\begin{proof}[Proof of {\hyperref[firstmomentprop]{Proposition \ref*{firstmomentprop} (2)}}]
A similar identity to \eqref{firstmomenteq} for $\widetilde{\MM}^{\hol}(h^{\hol})$ may be obtained by using the Petersson formula, \hyperref[Peterssonthm]{Theorems \ref*{Peterssonthm}} and \ref{infty1Peterssonthm}, instead of the Kuznetsov formula, namely
\begin{multline}
\label{firstmoment2eq}
\frac{D}{4\pi^2} \sum_{\substack{k = 4 \\ k \equiv 0 \hspace{-.25cm} \pmod{4}}}^{\infty} (k - 1) \widetilde{V_2^{\hol}}\left(\frac{1}{D^{3/2}},k\right) h^{\hol}(k)	\\
+ \frac{D}{2} \sum_{n = 1}^{\infty} \frac{\lambda_{g_{\psi^2}}(n)}{\sqrt{n}} \sum_{\substack{c = 1 \\ c \equiv 0 \hspace{-.25cm} \pmod{D}}}^{\infty} \frac{S(1,n;c)}{c} \left(\Ks^{\hol} \widetilde{V_2^{\hol}}\left(\frac{n}{D^{3/2}},\cdot\right) h^{\hol}\right)\left(\frac{\sqrt{n}}{c}\right)	\\
+ \frac{D}{2} \sum_{n = 1}^{\infty} \frac{\lambda_{g_{\psi^2}}(n)}{\sqrt{n}} \sum_{\substack{c = 1 \\ (c,D) = 1}}^{\infty} \frac{S(1,n\overline{D};c)}{c\sqrt{D}} \left(\Ks^{\hol} \widetilde{V_2^{\hol}}\left(\frac{n}{D^{3/2}},\cdot\right) h^{\hol}\right)\left(\frac{\sqrt{n}}{c\sqrt{D}}\right).
\end{multline}
Here $\Ks^{\hol}$ is as in \eqref{KsholJkholeq} and
\[\widetilde{V_2^{\hol}}(x,k) = \frac{1}{2\pi i} \int_{\sigma - i\infty}^{\sigma + i\infty} L(1 + 2s,\chi_D) e^{s^2} x^{-s} \prod_{\pm_1} \prod_{\pm_2} \frac{\Gamma_{\R}\left(s + \frac{k \pm_1 1}{2} \pm_2 2it_g\right)}{\Gamma_{\R}\left(\frac{1}{2} + \frac{k \pm_1 1}{2} \pm_2 2it_g\right)} \, \frac{ds}{s}.\]

The first term in \eqref{firstmoment2eq} is bounded by $O_{\e}(T^{2 + \e})$. For the latter two terms, we use the methods of \cite[Section 5.5]{Iwa97} to understand $\Ks^{\hol}$ in place of \cite[Lemmata 3.7 and 3.8]{BuK17a} to understand $\Ks^{\pm}$: this gives terms of the form
\begin{multline*}
\sum_{\pm} \sum_{\substack{c = 1 \\ c \equiv 0 \hspace{-.25cm} \pmod{D}}}^{\infty} \sum_{n = 1}^{\infty} \frac{\lambda_{g_{\psi^2}}(n)}{\sqrt{n}} W\left(\frac{n}{N}\right) \frac{S(1,n;c)}{c} \\
\times \int_{-\infty}^{\infty} e\left(\pm \frac{2\sqrt{n}}{c} \cos 2\pi u\right) \int_{0}^{\infty} \widetilde{V_2^{\hol}}\left(\frac{n}{D^{3/2}},r + 1\right) h^{\hol}(r + 1) r e(-ur) \, dr \, du
\end{multline*}
and
\begin{multline*}
\sum_{\pm} \sum_{\substack{c = 1 \\ c \equiv 0 \hspace{-.25cm} \pmod{D}}}^{\infty} \sum_{n = 1}^{\infty} \frac{\lambda_{g_{\psi^2}}(n)}{\sqrt{n}} W\left(\frac{n}{N}\right) \frac{S(1,n;c)}{c} \\
\times \int_{-\infty}^{\infty} e\left(\pm \frac{2\sqrt{n}}{c} \sin 2\pi u\right) \int_{0}^{\infty} \widetilde{V_2^{\hol}}\left(\frac{n}{D^{3/2}},r + 1\right) h^{\hol}(r + 1) r e(-ur) \, dr \, du,
\end{multline*}
as well as the counterparts involving sums over $c \in \N$ with $(c,D) = 1$. The former term is then treated via the same methods as \hyperref[Ks+lemma]{Lemma \ref*{Ks+lemma}}, while the latter is treated as in \hyperref[Ks-lemma]{Lemma \ref*{Ks-lemma}}.
\end{proof}

\section{Spectral Reciprocity for the Short Initial Range}
\label{spectralrecsect}

The main result of this section is an identity for
\[\MM^{\pm}(\hf) \defeq \MM^{\Maass}(h) + \MM^{\Eis}(h) + \delta_{+,\pm} \MM^{\hol}(h^{\hol})\]
for a (suitably well-behaved) function $\hf \defeq (h,h^{\hol}) : (\R \cup i(-1/2,1/2)) \times 2\N \to \C^2$, with $\MM^{\Maass}(h)$ and $\MM^{\Eis}(h)$ as in \eqref{MMMaasseq} and \eqref{MMEiseq}, and
\[\MM^{\hol}\left(h^{\hol}\right) \defeq \sum_{d_1 d_2 = D} 2^{\omega(d_2)} \frac{\varphi(d_2)}{d_2} \sum_{f \in \BB_{\hol}^{\ast}(\Gamma_0(d_1))} \frac{L^{d_2}\left(\frac{1}{2},f\right) L\left(\frac{1}{2},f \otimes \chi_D\right) L\left(\frac{1}{2},f \otimes g_{\psi^2}\right)}{L^{d_2}(1,\sym^2 f)} h^{\hol}(k_f).\]
We will take $\hf$ to be an admissible function in the sense of \cite[Lemma 8b)]{BlK19b}, namely $h(t)$ is even and holomorphic in the horizontal strip $|\Im(t)| < 500$, in which it satisfies $h(t) \ll (1 + |t|)^{-502}$ and has zeroes at $\pm (n + 1/2) i$ for nonnegative integers $n < 500$, while $h^{\hol}(k) \equiv 0$. We will later make the choice
\[h(t) = h_T(t) \defeq e^{-\frac{t^2}{T^2}} \prod_{j = 1}^{N} \left(\frac{t^2 + \left(j - \frac{1}{2}\right)^2}{T^2}\right)^2\]
for some fixed large integer $N \geq 500$ and $T > 0$; suffice it to say, one may read the rest of this section with this test function in mind. 

\begin{proposition}
\label{spectralreciprocityprop}
For an admissible function $\hf$, we have the identity
\begin{equation}
\label{spectralreciprocityeq}
\MM^{-}(\hf) = \NN(\hf) + \sum_{\pm} \MM^{\pm} \left(\Ts_{t_g}^{\pm} \hf\right),
\end{equation}
where
\begin{align}
\notag
\NN(\hf) & \defeq \frac{6}{\pi^2} L(1,\chi_D)^2 L(1, g_{\psi^2})^2 \frac{D^2}{\nu(D)} \Ns h,	\\
\label{Tshfeq}
\Ts_{t_g}^{+} \hf & \defeq \left(\Ls^{+} H_{t_g}^{+},\Ls^{\hol} H_{t_g}^{+}\right), \qquad \Ts_{t_g}^{-} \hf \defeq \left(\Ls^{-} H_{t_g}^{-},0\right),	\\
\label{Htgeq}
H_{t_g}^{\pm}(x) & \defeq \frac{2}{\pi i} \int_{\sigma_1 - i\infty}^{\sigma_1 + i\infty} \widehat{\Ks^{-} h}(s) \GG_{t_g}^{\pm}(1 - s) x^s \, ds, \quad -3 < \sigma_1 < 1,	\\
\label{GGtgeq}
\GG_{t_g}^{\pm}(s) & \defeq \widehat{\JJ_0^{+}}(s) \widehat{\JJ_{2t_g}^{\mp}}(s) + \widehat{\JJ_0^{-}}(s) \widehat{\JJ_{2t_g}^{\pm}}(s).
\end{align}
\end{proposition}

Here $\Ls^{\pm}$ and $\Ls^{\hol}$ are as in \eqref{Lseq}, $\Ns$ and $\Ks^{-}$ as in \eqref{NsKspmeq}, and $\JJ_r^{\pm}$ as in \eqref{Jrpmeq}. The proof of \hyperref[spectralreciprocityprop]{Proposition \ref*{spectralreciprocityprop}}, which we give at the end of this section, is via the triad of Kuznetsov, Vorono\u{\i}, and Kloosterman summation formul\ae{}. Following the work of Blomer, Li, and Miller \cite{BLM19} and Blomer and the second author \cite{BlK19a,BlK19b}, we avoid using approximate functional equations but instead use Dirichlet series in regions of absolute convergence to obtain an identity akin to \eqref{spectralreciprocityeq}, and then extend this identity holomorphically to give the desired identity.

\begin{remark}
This approach obviates the need for complicated stationary phase estimates and any utilisation of the spectral decomposition of shifted convolution sums, which is the (rather technically demanding) approach taken by Jutila and Motohashi \cite[Theorem 2]{JM05} in obtaining the bound
\begin{multline*}
\sum_{T \leq t_f \leq 2T} \frac{L\left(\frac{1}{2},f\right)^2 \left|L\left(\frac{1}{2} + 2it_g,f\right)\right|^2}{L(1,\sym^2 f)}	\\
+ \frac{1}{2\pi} \int\limits_{T \leq |t| \leq 2T} \left|\frac{\zeta\left(\frac{1}{2} + it\right)^2 \zeta\left(\frac{1}{2} + i(2t_g + t)\right) \zeta\left(\frac{1}{2} + i(2t_g - t)\right)}{\zeta(1 + 2it)}\right|^2 \, dt \ll_{\e} T^{2 + \e} + t_g^{\frac{4}{3} + \e},
\end{multline*}
which is used in \cite{DK18b,Hum18} in the proofs of \hyperref[Planckthm]{Theorems \ref*{Planckthm}} and \ref{fourthmomentthm} for Eisenstein series. Indeed, the method of proof of spectral reciprocity in \hyperref[spectralreciprocityprop]{Proposition \ref*{spectralreciprocityprop}} could be used to give a simpler proof (and slightly stronger version) of \cite[Theorem 2]{JM05}.
\end{remark}

\begin{remark}
Structurally, \hyperref[spectralreciprocityprop]{Proposition \ref*{spectralreciprocityprop}} is proven in a similar way to \cite[Theorem 1.1]{BuK17a}, where an asymptotic with a power savings is given for a moment of $L$-functions that closely resembles $\MM^{-}(\hf)$; see in particular the sketch of proof in \cite[Section 2]{BuK17a}, which highlights the process of Kuznetsov, Vorono\u{\i}, and Kloosterman summation formul\ae{}. The chief difference is the usage of Dirichlet series in regions of absolute convergence coupled with analytic continuation in place of approximate functional equations.
\end{remark}

We define
\begin{align*}
\MM^{\Maass,\pm}\left(s_1,s_2;h\right) & \defeq \sum_{d_1 d_2 = D} \sum_{f \in \BB_0^{\ast}(\Gamma_0(d_1))} \epsilon_f^{\frac{1 \mp 1}{2}} \Ell_{d_2}(s_1,s_2,f)	\\
& \hspace{3cm} \times \frac{L(s_1,f) L(s_1,f \otimes \chi_D) L(s_2,f \otimes g_{\psi^2})}{L(1,\sym^2 f)} h(t_f),	\\
\MM^{\Eis}\left(s_1,s_2;h\right) & \defeq \frac{1}{2\pi} \int_{-\infty}^{\infty} \Ell_D(s_1,s_2,t)	\\
& \hspace{3cm} \times \prod_{\pm} \frac{\zeta(s_1 \pm it) L(s_1 \pm it, \chi_D) L(s_2 \pm it, g_{\psi^2})}{\zeta(1 \pm 2it)} h(t) \, dt,	\\
\MM^{\hol}\left(s_1,s_2;h^{\hol}\right) & \defeq \sum_{d_1 d_2 = D} \sum_{f \in \BB_{\hol}^{\ast}(\Gamma_0(d_1))} \Ell_{d_2}(s_1,s_2,f)	\\
& \hspace{3cm} \times \frac{L(s_1,f) L(s_1,f \otimes \chi_D) L(s_2,f \otimes g_{\psi^2})}{L(1,\sym^2 f)} h^{\hol}(k_f)
\end{align*}
for $s_1,s_2 \in \C$, where
\begin{align*}
\Ell_{d_2}(s_1,s_2,f) & \defeq \frac{d_2}{\nu(d_2)} \sum_{\ell \mid d_2} L_{\ell}(1,\sym^2 f) \frac{\varphi(\ell)}{\ell^{s_1 + s_2}}	\\
& \hspace{3cm} \times \sum_{v_1 w_1 = \ell} \frac{\nu(v_1)}{v_1} \frac{\mu(w_1) \lambda_f(w_1)}{w_1^{1 - s_1}} \sum_{v_2 w_2 = \ell} \frac{\nu(v_2)}{v_2} \frac{\mu(w_2) \lambda_f(w_2)}{w_2^{1 - s_2}},	\\
\Ell_D(s_1,s_2,t) & \defeq \frac{D}{\nu(D)} \sum_{\ell \mid D} \zeta_{\ell}(1 + 2it) \zeta_{\ell}(1 - 2it) \frac{1}{\ell^{s_1 + s_2 - 1}}	\\
& \hspace{3cm} \times \sum_{v_1 w_1 = \ell} \frac{\nu(v_1)}{v_1} \frac{\mu(w_1) \lambda(w_1,t)}{w_1^{1 - s_1}} \sum_{v_2 w_2 = \ell} \frac{\nu(v_2)}{v_2} \frac{\mu(w_2) \lambda(w_2,t)}{w_2^{1 - s_2}}.
\end{align*}
We additionally set
\[\MM^{\pm}\left(s_1,s_2;\hf\right) \defeq \MM^{\Maass,\pm}\left(s_1,s_2;h\right) + \MM^{\Eis}\left(s_1,s_2;h\right) + \delta_{\pm,+} \MM^{\hol}\left(s_1,s_2;h^{\hol}\right).\]

\begin{lemma}
\label{spectralrecpreaclemma}
For admissible $\hf$ and $5/4 < \Re(s_1), \Re(s_2) < 3/2$, we have that
\[\MM^{-}\left(s_1,s_2;\hf\right) = \NN\left(s_1,s_2;\hf\right) + \sum_{\pm} \MM^{\pm} \left(s_2,s_1;\Ts_{s_1,s_2,t_g}^{\pm} \hf\right),\]
where
\[
\NN\left(s_1,s_2;\hf\right) \defeq \frac{L(1,\chi_D) L(2s_2, \chi_D) L(s_1 + s_2, g_{\psi^2}) L^D(1 - s_1 + s_2,g_{\psi^2})}{\zeta^D(1 + 2s_2)} 2D^{2(1 - s_1)} \widehat{\Ks^{-} h}(2(1 - s_1))\]
and
\begin{equation}
\label{Tss1s2tgpmeq}
\begin{split}
\Ts_{s_1,s_2,t_g}^{+} \hf & \defeq \left(\Ls^{+} H_{s_1,s_2,t_g}^{+},\Ls^{\hol} H_{s_1,s_2,t_g}^{+}\right),	\\
\Ts_{s_1,s_2,t_g}^{-} \hf & \defeq \left(\Ls^{-} H_{s_1,s_2,t_g}^{-},0\right),
\end{split}
\end{equation}
with
\begin{multline}
\label{Hs1s2tgeq}
H_{s_1,s_2,t_g}^{\pm}(x) \defeq \frac{2}{\pi i} \int_{\sigma_1 - i\infty}^{\sigma_1 + i\infty} \widehat{\Ks^{-} h}(s) \left(\widehat{\JJ_0^{+}}(2 - s - 2s_1) \widehat{\JJ_{2t_g}^{\mp}}(2 - s - 2s_2) \right.	\\
\left. + \widehat{\JJ_0^{-}}(2 - s - 2s_1) \widehat{\JJ_{2t_g}^{\pm}}(2 - s - 2s_2)\right) x^{s + 2(s_1 + s_2 - 1)} \, ds,
\end{multline}
where $-3 < \sigma_1 < 2(1 - \max\{\Re(s_1),\Re(s_2)\})$.
\end{lemma}

The proof of this is similar to the proofs of analogous results in \cite{BLM19,BlK19a,BlK19b}; as such, we will be terse at times in justifying various technical steps, especially governing the absolute convergence required for the valid shifting of contours and interchanging of orders of integration and summation, for the details may be found in the aforementioned references.

\begin{proof}
We multiply the opposite sign Kuznetsov formula, \hyperref[Kuznetsovthm]{Theorem \ref*{Kuznetsovthm}}, by
\[\frac{\lambda_{\chi_D,1}(m,0) \lambda_{g_{\psi^2}}(n)}{m^{s_1} n^{s_2}}\]
with $\Re(s_1), \Re(s_2) > 1$ and sum over $m,n \in \N$, with $\lambda_{\chi_D,1}(m,0) = \sum_{ab = m} \chi_D(a)$ as in \eqref{lambdachi1eq}. Via \hyperref[Ramanujanlemma]{Lemmata \ref*{Ramanujanlemma}} and \ref{Kuznetsovlemma}, the Maa\ss{} cusp form and the Eisenstein terms are
\[\frac{\MM^{-}\left(s_1,s_2;\hf\right)}{L(2s_1,\chi_D) L(2s_2,\chi_D)}\]
after making the change of variables $m \mapsto v_1 m$ and $n \mapsto v_2 n$, and noting that $\lambda_{\chi_D,1}(vm,0) = \lambda_{\chi_D,1}(m,0)$ and $\lambda_{g_{\psi^2}}(vn) = \lambda_{g_{\psi^2}}(n)$ whenever $v \mid D$ via \hyperref[dihedralHeckeeigenlemma]{Lemma \ref*{dihedralHeckeeigenlemma}}. Since this is an application of the opposite sign Kuznetsov formula, there is no delta term. Finally, Mellin inversion together with \hyperref[MellinKsextendlemma]{Lemma \ref*{MellinKsextendlemma}} give the identity
\[(\Ks^{-} h)(x) = \frac{1}{2\pi i} \int_{\sigma_0 - i\infty}^{\sigma_0 + i\infty} \widehat{\Ks^{-} h}(s) x^{-s} \, ds\]
for $-3 < \sigma_0 < 3$. Using this, the Kloosterman term is seen to be
\begin{equation}
\label{Kloosttermeq}
\frac{D}{2\pi i} \int_{\sigma_0 - i\infty}^{\sigma_0 + i\infty} \widehat{\Ks^{-} h}(s) \sum_{\substack{c = 1 \\ c \equiv 0 \hspace{-.25cm} \pmod{D}}}^{\infty} \frac{1}{c^{1 - s}} \sum_{d \in (\Z/c\Z)^{\times}} L\left(\frac{s}{2} + s_1, E_{\chi_D,1}, \frac{d}{c}\right) L\left(\frac{s}{2} + s_2, g_{\psi^2}, -\frac{\overline{d}}{c}\right) \, ds,
\end{equation}
with the Vorono\u{\i} $L$-series as in \eqref{VoronoiLseriesdefeq}. This rearrangement is valid for $2 - 2 \min\{\Re(s_1),\Re(s_2)\} < \sigma_0 < -1/2$, for then both Vorono\u{\i} $L$-series converge absolutely, while the Weil bound ensures that the sum over $c$ converges.

Assuming that $\max\{\Re(s_1),\Re(s_2)\} < 3/2$, we may move the contour $\Re(s) = \sigma_0$ to $\Re(s) = \sigma_1$ such that $-3 < \sigma_1 < -2\max\{\Re(s_1),\Re(s_2)\}$; the Phragm\'{e}n--Lindel\"{o}f convexity principle ensures that the ensuing integral converges. The only pole that we encounter along the way is at $s = 2(1 - s_1)$, with the resulting residue being
\begin{equation}
\label{residue}
2D^{3/2} L(1,\chi_D) \widehat{\Ks^{-} h}(2(1 - s_1)) \sum_{\substack{c = 1 \\ c \equiv 0 \hspace{-.25cm} \pmod{D}}}^{\infty} \frac{1}{c^{2s_1}} \sum_{d \in (\Z/c\Z)^{\times}} \chi_D(d) L\left(1 - s_1 + s_2, g_{\psi^2}, -\frac{\overline{d}}{c}\right)
\end{equation}
via \hyperref[Voronoilemma]{Lemma \ref*{Voronoilemma}}. For $\Re(s_2) > \Re(s_1)$, the Vorono\u{\i} $L$-series $L(1 - s_1 + s_2, g_{\psi^2}, -\overline{d}/c)$ may be written as an absolutely convergent Dirichlet series, so that the sum over $c$ and $d$ is equal to
\begin{equation}
\label{cdsum}
\sum_{m = 1}^{\infty} \frac{\lambda_{g_{\psi^2}}(m)}{m^{1 - s_1 + s_2}} \sum_{\substack{c = 1 \\ c \equiv 0 \hspace{-.25cm} \pmod{D}}}^{\infty} \frac{1}{c^{2s_1}} \sum_{d \in (\Z/c\Z)^{\times}} \chi_D(d) e\left(-\frac{m\overline{d}}{c}\right).
\end{equation}
The sum over $d$ is a Gauss sum, which may be reexpressed as a sum over $a \mid (c/D, m)$ via \hyperref[Miyakelemma]{Lemma \ref*{Miyakelemma}}. By making the change of variables $c \mapsto acD$ and $m \mapsto am$, \eqref{cdsum} becomes
\[D^{\frac{1}{2} - 2s_1} \sum_{a = 1}^{\infty} \frac{1}{a^{s_1 + s_2}} \sum_{m = 1}^{\infty} \frac{\lambda_{g_{\psi^2}}(am) \chi_D(m)}{m^{1 - s_1 + s_2}} \sum_{c = 1}^{\infty} \frac{\mu(c) \chi_D(c)}{c^{2s_1}}.\]
Applying M\"{o}bius inversion to \eqref{cuspmult}, we see that
\begin{equation}
\label{invertcuspmult}
\lambda_{g_{\psi^2}}(am) = \sum_{b \mid (a,m)} \mu(b) \chi_D(b) \lambda_{g_{\psi^2}}\left(\frac{a}{b}\right) \lambda_{g_{\psi^2}}\left(\frac{m}{b}\right).
\end{equation}
Making the change of variables $a \mapsto ab$ and $m \mapsto bm$, \eqref{cdsum} is rewritten as
\begin{multline*}
D^{\frac{1}{2} - 2s_1} \sum_{a = 1}^{\infty} \frac{\lambda_{g_{\psi^2}}(a)}{a^{s_1 + s_2}} \sum_{m = 1}^{\infty} \frac{\lambda_{g_{\psi^2}}(m) \chi_D(m)}{m^{1 - s_1 + s_2}} \sum_{c = 1}^{\infty} \frac{\mu(c) \chi_D(c)}{c^{2s_1}} \sum_{\substack{b = 1 \\ (b,D) = 1}}^{\infty} \frac{\mu(b)}{b^{1 + 2s_2}}	\\
= D^{\frac{1}{2} - 2s_1} \frac{L(s_1 + s_2, g_{\psi^2}) L^D(1 - s_1 + s_2,g_{\psi^2})}{\zeta^D(1 + 2s_2) L(2s_1, \chi_D)},
\end{multline*}
recalling that $g_{\psi^2}$ being dihedral means that it is twist-invariant by $\chi_D$. So the residue \eqref{residue} is $\NN(s_1,s_2;\hf) / L(2s_1,\chi_D) L(2s_2,\chi_D)$, at least initially for $\Re(s_2) > \Re(s_1)$, and this is also valid for $5/4 < \Re(s_1),\Re(s_2) < 3/2$, since it is holomorphic in this region.

Now we wish to reexpress \eqref{Kloosttermeq}, where $\sigma_0$ has been replaced by $\sigma_1$, with $-3 < \sigma_1 < -2 \max\{\Re(s_1),\Re(s_2)\}$. We apply the Vorono\u{\i} summation formul\ae{}, \hyperref[Voronoilemma]{Lemma \ref*{Voronoilemma}}, to both Vorono\u{\i} $L$-series. The resulting Vorono\u{\i} $L$-series are absolutely convergent Dirichlet series; opening these up and interchanging the order of summation and integration then leads to the expression
\[\sum_{m = 1}^{\infty} \frac{\lambda_{\chi_D,1}(m,0)}{m^{s_2}} \sum_{n = 1}^{\infty} \frac{\lambda_{g_{\psi^2}}(n)}{n^{s_1}} \sum_{\pm} \OO_D\left(m,\pm n; H_{s_1,s_2,t_g}^{\pm}\right)\]
with $\OO_D$ as in \eqref{OOqeq} and $H_{s_1,s_2,t_g}^{\pm}$ as in \eqref{Hs1s2tgeq}. As the Mellin transform of $\Ks^{-} h$ defines a holomorphic function of $s$ for $-3 < \Re(s) < 3$, while the Mellin transform of $\JJ_r^{\pm}$ has simple poles at $s = 2(\pm ir - n)$ with $n \in \N_0$, the integrand is holomorphic in the strip $-3 < \Re(s) < 2(1 - \max\{\Re(s_1),\Re(s_2)\})$.

Finally, we apply \hyperref[Kloostermanthm]{Theorem \ref*{Kloostermanthm}}, the Kloosterman summation formula, in order to express this sum of Kloosterman sums in terms of Fourier coefficients of automorphic forms; the admissibility of $\hf$ ensures that $H_{s_1,s_2,t_g}^{\pm}$ satisfies the requisite conditions for this formula to be valid. We then interchange the order of summation and once again use \hyperref[Ramanujanlemma]{Lemma \ref*{Ramanujanlemma}} and \hyperref[Kuznetsovlemma]{Lemma \ref*{Kuznetsovlemma}}, making the change of variables $m \mapsto v_1 m$ and $n \mapsto v_2 n$. In this way, we arrive at
\[\sum_{\pm} \frac{\MM^{\pm} \left(s_2,s_1;\Ts_{s_1,s_2,t_g}^{\pm} \hf\right)}{L(2s_1,\chi_D) L(2s_2,\chi_D)}.\]
The proof is complete upon multiplying both sides by $L(2s_1,\chi_D) L(2s_2,\chi_D)$.
\end{proof}

\begin{proof}[Proof of {\hyperref[spectralreciprocityprop]{Proposition \ref*{spectralreciprocityprop}}}]
This follows the same method as \cite[Proof of Theorem 1]{BLM19}, \cite[Proof of Theorem 1]{BlK19b}, and \cite[Proof of Theorem 2]{BlK19a}: it is shown in \cite[Section 10]{BlK19b} that for $1/2 \leq \Re(s_1), \Re(s_2) \leq 3/2$, $\Ts_{s_1,s_2,t_g}^{\pm} \hf$ is weakly admissible in the sense of \cite[(1.3)]{BlK19b}, which implies that $\NN(s_1,s_2;\hf)$ and $\MM^{\pm} (s_2,s_1;\Ts_{s_1,s_2,t_g}^{\pm} \hf)$ extend meromorphically to this region. Moreover, we have the identity $\MM^{\pm}(1/2,1/2;\hf) = \MM^{\pm}(\hf)$, since
\[\Ell_{d_2}(f) = 2^{\omega(d_2)} \frac{\varphi(d_2)}{d_2} \frac{L_{d_2}(1,\sym^2 f)}{L_{d_2}\left(\frac{1}{2},f\right)}, \qquad \Ell_D(t) = 2^{\omega(D)} \frac{\zeta_D(1 + 2it) \zeta_D(1 - 2it)}{\zeta_D\left(\frac{1}{2} + it\right) \zeta_D\left(\frac{1}{2} - it\right)}\]
via \hyperref[EllMaassidentitylemma]{Lemmata \ref*{EllMaassidentitylemma}} and \ref{EllEisidentitylemma}, while $\MM^{\Maass,\pm}(1/2,1/2,h)$ is equal to $\MM^{\Maass}(h)$ as $L(1/2,f \otimes \chi_D) = L(1/2,f) L(1/2,f \otimes g_{\psi^2}) = 0$ when $\epsilon_f = -1$.

This process of meromorphic continuation is straightforward for the terms $\MM^{\Maass,-}(s_1,s_2;h)$, $\MM^{\Maass,\pm} (s_2,s_1;\Ls^{\pm} H_{s_1,s_2,t_g}^{\pm})$, $\MM^{\hol} (s_2,s_1;\Ls^{\hol} H_{s_1,s_2,t_g}^{+})$, and $\NN(s_1,s_2;\hf)$, but for $\MM^{\Eis}(s_1,s_2;h)$ and $\MM^{\Eis} (s_2,s_1;\Ls^{\pm} H_{s_1,s_2,t_g}^{\pm})$, additional polar divisors arise via shifting the contour in the integration over $t$; see, for example, \cite[Lemma 16]{BlK19b} and \cite[Lemma 3]{BlK19a}. In this way, the additional terms
\begin{align*}
\RR\left(s_1,s_2;\hf\right) & \defeq_{\pm_1} \sum \Res_{t = \pm_1 i (1 - s_1)} (\pm_1 i) h(t) \Ell_D(s_1,s_2,t)	\\
& \hspace{3cm} \times \prod_{\pm_2} \frac{\zeta(s_1 \pm_2 it) L(s_1 \pm_2 it, \chi_D) L(s_2 \pm_2 it, g_{\psi^2})}{\zeta(1 \pm_2 2it)},	\\
\RR^{\pm}\left(s_2,s_1;\Ts_{s_1,s_2,t_g}^{\pm} \hf\right) & \defeq \sum_{\pm_1} \Res_{t = \pm_1 i (1 - s_2)} (\pm_1 i) \left(\Ls^{\pm} H_{s_1,s_2,t_g}^{\pm}\right)(t) \Ell_D(s_2,s_1,t)	\\
& \hspace{3cm} \times \prod_{\pm_2} \frac{\zeta(s_2 \pm_2 it) L(s_2 \pm_2 it, \chi_D) L(s_1 \pm_2 it, g_{\psi^2})}{\zeta(1 \pm_2 2it)}
\end{align*}
arise when $\Re(s_1), \Re(s_2) < 1$. But these vanish when $s_1 = s_2 = 1/2$ since $\chi_D$ is even and so $L(s,\chi_D)$ has a trivial zero at $s = 0$.
\end{proof}

\section{Bounds for the Transform for the Short Initial Range}
\label{boundstranssect}

We take $\hf = (h,0)$ in \hyperref[spectralreciprocityprop]{Proposition \ref*{spectralreciprocityprop}} to be
\begin{equation}
\label{hTeq}
h(t) = h_T(t) \defeq e^{-\frac{t^2}{T^2}} P_T(t), \qquad P_T(t) \defeq \prod_{j = 1}^{N} \left(\frac{t^2 + \left(j - \frac{1}{2}\right)^2}{T^2}\right)^2
\end{equation}
for some fixed large integer $N \geq 500$ and $T > 0$, which is positive on $\R \cup i(-1/2,1/2)$ and bounded from below by a constant for $t \in [-2T,-T] \cup [T,2T]$. We wish to determine the asymptotic behaviour of the functions $(\Ls^{\pm} H_{T,t_g}^{\pm})(t)$ and $(\Ls^{\hol} H_{T,t_g}^{+})(k)$ with uniformity in all variables $T$, $t_g$, and $t$ or $k$, where $H_{t_g}^{\pm} = H_{T,t_g}^{\pm}$ is as in \eqref{Htgeq}. Were we to consider $t_g$ as being fixed, then such asymptotic behaviour has been studied by Blomer, Li, and Miller \cite[Lemma 3]{BLM19}. As we are interested in the behaviour of $\Ts_{t_g}^{\pm} \hf$ as $t_g$ tends to infinity, a little additional work is required.

\begin{lemma}
\label{Lspmintegrandboundslemma}
Define
\begin{align*}
\Omega^{+}(\tau,t,t_g) & \defeq \begin{dcases*}
2|t| - |\tau| & if $|\tau| \leq \min\{2|t|,4t_g\}$,	\\
0 & if $2|t| \leq |\tau| \leq 4t_g$,	\\
2|t| - 4t_g & if $4t_g \leq |\tau| \leq 2|t|$,	\\
|\tau| - 4t_g & if $|\tau| \geq \max\{2|t|,4t_g\}$,
\end{dcases*}	\\
\Omega^{-}(\tau,t,t_g) & \defeq \begin{dcases*}
|\tau| & if $|\tau| \leq \min\{2|t|,2t_g\}$,	\\
2|\tau| - 2|t| & if $2|t| \leq |\tau| \leq 2t_g$,	\\
4t_g - |\tau| & if $2t_g \leq |\tau| \leq \min\{2|t|,4t_g\}$,	\\
4t_g - 2|t| & if $\max\{2|t|,2t_g\} \leq |\tau| \leq 4t_g$,	\\
0 & if $4t_g \leq |\tau| \leq 2|t|$,	\\
|\tau| - 2|t| & if $|\tau| \geq \max\{2|t|,4t_g\}$,
\end{dcases*}	\\
\Omega^{\hol}(\tau,k,t_g) & \defeq \begin{dcases*}
0 & if $|\tau| \leq 4t_g$,	\\
|\tau| - 4t_g & if $|\tau| \geq 4t_g$.
\end{dcases*}
\end{align*}
For $s = \sigma + i\tau$ with $-N/2 < \sigma < 1$, provided that additionally $s$ is at least a bounded distance away from $\{2(\pm it - n) : n \in \N_0\}$, and for $t \in \R \cup i(-1/2,1/2)$ we have that
\begin{multline*}
\widehat{\Ks^{-} h_T}(s) \widehat{\JJ_t^{\pm}}(s) \GG_{t_g}^{\pm}(1 - s) \ll_{\sigma} T^{1 + \sigma} (1 + |\tau|)^{-N - \sigma} \left(\left(1 + \left|\tau + 4t_g\right|\right) \left(1 + \left|\tau - 4t_g\right|\right)\right)^{-\frac{\sigma}{2}}	\\
\times \left(\left(1 + \left|\tau + 2t\right|\right) \left(1 + \left|\tau - 2t\right|\right)\right)^{\frac{1}{2} (\sigma - 1)} e^{-\frac{\pi}{2} \Omega^{\pm}(\tau,t,t_g)},
\end{multline*}
and for $t \in \R$,
\begin{multline*}
\Res_{s = 2(\pm it - n)} \widehat{\Ks^{-} h_T}(s) \widehat{\JJ_t^{-}}(s) \GG_{t_g}^{-}(1 - s)	\\
\ll_n T^{1 - 2n} (1 + |t|)^{-N + n - \frac{1}{2}} (1 + |t + 4t_g|) (1 + |t - 4t_g|))^n e^{-\frac{\pi}{2} \Omega^{-}(t,t,t_g)}.
\end{multline*}
For $s = \sigma + i\tau$ with $-N/2 < \sigma < 1$, provided that additionally $s$ is at least a bounded distance away from $\{2(\pm it - n) : n \in \N_0\}$, and for $k \in 2\N$, we have that
\begin{multline*}
\widehat{\Ks^{-} h_T}(s) \widehat{\JJ_k^{\hol}}(s) \GG_{t_g}^{+}(1 - s) \ll_{\sigma} T^{1 + \sigma} (1 + |\tau|)^{-N - \sigma} \left(\left(1 + \left|\tau + 4t_g\right|\right) \left(1 + \left|\tau - 4t_g\right|\right)\right)^{-\frac{\sigma}{2}}	\\
\times \left(k + |\tau|\right)^{\sigma - 1} e^{-\frac{\pi}{2} \Omega^{\hol}(\tau,k,t_g)},
\end{multline*}
and
\[\Res_{s = 1 - k - 2n} \widehat{\Ks^{-} h_T}(s) \widehat{\JJ_k^{\hol}}(s) \GG_{t_g}^{+}(1 - s) \ll_n T^{2 - k - 2n} t_g^{k - 1 + 2n} \left(\frac{k - 1}{2\pi e}\right)^{1 - k} k^{-1/2}.\]
\end{lemma}

\begin{proof}
From \cite[Lemma 4]{BLM19}, we have the bound
\[x^j \frac{d^j}{dx^j} (\Ks^{-} h_T)(x) \ll_j T \min\left\{\left(\frac{x}{T}\right)^{N/2}, \left(\frac{x}{T}\right)^{-N/2}\right\}\]
for $j \in \{0,\ldots,N\}$, and consequently the Mellin transform of $\Ks^{-} h_T$ is holomorphic in the strip $-N/2 < \Re(s) < N/2$, in which it satisfies the bounds
\[\widehat{\Ks^{-} h_T}(s) \ll_{\sigma} T^{1 + \sigma} (1 + |\tau|)^{-N}\]
for $s = \sigma + i\tau$. Next, we use \hyperref[JJrMellinasympcor]{Corollary \ref*{JJrMellinasympcor}} to bound $\widehat{\JJ_k^{\hol}}(s)$ and $\widehat{\JJ_t^{\pm}}(s)$, as well as bound the residues at $s = 1 - k - 2n$ and $s = 2(\pm it - n)$ respectively, where $n \in \N_0$. Finally, Stirling's formula \eqref{Stirlingeq} shows that
\[\GG_{t_g}^{+}(1 - s) \ll_{\sigma} \left(1 + |\tau|\right)^{-\sigma} \left(\left(1 + \left|\tau + 4t_g\right|\right) \left(1 + \left|\tau - 4t_g\right|\right)\right)^{-\frac{\sigma}{2}} \times \begin{dcases*}
1 & if $|\tau| \leq 4t_g$,	\\
e^{-\frac{\pi}{2} (|\tau| - 4t_g)} & if $|\tau| \geq 4t_g$
\end{dcases*}\]
for $s = \sigma + i\tau$ with $\sigma < 1$, and similarly
\[\GG_{t_g}^{-}(1 - s) \ll_{\sigma} \left(1 + |\tau|\right)^{-\sigma} \left(\left(1 + \left|\tau + 4t_g\right|\right) \left(1 + \left|\tau - 4t_g\right|\right)\right)^{-\frac{\sigma}{2}} \times \begin{dcases*}
e^{-\frac{\pi}{2} |\tau|} & if $|\tau| \leq 2t_g$,	\\
e^{-\frac{\pi}{2}(4t_g - |\tau|)} & if $2t_g \leq |\tau| \leq 4t_g$,	\\
1 & if $|\tau| \geq 4t_g$.
\end{dcases*}\]
Combining these bounds yields the result.
\end{proof}

\begin{corollary}
\label{LHboundscor}
For fixed $-N/2 < \sigma < 1$, $t_g^{1/2} \ll T \ll t_g$, $t \in \R \cup i(-1/2,1/2)$, and $k \in 2\N$, we have that
\begin{align*}
\left(\Ls^{+} H_{T,t_g}^{+}\right)(t) & \ll_{\sigma} T \left(\frac{t_g}{T}\right)^{-\sigma} (1 + |t|)^{-N + \frac{1}{2}},	\\
\left(\Ls^{-} H_{T,t_g}^{-}\right)(t) & \ll_{\sigma} T \left(\frac{t_g}{(1 + |t|)T}\right)^{-\sigma} (1 + |t|)^{-1},	\\
\left(\Ls^{\hol} H_{T,t_g}^{+}\right)(k) & \ll_{\sigma} T \left(\frac{t_g}{kT}\right)^{-\sigma} k^{-1}.
\end{align*}
\end{corollary}

\begin{proof}
By Mellin inversion,
\begin{align*}
\left(\Ls^{\pm} H_{T,t_g}^{\pm}\right)(t) & = \frac{2}{\pi i} \int_{\sigma_1 - i\infty}^{\sigma_1 + i\infty} \widehat{\Ks^{-} h_T}(s) \widehat{\JJ_t^{\pm}}(s) \GG_{t_g}^{\pm}(1 - s) \, ds,	\\
\left(\Ls^{\hol} H_{T,t_g}^{+}\right)(k) & = \frac{2}{\pi i} \int_{\sigma_1 - i\infty}^{\sigma_1 + i\infty} \widehat{\Ks^{-} h_T}(s) \widehat{\JJ_k^{\hol}}(s) \GG_{t_g}^{+}(1 - s) \, ds
\end{align*}
for any $0 < \sigma_1 < 1$. We break each of these integrals over $s = \sigma_1 + i\tau$ into different ranges of $\tau$ depending on the size of $|t|$ or $k$ relative to $t_g$ and use the bounds for the integrands obtained in \hyperref[Lspmintegrandboundslemma]{Lemma \ref*{Lspmintegrandboundslemma}} to bound each portion of the integrals. In most regimes, we have exponential decay of the integrands due to the presence of $e^{-\frac{\pi}{2} \Omega^{\pm}(\tau,t,t_g)}$ or $e^{-\frac{\pi}{2} \Omega^{\hol}(\tau,k,t_g)}$; it is predominantly the regimes for which $\Omega^{\pm}(\tau,t,t_g)$ or $\Omega^{\hol}(\tau,k,t_g)$ are zero that have nonnegligible contributions.

For $(\Ls^{+} H_{T,t_g}^{+})(t)$, this is straightforward, noting that we can assume without loss of generality in this case that $0 < \sigma < 1$ with $\sigma_1 = \sigma$; the dominant contribution comes from the section of the integral with $2|t| \leq |\tau| \leq 4t_g$, as this is the regime for which $\Omega^+(\tau,t,t_g)$ is equal to zero.

Similarly, for $(\Ls^{-} H_{T,t_g}^{-})(t)$, we may assume that $0 \leq \sigma < 1$ with $\sigma_1 = \sigma$ for $1 + |t| \leq t_g T^{-1}$. For $1 + |t| > t_g T^{-1}$, we may assume that $-N/2 < \sigma \leq 0$: we shift the contour from $\Re(s) = \sigma_1$ to $\Re(s) = \sigma$, picking up residues at the poles at $s = 2(\pm it - n)$ for $0 \leq n < N/4$, with the dominant contribution in both cases being from the section of the integral with $|\tau|$ bounded (the remaining regimes involve exponential decay from the presence of $e^{-\frac{\pi}{2} \Omega^{-}(\tau,t,t_g)}$ unless $4t_g \leq |\tau| \leq 2|t|$, in which case $(1 + |\tau|)^{-N - \sigma}$ contributes significant polynomial decay).

Finally, we may again assume without loss of generality for $(\Ls^{\hol} H_{T,t_g}^{+})(k)$ that $0 \leq \sigma < 1$ for $k \leq t_g T^{-1}$ and $-N/2 < \sigma \leq 0$ for $k > t_g T^{-1}$, since we may shift the contour with impunity in this vertical strip; once again, the dominant contribution comes from the section of the integral with $|\tau|$ bounded due to the polynomial decay of $(1 + |\tau|)^{-N - \sigma}$.
\end{proof}

\section{Proof of \texorpdfstring{\hyperref[dihedralmomentsprop]{Proposition \ref*{dihedralmomentsprop} (1)}}{Proposition \ref{dihedralmomentsprop} (1)}: the Short Initial Range}
\label{(1)sect}

\begin{proof}[Proof of {\hyperref[dihedralmomentsprop]{Proposition \ref*{dihedralmomentsprop} (1)}}]
For $T < t_g^{\delta/2(1 + A)}$, where $\delta, A > 0$ are absolute constants arising from \hyperref[MVsubconvthm]{Theorem \ref*{MVsubconvthm}}, we use the subconvex bounds in \hyperref[MVsubconvthm]{Theorem \ref*{MVsubconvthm}} to bound the terms $L(1/2, f \otimes g_{\psi^2})$ and $|L(1/2 + it,g_{\psi^2})|$ by $O(T^A t_g^{1 - \delta})$, so that for $h(t) = 1_{E \cup -E}(t)$ with $E = [T,2T]$,
\begin{multline*}
\MM^{\Maass}(h) + \MM^{\Eis}(h) \ll T^A t_g^{1 - \delta} \sum_{d_1 d_2 = D} 2^{\omega(d_2)} \frac{\varphi(d_2)}{d_2} \sum_{\substack{f \in \BB_0^{\ast}(\Gamma_0(d_1)) \\ T \leq t_f \leq 2T}} \frac{L^{d_2}\left(\frac{1}{2},f\right) L\left(\frac{1}{2},f \otimes \chi_D\right)}{L^{d_2}(1,\sym^2 f)}	\\
+ T^A t_g^{1 - \delta} \frac{2^{\omega(D)}}{2\pi} \int\limits_{T \leq |t| \leq 2T} \left|\frac{\zeta^D\left(\frac{1}{2} + it\right) L\left(\frac{1}{2} + it, \chi_D\right)}{\zeta^D(1 + 2it)}\right|^2 \, dt.
\end{multline*}
We then use the Cauchy--Schwarz inequality, the approximate functional equation, \hyperref[approxfunclemma]{Lemma \ref*{approxfunclemma}}, and the large sieve, \hyperref[largesievethm]{Theorem \ref*{largesievethm}}, to bound the remaining moments of $L(1/2,f) L(1/2,f \otimes \chi_D)$ and of $|\zeta(1/2 + it) L(1/2 + it,\chi_D)|^2$ by $O_{\e}(T^{2 + \e})$, and so in this range,
\[\MM^{\Maass}(h) + \MM^{\Eis}(h) \ll_{\e} T^{2 + A + \e} t_g^{1 - \delta} \ll T t_g^{1 - \frac{\delta}{3}}.\]

For $t_g^{\delta/2(1 + A)} \leq T < t_g^{1/2}$, the subconvex bounds in \hyperref[Youngsubconvthm]{Theorems \ref*{Youngsubconvthm}} and \ref{MVsubconvthm} are used to bound the terms $2^{\omega(d_1)} L(1/2,f) L(1/2,f \otimes \chi_D)$ and $2^{\omega(D)} |\zeta(1/2 + it) L(1/2 + it, \chi_D)|^2$ by $O_{\e}(T^{5/6 + \e})$, so that
\[\MM^{\Maass}(h) + \MM^{\Eis}(h) \ll_{\e} T^{\frac{5}{6} + \e} \left(\widetilde{\MM}^{\Maass}(h_T) + \widetilde{\MM}^{\Eis}(h_T)\right)\]
with $h_T$ as in \eqref{hTfirstmomenteq}. \hyperref[firstmomentprop]{Proposition \ref*{firstmomentprop} (1)} then bounds $\widetilde{\MM}^{\Maass}(h_T) + \widetilde{\MM}^{\Eis}(h_T)$ by $O_{\e}(t_g^{1 + \e})$. So in this range,
\[\MM^{\Maass}(h) + \MM^{\Eis}(h) \ll_{\e} T^{\frac{5}{6} + \e} t_g^{1 + \e} \ll_{\e} T t_g^{1 - \frac{\delta}{12(1 + A)} + \e}.\]

For $t_g^{1/2} \leq T \ll t_g^{1 - \alpha}$, \hyperref[spectralreciprocityprop]{Proposition \ref*{spectralreciprocityprop}} implies that
\[\MM^{\Maass}(h) + \MM^{\Eis}(h) \ll \NN(\hf) + \sum_{\pm} \MM^{\pm} \left(\Ts_{t_g}^{\pm} \hf\right),\]
where $\hf = (h_T,0)$ with $h_T$ as in \eqref{hTeq}. Noting that $\NN(\hf) \ll_{\e} T^{2 + \e}$, \hyperref[LHboundscor]{Corollary \ref*{LHboundscor}} then shows that $\MM^{\pm} (\Ts_{t_g}^{\pm} \hf)$ are both $O(Tt_g^{1 - \delta})$ via the Cauchy--Schwarz inequality together with the approximate functional equation and the large sieve, except in a select few ranges, namely the range $t_f \asymp t_g/T$ in the term $\MM^{\Maass,-}(\Ls^{-} H_{T,t_g}^{-})$, the range $|t| \asymp t_g/T$ in $\MM^{\Eis}(\Ls^{-} H_{T,t_g}^{-})$, and the range $k_f \asymp t_g/T$ in $\MM^{\hol}(\Ls^{\hol} H_{T,t_g}^{+})$. The former two terms are then treated as we have just done for $T < t_g^{\delta/2(1 + A)}$ and for $t_g^{\delta/2(1 + A)} \leq T < t_g^{1/2}$, and the latter is treated via the same method, recalling that \hyperref[firstmomentprop]{Proposition \ref*{firstmomentprop} (2)} entails such bounds for holomorphic cusp forms.
\end{proof}

\begin{remark}
For the treatment of the range $t_g^{\delta/2(1 + A)} \leq T < t_g^{1/2}$, we in fact have the bound $O_{\e}(T^{2/3 + \e})$ for $2^{\omega(d_1)} L(1/2,f) L(1/2,f \otimes \chi_D)$ and $2^{\omega(D)} |\zeta(1/2 + it) L(1/2 + it, \chi_D)|^2$; see \hyperref[Weylsubconvexityremark]{Remark \ref*{Weylsubconvexityremark}}. In the treatment of the range $t_g^{1/2} \leq T \ll t_g^{1 - \alpha}$, we use spectral reciprocity and subsequently require subconvex bounds for $2^{\omega(d_1)} L(1/2,f) L(1/2,f \otimes \chi_D)$ with $f$ a holomorphic newform of level $d_1 \mid D$ and weight $k_f \asymp t_g/T$. Here we do \emph{not} know of such strong bounds if $d_1 > 1$: while the bound $L(1/2,f \otimes \chi_D) \ll_{\e} k_f^{1/3 + \e}$ is known \cite[Theorem 1.1]{You17}, and of course $L(1/2,f) \ll_{\e} k_f^{1/2 + \e}$ is merely the convexity bound, the bound $L(1/2,f) \ll_{\e} k_f^{1/3 + \e}$ is only known for $d_1 = 1$ \cite[Theorem 3.1.1]{Pen01}, and a modification of the proof of this bound to allow $d_1 > 1$ seems to be reasonably nontrivial.
\end{remark}

\section{Proof of \texorpdfstring{\hyperref[dihedralmomentsprop]{Proposition \ref*{dihedralmomentsprop} (2)}}{Proposition \ref{dihedralmomentsprop} (2)}: the Bulk Range}
\label{(2)sect}

The proof that we give of \hyperref[dihedralmomentsprop]{Proposition \ref*{dihedralmomentsprop} (2)} follows the approach of \cite{DK18b}, where an asymptotic formula is obtained for a similar expression pertaining instead to the regularised fourth moment of an Eisenstein series. As such, we shall be extremely brief, detailing only the minor ways in which our proof differs from that of \cite{DK18b}.

\subsection{An Application of the Kuznetsov Formula}

Following \cite[Section 2.1]{DK18b}, it suffices to obtain asymptotic formul\ae{} for
\[\MM^{\Maass}(h) + \MM^{\Eis}(h)\]
as in \eqref{MMMaasseq} and \eqref{MMEiseq} with
\begin{equation}
\label{hforbulkeq}
h(t) = \frac{\pi W(t) H(t)}{8D^2 L(1,\chi_D)^2 L(1,g_{\psi^2})^2},
\end{equation}
analogously to \cite[(2.2)]{DK18b}, where $H(t)$ is as in \eqref{H(t)defeq} and $W(t) = W_{\alpha}(t)$ is a certain weight function given in \cite[Lemma 5.1]{BuK17b} that localises $h(t)$ to the range $[-2t_g + t_g^{1 - \alpha}, -t_g^{1 - \alpha}] \cup [t_g^{1 - \alpha},2t_g - t_g^{1 - \alpha}]$. We may artificially insert the parity $\epsilon_f$ into the spectral sum $\MM^{\Maass}(h)$ since $L(1/2,f \otimes \chi_D) = L(1/2,f) L(1/2,f \otimes g_{\psi^2}) = 0$ when $\epsilon_f = -1$; this allows us to use the opposite sign Kuznetsov formula, which greatly simplifies future calculations.

Akin to the proof of \hyperref[MMtildeidentitylemma]{Lemma \ref*{MMtildeidentitylemma}}, we make use of the Kuznetsov formula associated to the pair of cusps $(\aa,\bb)$ with $\aa \sim \infty$ and $\bb \sim 1$, which once again naturally introduces the root numbers of $f \boxplus f \otimes \chi_D$ and of $f \otimes g_{\psi^2}$ in such a way to give approximate functional equations of the correct length for each level dividing $D$.

\begin{lemma}
\label{bulkKuzlemma}
With $h$ as in \eqref{hforbulkeq}, we have that
\begin{multline}
\label{bulkKuzeq}
\MM^{\Maass}(h) + \MM^{\Eis}(h)	\\
= \frac{\pi}{4D L(1,\chi_D)^2 L(1,g_{\psi^2})^2} \sum_{n,m,k,\ell = 1}^{\infty} \frac{\lambda_{\chi_D,1}(n,0) \lambda_{g_{\psi^2}}(m) \chi_D(k\ell)}{\sqrt{mn} k\ell}	\\
\times \left(\sum_{\substack{c = 1 \\ c \equiv 0 \hspace{-.25cm} \pmod{D}}}^{\infty} \frac{S(m,-n;c)}{c} \int_{-\infty}^{\infty} \JJ_r^{-}\left(\frac{\sqrt{mn}}{c}\right) V_1^1\left(\frac{nk^2}{D^{3/2}},r\right) V_2^1\left(\frac{m\ell^2}{D^{3/2}},r\right) W(r) H(r) \, d_{\spec}r \right.	\\
\left. + \sum_{\substack{c = 1 \\ (c,D) = 1}}^{\infty} \frac{S\left(m,-n\overline{D};c\right)}{c\sqrt{D}} \int_{-\infty}^{\infty} \JJ_r^{-}\left(\frac{\sqrt{mn}}{c\sqrt{D}}\right) V_1^1\left(\frac{nk^2}{D^{3/2}},r\right) V_2^1\left(\frac{m\ell^2}{D^{3/2}},r\right) W(r) H(r) \, d_{\spec}r\right)	\\
+ O_{\e}(t_g^{-1 + \e}).
\end{multline}
\end{lemma}

Here $V_1^1$ and $V_2^1$ are as in \eqref{V1eq} and \eqref{V2eq}.

\begin{proof}
We use the opposite sign Kuznetsov formula associated to the $(\infty,\infty)$ pair of cusps, \hyperref[Kuznetsovthm]{Theorem \ref*{Kuznetsovthm}}, with
\[h(t) = \frac{\pi}{8D^2 L(1,\chi_D)^2 L(1,g_{\psi^2})^2} V_1^1\left(\frac{nk^2}{D^{3/2}},t\right) V_2^1\left(\frac{m\ell^2}{D^{3/2}},t\right) W(t) H(t),\]
noting that this requires Yoshida's extension of the Kuznetsov formula \cite[Theorem]{Yos97}, since $H(t)$ has poles at $t = \pm_1 2t_g \pm_2 i/2$. We subsequently multiply through by
\[\frac{\lambda_{\chi_D,1}(n,0) \lambda_{g_{\psi^2}}(m) \chi_D(k\ell)}{\sqrt{mn} k\ell}\]
and sum over $n,m,k,\ell \in \N$. Via the explicit expression in \hyperref[Kuznetsovlemma]{Lemma \ref*{Kuznetsovlemma}}, the Maa\ss{} cusp form term is
\begin{multline*}
\frac{\pi}{8D^2 L(1,\chi_D)^2 L(1,g_{\psi^2})^2} \sum_{d_1 d_2 = D} \frac{d_2}{\nu(d_2)} \sum_{f \in \BB_0^{\ast}(\Gamma_0(d_1))} \epsilon_f \frac{W(t_f) H(t_f)}{L(1,\sym^2 f)} \sum_{\ell \mid d_2} L_{\ell}(1,\sym^2 f) \frac{\varphi(\ell)}{\ell}	\\
\times \sum_{v_1 w_1 = \ell} \frac{\nu(v_1)}{v_1} \frac{\mu(w_1) \lambda_f(w_1)}{\sqrt{w_1}} \sum_{m = 1}^{\infty} \sum_{\ell = 1}^{\infty} \frac{\lambda_f(m) \lambda_{g_{\psi^2}}(m) \chi_D(\ell)}{\sqrt{m} \ell} V_2^1\left(\frac{v_1 m\ell^2}{D^{3/2}},t_f\right)	\\
\times \sum_{v_2 w_2 = \ell} \frac{\nu(v_2)}{v_2} \frac{\mu(w_2) \lambda_f(w_2)}{\sqrt{w_2}} \sum_{n = 1}^{\infty} \sum_{k = 1}^{\infty} \frac{\lambda_f(n) \lambda_{\chi_D,1}(n,0) \chi_D(k)}{\sqrt{n} k} V_1^1\left(\frac{v_2 nk^2}{D^{3/2}},t_f\right)
\end{multline*}
after making the change of variables $m \mapsto v_1 m$ and $n \mapsto v_2 n$.

We do the same with the opposite sign Kuznetsov formula associated to the $(\infty,1)$ pair of cusps, \hyperref[infty1Kuznetsovthm]{Theorem \ref*{infty1Kuznetsovthm}}, for which the resulting Maa\ss{} cusp form term is
\begin{multline*}
\frac{\pi}{8D^2 L(1,\chi_D)^2 L(1,g_{\psi^2})^2} \sum_{d_1 d_2 = D} \frac{d_2}{\nu(d_2)} \sum_{f \in \BB_0^{\ast}(\Gamma_0(d_1))} \epsilon_f \frac{W(t_f) H(t_f)}{L(1,\sym^2 f)} \sum_{\ell \mid d_2} L_{\ell}(1,\sym^2 f) \frac{\varphi(\ell)}{\ell}	\\
\times \sum_{v_1 w_1 = \ell} \frac{\nu(v_1)}{v_1} \frac{\mu(w_1) \lambda_f(w_1)}{\sqrt{w_1}} \eta_f(d_1) \sum_{n = 1}^{\infty} \sum_{\ell = 1}^{\infty} \frac{\lambda_f(n) \lambda_{g_{\psi^2}}(n) \chi_D(\ell)}{\sqrt{n} \ell} V_2^1\left(\frac{d_2 n\ell^2}{v_1 D^{3/2}},t_f\right)	\\
\times \sum_{v_2 w_2 = \ell} \frac{\nu(v_2)}{v_2} \frac{\mu(w_2) \lambda_f(w_2)}{\sqrt{w_2}} \sum_{m = 1}^{\infty} \sum_{k = 1}^{\infty} \frac{\lambda_f(m) \lambda_{\chi_D,1}(m,0) \chi_D(k)}{\sqrt{m} k} V_1^1\left(\frac{v_2 m k^2}{D^{3/2}},t_f\right)
\end{multline*}
via the explicit expression in \hyperref[infty1Kuznetsovlemma]{Lemma \ref*{infty1Kuznetsovlemma}}, after making the change of variables $m \mapsto d_2 m/ w_1$, $n \mapsto v_2 n$, and interchanging $v_1$ and $w_1$. We also do the same but with $m$ and $n$ interchanged.

We add twice the first expression to the second and the third. Using the approximate functional equations, \hyperref[approxfunclemma]{Lemma \ref*{approxfunclemma}}, with $X = \sqrt{d_2}/v_1$ and $X = \sqrt{d_2}/v_2$ respectively, and recalling \hyperref[EllMaassidentitylemma]{Lemma \ref*{EllMaassidentitylemma}}, we obtain $\MM^{\Maass}(h)$ with $h$ as in \eqref{hforbulkeq} as well as an error term arising from using $V_1^1$ in place of $V_1^{-1}$ for the odd Maa\ss{} cusp forms, just as in \cite[(2.9)]{DK18b}. By \cite[(2.5)]{DK18b}, the Cauchy--Schwarz inequality, and the large sieve, \hyperref[largesievethm]{Theorem \ref*{largesievethm}}, this error is $O_{\e}(t_g^{-1 + \e})$.

The Eisenstein terms from these instances of the Kuznetsov formula give rise to $\MM^{\Eis}(h)$ plus an error term of size $O_A(T^{-A})$ for any $A > 0$. There are no delta terms as these are opposite sign Kuznetsov formul\ae{}. Finally, the Kloosterman terms sum to the desired expression in \eqref{bulkKuzeq}.
\end{proof}

Following \cite[Section 2.3]{DK18b}, we insert a smooth compactly supported function $U(r/2t_g)$ as in \cite[(2.13)]{DK18b} into the integrand of the right-hand side of \eqref{bulkKuzeq}, absorb $W(r)$ into $U(r/2t_g)$, replace $H(r)$ with its leading order term via Stirling's formula \eqref{Stirlingeq}, and treat only the leading order terms $V(nk^2/D^{3/2}r^2)$ and $V(m\ell^2/D^{3/2}(4t_g^2 - r^2))$ of $V_1^1(nk^2/D^{3/2},r)$ and $V_2^1(m\ell^2/D^{3/2},r)$ respectively, with
\begin{equation}
\label{V(x)defeq}
V(x) \defeq \frac{1}{2\pi i} \int_{\sigma - i\infty}^{\sigma + i\infty} e^{s^2} (4\pi^2 x)^{-s} \, \frac{ds}{s}
\end{equation}
as in \cite[(2.14)]{DK18b}. Defining
\begin{equation}
\label{Q(r)defeq}
Q(r) \defeq \frac{U\left(\frac{r}{2t_g}\right)}{|r| (4t_g^2 - r^2)^{1/2}}
\end{equation}
as in \cite[(2.16)]{DK18b}, this shows that the integrals in \eqref{bulkKuzeq} can be replaced with
\begin{gather*}
\frac{16\pi \sqrt{mn}}{c} Q\left(\frac{2\pi\sqrt{mn}}{c}\right) V\left(\frac{k^2 c^2}{4\pi^2 D^{3/2} m}\right) V\left(\frac{m\ell^2}{4D^{3/2} t_g^2} \frac{1}{1 - \frac{\pi^2 mn}{t_g^2 c^2}}\right),	\\
\frac{16\pi \sqrt{mn}}{c\sqrt{D}} Q\left(\frac{2\pi\sqrt{mn}}{c\sqrt{D}}\right) V\left(\frac{k^2 c^2}{4\pi^2 \sqrt{D} m}\right) V\left(\frac{m\ell^2}{4D^{3/2} t_g^2} \frac{1}{1 - \frac{\pi^2 mn}{D t_g^2 c^2}}\right)
\end{gather*}
respectively, as in \cite[(2.15)]{DK18b}, at the cost of a negligible error. We are left with obtaining an asymptotic formula for
\begin{multline}
\label{preVoronoibulkeq}
\frac{4\pi^2}{D L(1,\chi_D)^2 L(1,g_{\psi^2})^2} \sum_{n,m,k,\ell = 1}^{\infty} \frac{\lambda_{\chi_D,1}(n,0) \lambda_{g_{\psi^2}}(m) \chi_D(k\ell)}{k\ell}	\\
\times \left(\sum_{\substack{c = 1 \\ c \equiv 0 \hspace{-.25cm} \pmod{D}}}^{\infty} \frac{S(m,-n;c)}{c^2} Q\left(\frac{2\pi\sqrt{mn}}{c}\right) V\left(\frac{k^2 c^2}{4\pi^2 D^{3/2} m}\right) V\left(\frac{m\ell^2}{4D^{3/2} t_g^2} \frac{1}{1 - \frac{\pi^2 mn}{t_g^2 c^2}}\right) \right.	\\
\left. + \sum_{\substack{c = 1 \\ (c,D) = 1}}^{\infty} \frac{S\left(m,-n\overline{D};c\right)}{c^2 D} Q\left(\frac{2\pi\sqrt{mn}}{c\sqrt{D}}\right) V\left(\frac{k^2 c^2}{4\pi^2 \sqrt{D} m}\right) V\left(\frac{m\ell^2}{4D^{3/2} t_g^2} \frac{1}{1 - \frac{\pi^2 mn}{D t_g^2 c^2}}\right)\right).
\end{multline}
We open up both Kloosterman sums and use the Vorono\u{\i} summation formula, \hyperref[Voronoilemma]{Lemma \ref*{Voronoilemma}}, for the sum over $n$. In both sums over $c$, the corresponding Vorono\u{\i} $L$-series has a pole at $s = 1$, which contributes a main term that we now calculate.

\subsection{The Main Term}

\begin{lemma}
\label{bulkpolelemma}
The pole at $s = 1$ in the Vorono\u{\i} $L$-series contributes a main term equal to
\[\frac{2}{\vol(\Gamma_0(D) \backslash \Hb)} + O\left(t_g^{-\delta}\right)\]
for \eqref{preVoronoibulkeq} for some $\delta > 0$.
\end{lemma}

\begin{proof}
For the first sum over $c$, the pole of the associated Vorono\u{\i} $L$-series as in \hyperref[Voronoilemma]{Lemma \ref*{Voronoilemma}} yields a residue equal to
\begin{multline*}
\frac{4\pi^2}{\sqrt{D} L(1,\chi_D) L(1,g_{\psi^2})^2} \sum_{m,k,\ell = 1}^{\infty} \sum_{\substack{c = 1 \\ c \equiv 0 \hspace{-.25cm} \pmod{D}}}^{\infty} \frac{\lambda_{g_{\psi^2}}(m) \chi_D(k\ell)}{k\ell c^3} V\left(\frac{k^2 c^2}{4\pi^2 D^{3/2} m}\right)	\\
\times \sum_{d \in (\Z/c\Z)^{\times}} \chi_D(d) e\left(\frac{md}{c}\right) \int_{0}^{\infty} Q\left(\frac{2\pi\sqrt{mx}}{c}\right) V\left(\frac{m\ell^2}{4D^{3/2} t_g^2} \frac{1}{1 - \frac{\pi^2 mx}{t_g^2 c^2}}\right) \, dx.
\end{multline*}
Following \cite[Section 3]{DK18b}, we make the change of variables $x \mapsto cx^2/2\pi\sqrt{m}$, extend the function $U(r/2t_g)$ in the definition \eqref{Q(r)defeq} of $Q(r)$ to the endpoints $0$ and $2t_g$ at the cost of a negligible error, make the change of variables $x \mapsto 2t_g x$, and use the definition \eqref{V(x)defeq} of $V$ as a Mellin transform, yielding an asymptotic expression of the form
\begin{multline*}
\frac{2}{\sqrt{D} L(1,\chi_D) L(1,g_{\psi^2})^2} \frac{1}{(2\pi i)^2} \int_{\sigma_1 - i\infty}^{\sigma_1 + i\infty} \int_{\sigma_2 - i\infty}^{\sigma_2 + i\infty} e^{s_1^2 + s_2^2} \pi^{-2s_2} t_g^{2s_2} D^{\frac{3}{2}(s_1 + s_2)} \int_{0}^{1} \frac{1}{(1 - x^2)^{\frac{1}{2} - s_2}} \, dx	\\
\times \sum_{m,k,\ell = 1}^{\infty} \sum_{\substack{c = 1 \\ c \equiv 0 \hspace{-.25cm} \pmod{D}}}^{\infty} \frac{\lambda_{g_{\psi^2}}(m) \chi_D(k\ell)}{m^{1 - s_1 + s_2} k^{1 + 2s_1} \ell^{1 + 2s_2} c^{1 + 2s_1}} \sum_{d \in (\Z/c\Z)^{\times}} \chi_D(d) e\left(\frac{md}{c}\right) \, \frac{ds_2}{s_2} \, \frac{ds_1}{s_1},
\end{multline*}
where $1/4 < \sigma_1 < \sigma_2 < 1/2$. We use \hyperref[Miyakelemma]{Lemma \ref*{Miyakelemma}} to reexpress the sum over $d$, a Gauss sum, as a sum over $a \mid (c/D,m)$; next, we make the change of variables $c \mapsto acD$ and $m \mapsto am$, then use \eqref{invertcuspmult} to separate $\lambda_{g_{\psi^2}}(am)$ as a sum over $b \mid (a,m)$; finally, we make the change of variables $a \mapsto ab$ and $m \mapsto bm$, yielding
\begin{multline*}
\frac{2}{D L(1,\chi_D) L(1,g_{\psi^2})^2} \frac{1}{(2\pi i)^2} \int_{\sigma_1 - i\infty}^{\sigma_1 + i\infty} \int_{\sigma_2 - i\infty}^{\sigma_2 + i\infty} e^{s_1^2 + s_2^2} \pi^{-2s_2} t_g^{2s_2} D^{\frac{3s_2 - s_1}{2}} \int_{0}^{1} \frac{1}{(1 - x^2)^{\frac{1}{2} - s_2}} \, dx	\\
\times \sum_{m = 1}^{\infty} \frac{\lambda_{g_{\psi^2}}(m) \chi_D(m)}{m^{1 - s_1 + s_2}} \sum_{k = 1}^{\infty} \frac{\chi_D(k)}{k^{1 + 2s_1}} \sum_{\ell = 1}^{\infty} \frac{\chi_D(\ell)}{\ell^{1 + 2s_2}} \sum_{c = 1}^{\infty} \frac{\mu(c) \chi_D(c)}{c^{1 + 2s_1}} \sum_{a = 1}^{\infty} \frac{\lambda_{g_{\psi^2}}(a)}{a^{1 + s_1 + s_2}} \sum_{\substack{b = 1 \\ (b,D) = 1}}^{\infty} \frac{\mu(b)}{b^{2 + 2s_2}} \, \frac{ds_2}{s_2} \, \frac{ds_1}{s_1}.
\end{multline*}
The sums over $m$, $k$, $\ell$, $c$, $a$, and $b$ in the second line simplify to
\[\frac{L(1 + 2s_2,\chi_D) L^D(1 - s_1 + s_2,g_{\psi^2}) L(1 + s_1 + s_2,g_{\psi^2})}{\zeta^D(2 + 2s_2)}.\]
We shift the contour in the integral over $s_2$ to the line $\Re(s_2) = \sigma_1 - 1/2$; via the subconvex bounds in \hyperref[MVsubconvthm]{Theorem \ref*{MVsubconvthm}}, the resulting contour integral is bounded by a negative power of $t_g$, so that the dominant contribution comes from the residue due to the simple pole at $s_2 = 0$, namely
\[\frac{6}{\pi \nu(D)} \frac{1}{2\pi i} \int_{\sigma_1 - i\infty}^{\sigma_1 + i\infty} e^{s_1^2} D^{-\frac{s_1}{2}} \frac{L^D(1 - s_1,g_{\psi^2}) L(1 + s_1,g_{\psi^2})}{L(1,g_{\psi^2}) L^D(1,g_{\psi^2})} \, \frac{ds_1}{s_1}.\]

Now we do the same with the second sum over $c$. We open up the Kloosterman sum, make the change of variables $d \mapsto -\overline{Dd}$, and use the Vorono\u{\i} summation formula, \hyperref[Voronoilemma]{Lemma \ref*{Voronoilemma}}, for the sum over $n$; the pole of the Vorono\u{\i} $L$-series at $s = 1$ yields the term
\begin{multline*}
\frac{4\pi^2}{D^2 L(1,\chi_D) L(1,g_{\psi^2})^2} \sum_{m,k,\ell = 1}^{\infty} \sum_{\substack{c = 1 \\ (c,D) = 1}}^{\infty} \frac{\lambda_{g_{\psi^2}}(m) \chi_D(k\ell c)}{k\ell c^2} V\left(\frac{k^2 c^2}{4\pi^2 \sqrt{D} m}\right)	\\
\times \sum_{d \in (\Z/c\Z)^{\times}} e\left(-\frac{m\overline{Dd}}{c}\right) \int_{0}^{\infty} Q\left(\frac{2\pi\sqrt{mx}}{c\sqrt{D}}\right) V\left(\frac{m\ell^2}{4D^{3/2} t_g^2} \frac{1}{1 - \frac{\pi^2 mx}{D t_g^2 c^2}}\right) \, dx.
\end{multline*}
We make the change of variables $x \mapsto c\sqrt{D} x^2/2\pi\sqrt{m}$, extend the function $U(r/2t_g)$ in the definition \eqref{Q(r)defeq} of $Q(r)$ to the endpoints $0$ and $2t_g$ at the cost of a negligible error, make the change of variables $x \mapsto 2t_g x$, and use the definition \eqref{V(x)defeq} of $V$ as a Mellin transform, yielding the asymptotic expression
\begin{multline*}
\frac{2}{D L(1,\chi_D) L(1,g_{\psi^2})^2} \frac{1}{(2\pi i)^2} \int_{\sigma_1 - i\infty}^{\sigma_1 + i\infty} \int_{\sigma_2 - i\infty}^{\sigma_2 + i\infty} e^{s_1^2 + s_2^2} \pi^{-2s_2} t_g^{2s_2} D^{\frac{s_1}{2}} \int_{0}^{1} \frac{1}{(1 - x^2)^{\frac{1}{2} - s_2}} \, dx	\\
\times \sum_{m,k,\ell = 1}^{\infty} \sum_{\substack{c = 1 \\ (c,D) = 1}}^{\infty} \frac{\lambda_{g_{\psi^2}}(m) \chi_D(k\ell c)}{m^{1 - s_1 + s_2} k^{1 + 2s_1} \ell^{1 + 2s_2} c^{1 + 2s_1}} \sum_{d \in (\Z/c\Z)^{\times}} e\left(-\frac{m\overline{Dd}}{c}\right) \, \frac{ds_2}{s_2} \, \frac{ds_1}{s_1}.
\end{multline*}
The sum over $d$ is a Ramanujan sum, $\sum_{a \mid (m,c)} a \mu(c/a)$. We make the change of variables $c \mapsto ac$ and $m \mapsto am$, then use \eqref{invertcuspmult} and make the change of variables $a \mapsto ab$ and $m \mapsto bm$, leading to
\begin{multline*}
\frac{2}{D L(1,\chi_D) L(1,g_{\psi^2})^2} \frac{1}{(2\pi i)^2} \int_{\sigma_1 - i\infty}^{\sigma_1 + i\infty} \int_{\sigma_2 - i\infty}^{\sigma_2 + i\infty} e^{s_1^2 + s_2^2} \pi^{-2s_2} t_g^{2s_2} D^{\frac{s_1}{2}} \int_{0}^{1} \frac{1}{(1 - x^2)^{\frac{1}{2} - s_2}} \, dx	\\
\times \sum_{m = 1}^{\infty} \frac{\lambda_{g_{\psi^2}}(m)}{m^{1 - s_1 + s_2}} \sum_{k = 1}^{\infty} \frac{\chi_D(k)}{k^{1 + 2s_1}} \sum_{\ell = 1}^{\infty} \frac{\chi_D(\ell)}{\ell^{1 + 2s_2}} \sum_{c = 1}^{\infty} \frac{\mu(c) \chi_D(c)}{c^{1 + 2s_1}} \sum_{a = 1}^{\infty} \frac{\lambda_{g_{\psi^2}}(a) \chi_D(a)}{a^{1 + s_1 + s_2}} \sum_{\substack{b = 1 \\ (b,D) = 1}}^{\infty} \frac{\mu(b)}{b^{2 + 2s_2}} \, \frac{ds_2}{s_2} \, \frac{ds_1}{s_1}.
\end{multline*}
The sums over $m$, $k$, $\ell$, $c$, $a$, and $b$ in the second line simplify to
\[\frac{L(1 + 2s_2,\chi_D) L(1 - s_1 + s_2,g_{\psi^2}) L^D(1 + s_1 + s_2,g_{\psi^2})}{\zeta^D(2 + 2s_2)}.\]
Again, we shift the contour in the integral over $s_2$ to the line $\Re(s_2) = \sigma_1 - 1/2$, with a main term coming from the residue at $s_2 = 0$ given by
\[\frac{6}{\pi \nu(D)} \frac{1}{2\pi i} \int_{\sigma_1 - i\infty}^{\sigma_1 + i\infty} e^{s_1^2} D^{\frac{s_1}{2}} \frac{L(1 - s_1,g_{\psi^2}) L^D(1 + s_1,g_{\psi^2})}{L(1,g_{\psi^2}) L^D(1,g_{\psi^2})} \, \frac{ds_1}{s_1}.\]

We finish by adding together these two main contributions and observing that the resulting integrand is odd and hence equal to half its residue at $s_1 = 0$, namely
\[\frac{6}{\pi \nu(D)} = \frac{2}{\vol(\Gamma_0(D) \backslash \Hb)}.\qedhere\]
\end{proof}

\subsection{The Vorono\u{\i} Dual Sums}

Having applied the Vorono\u{\i} summation formula, \hyperref[Voronoilemma]{Lemma \ref*{Voronoilemma}}, to the sum over $n$ in \eqref{preVoronoibulkeq} and dealt with the terms arising from the pole of the Vorono\u{\i} $L$-series, we now treat the terms arising from the Vorono\u{\i} dual sums.

\begin{lemma}
\label{bulkduallemma}
The Vorono\u{\i} dual sums are of size $O(t_g^{-\delta})$ for some $\delta > 0$.
\end{lemma}

\begin{proof}
There are two dual sums associated to the two sums over $c$ in \eqref{preVoronoibulkeq}. We prove this bound only for the former dual sum; the proof for the latter follows with minor modifications. The dual sum to the first term can be expressed as a dyadic sum over $N \leq t_g^{2 + \e}$ times
\begin{multline*}
\frac{4\pi^2 N}{D L(1,\chi_D)^2 L(1,g_{\psi^2})^2} \sum_{\pm} \sum_{n,m,k,\ell= 1}^{\infty} \sum_{\substack{c = 1 \\ c \equiv 0 \hspace{-.25cm} \pmod{D}}}^{\infty} \sum_{d \in (\Z/c\Z)^{\times}} \chi_D(d) e\left(\frac{(m \pm n)d}{c}\right)	\\
\times \frac{\lambda_{\chi_D,1}(n,0) \lambda_{g_{\psi^2}}(m) \chi_D(k\ell)}{k\ell c^3} V\left(\frac{k^2 c^2}{4\pi^2 D^{3/2} m}\right) \check{\Phi}_1^{\pm}\left(\frac{Nn}{c^2}\right),
\end{multline*}
where $\Phi_1$ is a smooth function compactly supported on $(1/2,3/2)$ and
\begin{align*}
\check{\Phi}_1^{\pm}(x) & \defeq \frac{1}{2\pi i} \int_{\sigma - i\infty}^{\sigma + i\infty} \widehat{\JJ_0^{\pm}}(s) \widehat{\Phi_1}\left(-\frac{s}{2}\right) x^{-\frac{s}{2}} \, ds,	\\
\Phi_1(x) & \defeq x\Psi_1(x) Q\left(\frac{2\pi\sqrt{mNx}}{c}\right) V\left(\frac{m\ell^2}{4D^{3/2} t_g^2} \frac{1}{1 - \frac{\pi^2 mNx}{t_g^2 c^2}}\right),
\end{align*}
with $\sigma > 0$. This identity for the dual sum is proven in the same way as in \cite[Section 4.1]{DK18b}: we insert a smooth partition of unity $\Psi_1(n/N)$ to the sum over $n$ in \eqref{preVoronoibulkeq}, then apply of the Vorono\u{\i} summation formula, \hyperref[Voronoilemma]{Lemma \ref*{Voronoilemma}}, to the ensuing sum over $n$.

We proceed along the exact same lines as \cite[Section 4.1]{DK18b}; in this way, the problem is reduced to proving that the quantity
\[\frac{N}{t_g^2} \sum_{\pm} \sum_{m = 1}^{\infty} \sum_{\substack{c = 1 \\ c \equiv 0 \hspace{-.25cm} \pmod{D}}}^{\infty} \sum_{d \in (\Z/c\Z)^{\times}} \chi_D(d) e\left(\frac{(m \pm n)d}{c}\right) \frac{\lambda_{g_{\psi^2}}(m)}{c^3} Z\left(\frac{\sqrt{Nm}}{ct_g}\right) \Psi_2\left(\frac{m}{M}\right)\]
is $O(t_g^{-\delta})$ for any $n < t_g^{\e}$ and $t_g^{2 - \e} < M < t_g^{2 + \e}$, as in \cite[(4.3)]{DK18b}, with $\Psi_2$ another smooth function supported on $(1/2,3/2)$ and $Z(x) \defeq U(x)/4|x|\sqrt{1 - x^2}$.

Now we apply the Vorono\u{\i} summation formula, \hyperref[Voronoilemma]{Lemma \ref*{Voronoilemma}}, to the sum over $m$, yielding
\[\frac{2MN}{t_g^2} \sum_{\pm_1} \sum_{\pm_2} \sum_{m = 1}^{\infty} \sum_{\substack{c = 1 \\ c \equiv 0 \hspace{-.25cm} \pmod{D}}}^{\infty} \lambda_{g_{\psi^2}}(m) \frac{S(m,\pm_1 n;c)}{c^4} \check{\Phi}_2^{\pm_2}\left(\frac{Mm}{c^2},t_g\right),\]
where for $\sigma > 0$,
\begin{align*}
\check{\Phi}_2^{\pm}(x,t_g) & \defeq \frac{1}{2\pi i} \int_{\sigma - i\infty}^{\sigma + i\infty} \widehat{\JJ_{2t_g}^{\pm}}(s) \widehat{\Phi_2}\left(-\frac{s}{2}\right) x^{-\frac{s}{2}} \, ds,	\\
\Phi_2(x) & \defeq x \Psi_2(x) Z\left(\frac{\sqrt{MNx}}{ct_g}\right).
\end{align*}
We continue to follow \cite[Section 4.2]{DK18b}, by which the problem is reduced to showing that the quantity
\[\frac{1}{t_g} \sum_{\pm} \sum_{m = 1}^{\infty} \frac{\lambda_{g_{\psi^2}}(m)}{\sqrt{m}} \sum_{\substack{c = 1 \\ c \equiv 0 \hspace{-.25cm} \pmod{D}}}^{\infty} \frac{S(m,\pm n;c)}{c} \Phi\left(\frac{\sqrt{mn}}{c}\right) \Psi\left(\frac{m}{B}\right)\]
is $O(t_g^{-\delta})$, as in \cite[(4.6)]{DK18b}, where $\Phi$ and $\Psi$ are smooth bump functions with $\Psi$ supported on $(1/2,3/2)$ and $B \leq t_g^{2 + \e}$.

We spectrally expand the sums of Kloosterman sums via Kloosterman summation formul\ae{}, \hyperref[Kloostermanthm]{Theorems \ref*{Kloostermanthm}} and \ref{infty1Kloostermanthm}, with $H = \Phi$. From \cite[Lemma 3.6]{BuK17b}, $(\Ls^{\pm} \Phi)(t) \ll t_g^{-A}$ and $(\Ls^{\hol} \Phi)(k) \ll t_g^{-A}$ for any $A > 0$ unless $|t| < t_g^{\e}$ and $k < t_g^{\e}$, in which case we instead have the bound $O_{\e}(t_g^{\e})$. Using the explicit expressions for the Maa\ss{} cusp form, Eisenstein, and holomorphic cusp form terms given in \hyperref[Kuznetsovlemma]{Lemmata \ref*{Kuznetsovlemma}} and \ref{infty1Kuznetsovlemma}, we have reduced the problem to showing that both
\[\frac{1}{t_g} \sum_{m = 1}^{\infty} \frac{\lambda_{g_{\psi^2}}(m) \lambda_f(m)}{\sqrt{m}} \Psi\left(\frac{m}{B}\right), \qquad \frac{1}{t_g} \sum_{m = 1}^{\infty} \frac{\lambda_{g_{\psi^2}}(m) \lambda(m,t)}{\sqrt{m}} \Psi\left(\frac{m}{B}\right)\]
are $O(t_g^{-\delta})$ for $B < t_g^{2 + \e}$ for all $f$ in either $\BB_0^{\ast}(\Gamma_0(d_1))$ with $|t_f| < t_g^{\e}$ or in $\BB_{\hol}^{\ast}(\Gamma_0(d_1))$ with $k_f < t_g^{\e}$, where $d_1 \mid D$, and for $|t| < t_g^{\e}$. By Mellin inversion, these two expressions are respectively equal to
\begin{gather*}
\frac{1}{t_g} \frac{1}{2\pi i} \int_{\sigma - i\infty}^{\sigma + i\infty} \frac{L\left(\frac{1}{2} + s, f \otimes g_{\psi^2}\right)}{L(1 + 2s,\chi_D)} B^s \widehat{\Psi}(s) \, ds,	\\
\frac{1}{t_g} \frac{1}{2\pi i} \int_{\sigma - i\infty}^{\sigma + i\infty} \frac{L\left(\frac{1}{2} + s + it, g_{\psi^2}\right) L\left(\frac{1}{2} + s - it, g_{\psi^2}\right)}{L(1 + 2s,\chi_D)} B^s \widehat{\Psi}(s) \, ds
\end{gather*}
for any $\sigma > 1/2$. The rapid decay of $\widehat{\Psi}$ in vertical strips allows the integral to be restricted to $|\Im(s)| < t_g^{\e}$ and shifted to $\sigma = 0$, at which point the subconvex bounds in \hyperref[MVsubconvthm]{Theorem \ref*{MVsubconvthm}} bound the numerators by $O(t_g^{1 - \delta})$ for some $\delta > 0$, which completes the proof.
\end{proof}

\begin{proof}[Proof of {\hyperref[dihedralmomentsprop]{Proposition \ref*{dihedralmomentsprop} (2)}}]
This follows directly upon combining \hyperref[bulkKuzlemma]{Lemmata \ref*{bulkKuzlemma}}, \ref{bulkpolelemma}, and \ref{bulkduallemma}.
\end{proof}

\begin{remark}
Perhaps one can prove this result using analytic continuation, as in the proof of \hyperref[spectralreciprocityprop]{Proposition \ref*{spectralreciprocityprop}}, instead of using approximate functional equations. We choose the latter path since the groundwork is laid out in \cite{DK18b}, and it avoids technical difficulties in the analytic continuation approach of ensuring a valid choice of test function $h$.
\end{remark}

\section{Spectral Reciprocity for the Short Transition Range}
\label{spectralrec2sect}

For $\hf \defeq (h,h^{\hol}) : (\R \cup i(-1/2,1/2)) \times 2\N \to \C^2$, let
\[\undertilde{\MM}^{\pm}(\hf) \defeq \undertilde{\MM}^{\Maass}(h) + \undertilde{\MM}^{\Eis}(h) + \delta_{+,\pm} \undertilde{\MM}^{\hol}(h^{\hol})\]
with
\begin{align*}
\undertilde{\MM}^{\Maass}(h) & \defeq \sum_{d_1 d_2 = D} 2^{\omega(d_2)} \frac{\varphi(d_2)}{d_2} \sum_{f \in \BB_0^{\ast}(\Gamma_0(d_1))} \frac{L\left(\frac{1}{2},f\right)^2 L\left(\frac{1}{2},f \otimes \chi_D\right)^2}{L_{d_2}\left(\frac{1}{2}, f\right) L^{d_2}(1,\sym^2 f)} h(t_f),	\\
\undertilde{\MM}^{\Eis}(h) & \defeq \frac{2^{\omega(D)}}{2\pi} \int_{-\infty}^{\infty} \left|\frac{\zeta\left(\frac{1}{2} + it\right)^2 L\left(\frac{1}{2} + it, \chi_D\right)^2}{\zeta_D\left(\frac{1}{2} + it\right) \zeta^D(1 + 2it)}\right|^2 h(t) \, dt,	\\
\undertilde{\MM}^{\hol}\left(h^{\hol}\right) & \defeq \sum_{d_1 d_2 = D} 2^{\omega(d_2)} \frac{\varphi(d_2)}{d_2} \sum_{f \in \BB_{\hol}^{\ast}(\Gamma_0(d_1))} \frac{L\left(\frac{1}{2},f\right)^2 L\left(\frac{1}{2},f \otimes \chi_D\right)^2}{L_{d_2}\left(\frac{1}{2}, f\right) L^{d_2}(1,\sym^2 f)} h^{\hol}(k_f).
\end{align*}
The main result of this section is the following identity.

\begin{proposition}[{Cf.~\hyperref[spectralreciprocityprop]{Proposition \ref*{spectralreciprocityprop}}}]
\label{spectralreciprocity2prop}
For admissible $\hf$, we have that
\[\undertilde{\MM}^{-}(\hf) = \undertilde{\GG}(\hf) + \sum_{\pm} \undertilde{\MM}^{\pm} \left(\Ts_0^{\pm} \hf\right),\]
where $\Ts_0^{\pm}$ is as in \eqref{Tshfeq} with $t_g$ replaced by $0$ and $\undertilde{\GG}(\hf)$ is the holomorphic extension to $(s_1,s_2) = (1/2,1/2)$ of
\[\undertilde{\GG}\left(s_1,s_2;\hf\right) \defeq \undertilde{\NN}\left(s_1,s_2;\hf\right) - \undertilde{\RR}\left(s_1,s_2;\hf\right) + \sum_{\pm} \undertilde{\RR}\left(s_2,s_1;\Ts_{s_1,s_2,0}^{\pm} \hf\right)\]
with
\begin{align*}
\undertilde{\NN}\left(s_1,s_2;\hf\right) & \defeq 2 D^{2(1 - s_1)} L(1,\chi_D) \widehat{\Ks^{-} h}(2(1 - s_1))	\\
& \hspace{2cm} \times \frac{\zeta(s_1 + s_2) \zeta^D(1 - s_1 + s_2) L(s_1 + s_2,\chi_D) L(1 - s_1 + s_2,\chi_D)}{\zeta^D(1 + 2s_1) L(2s_1,\chi_D)}	\\
& \hspace{1.5cm} + 2 D^{2(1 - s_2)} L(1,\chi_D) \widehat{\Ks^{-} h}(2(1 - s_2))	\\
& \hspace{2cm} \times \frac{\zeta(s_1 + s_2) \zeta^D(1 + s_1 - s_2) L(s_1 + s_2,\chi_D) L(1 + s_1 - s_2,\chi_D)}{\zeta^D(1 + 2s_2) L(2s_2,\chi_D)},	\\
\undertilde{\RR}\left(s_1,s_2;\hf\right) & \defeq \sum_{\pm_1} \Res_{\substack{t = \pm_1 i (1 - s_1) \\ t = \pm_1 i (1 - s_2)}} (\pm_1 i) h(t) \Ell_D(s_1,s_2,t)	\\
& \hspace{2cm} \times \prod_{\pm_2} \frac{\zeta(s_1 \pm_2 it) \zeta(s_2 \pm_2 it) L(s_1 \pm_2 it, \chi_D) L(s_2 \pm_2 it, \chi_D)}{\zeta(1 \pm_2 2it)},	\\
\undertilde{\RR}\left(s_2,s_1;\Ts_{s_1,s_2,0}^{\pm} \hf\right) & \defeq \sum_{\pm_1} \Res_{\substack{t = \pm_1 i (1 - s_1) \\ t = \pm_1 i (1 - s_2)}} (\pm_1 i) \left(\Ls^{\pm} H_{s_1,s_2,0}^{\pm}\right)(t) \Ell_D(s_2,s_1,t)	\\
& \hspace{2cm} \times \prod_{\pm_2} \frac{\zeta(s_2 \pm_2 it) \zeta(s_1 \pm_2 it) L(s_2 \pm_2 it, \chi_D) L(s_1 \pm_2 it, \chi_D)}{\zeta(1 \pm_2 2it)}.
\end{align*}
Here $\Ts_{s_1,s_2,0}^{\pm} \hf$ is as in \eqref{Tss1s2tgpmeq} with $t_g$ replaced by $0$.
\end{proposition}

Similarly to \hyperref[spectralrecsect]{Section \ref*{spectralrecsect}}, we define
\begin{align*}
\undertilde{\MM}^{\Maass,\pm}\left(s_1,s_2;h\right) & \defeq \sum_{d_1 d_2 = D} \sum_{f \in \BB_0^{\ast}(\Gamma_0(d_1))} \epsilon_f^{\frac{1 \mp 1}{2}} \Ell_{d_2}(s_1,s_2,f)	\\
& \hspace{2cm} \times \frac{L(s_1,f) L(s_2,f) L(s_1,f \otimes \chi_D) L(s_2,f \otimes \chi_D)}{L(1,\sym^2 f)} h(t_f),	\\
\undertilde{\MM}^{\Eis}\left(s_1,s_2;h\right) & \defeq \frac{1}{2\pi} \int_{-\infty}^{\infty} \Ell_D(s_1,s_2,t)	\\
& \hspace{2cm} \times \prod_{\pm} \frac{\zeta(s_1 \pm it) \zeta(s_2 \pm it) L(s_1 \pm it, \chi_D) L(s_2 \pm it, \chi_D)}{\zeta(1 \pm 2it)} h(t) \, dt,	\\
\undertilde{\MM}^{\hol}\left(s_1,s_2;h^{\hol}\right) & \defeq \sum_{d_1 d_2 = D} \sum_{f \in \BB_{\hol}^{\ast}(\Gamma_0(d_1))} \Ell_{d_2}(s_1,s_2,f)	\\
& \hspace{2cm} \times \frac{L(s_1,f) L(s_2,f) L(s_1,f \otimes \chi_D) L(s_2,f \otimes \chi_D)}{L(1,\sym^2 f)} h^{\hol}(k_f),
\end{align*}
for $s_1,s_2 \in \C$. We additionally set
\[\undertilde{\MM}^{\pm}\left(s_1,s_2;\hf\right) \defeq \undertilde{\MM}^{\Maass,\pm}\left(s_1,s_2;h\right) + \undertilde{\MM}^{\Eis}\left(s_1,s_2;h\right) + \delta_{\pm,+} \undertilde{\MM}^{\hol}\left(s_1,s_2;h^{\hol}\right).\]

\begin{lemma}[{Cf.~\hyperref[spectralrecpreaclemma]{Lemma \ref*{spectralrecpreaclemma}}}]
For admissible $\hf$ and $5/4 < \Re(s_1), \Re(s_2) < 3/2$ with $s_1 \neq s_2$, we have that
\[\undertilde{\MM}^{-}\left(s_1,s_2;\hf\right) = \undertilde{\NN}\left(s_1,s_2;\hf\right) + \sum_{\pm} \undertilde{\MM}^{\pm} \left(s_2,s_1;\Ts_{s_1,s_2,0}^{\pm} \hf\right).\]
\end{lemma}

\begin{proof}
This follows by the same method of proof as for \hyperref[spectralreciprocityprop]{Proposition \ref*{spectralreciprocityprop}} except that we replace $\lambda_{g_{\psi^2}}(n)$ with $\lambda_{\chi_D,1}(n,0)$, so that $t_g$ is replaced by $0$. In place of a simple pole at $s = 2(1 - s_1)$ with residue given by \eqref{residue}, there are two simple poles at $s = 2(1 - s_1)$ and $s = 2(1 - s_2)$. When $\Re(s_2) > \Re(s_1)$, the former is given by
\[2 D^{3/2} L(1,\chi_D) \widehat{\Ks^{-} h}(2(1 - s_1)) \sum_{\substack{c = 1 \\ c \equiv 0 \hspace{-.25cm} \pmod{D}}}^{\infty} \frac{1}{c^{2s_1}} \sum_{d \in (\Z/c\Z)^{\times}} \chi_D(d) L\left(1 - s_1 + s_2, E_{\chi_D,1},-\frac{\overline{d}}{c}\right)\]
by \hyperref[Voronoilemma]{Lemma \ref*{Voronoilemma}}. Just as in the proof of \hyperref[spectralreciprocityprop]{Proposition \ref*{spectralreciprocityprop}}, we open up the Vorono\u{\i} $L$-series, reexpress the Gauss sum over $d$ as a sum over $a \mid (c/D,m)$ via \hyperref[Miyakelemma]{Lemma \ref*{Miyakelemma}}, make the change of variables $c \mapsto acD$ and $m \mapsto am$, apply \eqref{invertcuspmult}, and then make the change of variables $m \mapsto bm$ and $a \mapsto ab$, which leads us to
\[2 D^{2(1 - s_1)} L(1,\chi_D) \widehat{\Ks^{-} h}(2(1 - s_1)) \frac{\zeta(s_1 + s_2) \zeta^D(1 - s_1 + s_2) L(s_1 + s_2,\chi_D) L(1 - s_1 + s_2,\chi_D)}{\zeta^D(1 + 2s_1) L(2s_1,\chi_D)}.\]
While only initially valid for $\Re(s_2) > \Re(s_1)$, this extends holomorphically in the region $5/4 < \Re(s_1), \Re(s_2) < 3/2$ with $s_1 \neq s_2$. An identical calculation yields the residue at $s = 2(1 - s_2)$.
\end{proof}

\begin{proof}[Proof of {\hyperref[spectralreciprocity2prop]{Proposition \ref*{spectralreciprocity2prop}}}]
This follows the same method as \cite[Proof of Theorem 1]{BLM19}, \cite[Proof of Theorem 1]{BlK19b}, and \cite[Proof of Theorem 2]{BlK19a}. The holomorphic extensions of $\MM^{\Eis}(s_1,s_2;h)$ and $\MM^{\Eis} (s_2,s_1;\Ls^{\pm} H_{s_1,s_2,0}^{\pm})$ for $\Re(s_1), \Re(s_2) < 1$ give rise to additional polar divisors arise via shifting the contour in the integration over $t$, namely $\undertilde{\RR}(s_1,s_2;\hf)$ and $\undertilde{\RR}(s_2,s_1;\Ts_{s_1,s_2,0}^{\pm} \hf)$. In this way, we obtain the identity
\[\undertilde{\MM}^{-}\left(s_1,s_2;\hf\right) = \undertilde{\GG}\left(s_1,s_2;\hf\right) + \sum_{\pm} \undertilde{\MM}^{\pm}\left(s_2,s_1;\Ts_{s_1,s_2,0}^{\pm} \hf\right)\]
for $\Re(s_1),\Re(s_2) \geq 1/2$ with $s_1 \neq s_2$. It remains to note that since the terms $\undertilde{\MM}^{-}(s_1,s_2;\hf)$ and $\undertilde{\MM}^{\pm}(s_2,s_1;\Ts_{s_1,s_2,0}^{\pm} \hf)$ extend holomorphically to $(s_1,s_2) = (1/2,1/2)$, so must $\undertilde{\GG}(s_1,s_2;\hf)$.
\end{proof}

\section{Bounds for the Transform for the Short Transition Range}
\label{boundstrans2sect}

We take $\hf = (h,0)$ in \hyperref[spectralreciprocity2prop]{Proposition \ref*{spectralreciprocity2prop}} to be
\begin{equation}
\label{hTUeq}
h(t) = h_{T,U}(t) \defeq \left(e^{-\left(\frac{t - T}{U}\right)^2} + e^{-\left(\frac{t + T}{U}\right)^2}\right) P_T(t), \quad P_T(t) \defeq \prod_{j = 1}^{N} \left(\frac{t^2 + \left(j - \frac{1}{2}\right)^2}{T^2}\right)^2,
\end{equation}
for some fixed large integer $N \geq 500$, $T > 0$, and $T^{1/3} \ll U \ll T$, so that $h_{T,U}(t)$ is positive for $t \in \R \cup i(-1/2,1/2)$ and bounded from below by a constant dependent only on $N$ for $t \in [-T - U, -T + U] \cup [T - U, T + U]$. The transform $H_{T,U}^{\pm}$ as in \eqref{Htgeq} of $h_{T,U}$ is
\[H_{T,U}^{\pm}(x) = \frac{2}{\pi i} \int_{\sigma_1 - i\infty}^{\sigma_1 + i\infty} \widehat{\Ks^{-} h_{T,U}}(s) \GG_0^{\pm}(1 - s) x^s \, ds\]
with $-3 < \sigma_1 < 1$, where $\GG_0^{\pm}(s)$ is as in \eqref{GGtgeq}. We once again wish to determine the asymptotic behaviour of the functions
\begin{align*}
\left(\Ls^{\pm} H_{T,U}^{\pm}\right)(t) & = \int_{0}^{\infty} \JJ_t^{\pm}(x) H_{T,U}^{\pm}(x) \, \frac{dx}{x},	\\
\left(\Ls^{\hol} H_{T,U}^{+}\right)(k) & = \int_{0}^{\infty} \JJ_k^{\hol}(x) H_{T,U}^{+}(x) \, \frac{dx}{x},
\end{align*}
with uniformity in all variables $T$, $U$, and $t$ or $k$.

\begin{lemma}[{Cf.~\cite[Lemma 4]{BLM19}, \cite[Lemma 1]{BlK19a}}]
\label{BLMLemma4lemma}
For $j \in \N_0$ with $j \leq N$, we have that
\[x^j \frac{d^j}{dx^j} \left(\Ks^{-} h_{T,U}\right)(x) \ll_j \begin{dcases*}
U \min\left\{\left(\frac{x}{T}\right)^{N/2}, \left(\frac{x}{T}\right)^{-N/2}\right\} & if $|x - T| > U \log T$,	\\
T \left(\frac{T}{U}\right)^j \left(1 + \frac{|x - T|}{U}\right)^{4N} e^{-\left(\frac{x - T}{U}\right)^2} & if $|x - T| \leq U \log T$.
\end{dcases*}\]
\end{lemma}

\begin{proof}
The proof will follow via the same methods as \cite[Proof of Lemma 4]{BLM19} and \cite[Proof of Lemma 1]{BlK19a}, which in turn are inspired by \cite[Proof of Lemma 3.8]{BuK17a}, so we only sketch the details. We recall that
\[\left(\Ks^{-} h_{T,U}\right)(x) = \int_{-\infty}^{\infty} \JJ_r^{-}(x) h_{T,U}(r) \, d_{\spec}r.\]
We will use the following, from \cite[(2.15), (A.1), (A.2), (A.3), (A.6)]{BLM19}:
\begin{align}
\label{(2.15)}
\JJ_r^{-}(x) & = 4 \cosh \pi r K_{2ir}(4\pi x) = \pi i \frac{I_{2ir}(4\pi x) - I_{-2ir}(4\pi x)}{\sinh \pi r},	\\
\label{(A.1)}
\frac{d^j}{dx^j} K_{2ir}(4\pi x) & = (-2\pi)^j \sum_{n = 0}^{j} \binom{j}{n} K_{2ir - j + 2n}(4\pi x) \quad \text{for $j \in \N_0$,}	\\
\
\label{(A.2)}
\frac{d^j}{dx^j} I_{2ir}(4\pi x) & = (2\pi)^j \sum_{n = 0}^{j} \binom{j}{n} I_{2ir - j + 2n}(4\pi x) \quad \text{for $j \in \N_0$,}	\\
\label{(A.3)}
\JJ_r^{-}(x) & \ll_{\Im(r)} e^{\min\{0,-4\pi x + \pi |\Re(r)|\}} \left(\frac{1 + |r| + 4\pi x}{4\pi x}\right)^{2|\Im(r)| + \frac{1}{10}},	\\
\label{(A.6)}
e^{-\pi|r|} I_{2ir}(4\pi x) & \ll_{\Im(r)} \frac{x^{-2\Im(r)}}{(1 + |r|)^{\frac{1}{2} - 2\Im(r)}} \quad \text{for $0 < x < \frac{(1 + |r|)^{1/2}}{4\pi}$.}
\end{align}

We first deal with the range $x \leq 1$. We use \eqref{(2.15)} to split up into $I_{2ir}(4\pi x)$ and $I_{-2ir}(4\pi x)$, then shift the contour to $\Im(r) = -N$ and $\Im(r) = N$ respectively. We differentiate under the integral sign and then use \eqref{(A.2)} and \eqref{(A.6)}, which shows that
\[x^j \frac{d^j}{dx^j} \left(\Ks^{-} h_{T,U}\right)(x) \ll_j x^j \int_{0}^{\infty}
\frac{x^{2N - j}}{(1 + r)^{\frac{1}{2} + 2N - j}} e^{-\left(\frac{r - T}{U}\right)^2} \left(\frac{1 + r}{T}\right)^{4N} r \, dr \ll_j U \frac{x^{2N}}{T^{2N - j + \frac{3}{2}}},\]
which is certainly sufficient.

Next, we deal with the range $1 \leq x \leq T^{13/12}$. We consider
\[h_{\spec}(r) \defeq \frac{1}{2\pi^2} h_{T,U}(r) r \tanh \pi r.\]
The $j$-th derivative of the Fourier transform $\check{h}_{\spec}(x)$ is
\[\frac{d^j}{dx^j} \check{h}_{\spec}(x) = (-2\pi i)^j \int_{-\infty}^{\infty} \frac{1}{2\pi^2} h_{T,U}(r) r^{1 + j} \tanh \pi r e(-rx) \, dr.\]
We integrate by parts $A_1$ times:
\[\frac{d^j}{dx^j} \check{h}_{\spec}(x) = \frac{1}{2\pi^2} (-1)^j (2\pi i)^{j - A_1} x^{-A_1} \int_{-\infty}^{\infty} \frac{d^{A_1}}{dr^{A_1}} \left(h_{T,U}(r) r^{1 + j} \tanh \pi r\right) e(-rx) \, dr.\]
By the Leibniz rule, we find that
\begin{equation}
\label{(6.3)}
\frac{d^j}{dx^j} \check{h}_{\spec}(x) \ll_{j,A_1} T^{1 + j} U (1 + T|x|)^{-A_1}	
\end{equation}
for $0 \leq A_1 \leq 4N$. Alternatively, we may shift the contour to $\Im(r) = -\sgn(x)N$, which gives
\begin{equation}
\label{(6.4)}
\frac{d^j}{dx^j} \check{h}_{\spec}(x) \ll_j T^{1 + j} U e^{-\pi N|x|}.
\end{equation}
Following \cite[Proof of Lemma 4]{BLM19}, using \eqref{(6.3)} and \eqref{(6.4)} in place of \cite[(6.3) and (6.4)]{BLM19}, we find that $x^j (\Ks^{-} h_{T,U})^{(j)}(x)$ is equal to \cite[(6.12)]{BLM19}, except for the three error terms in this equation being bounded by $UT^{1 + \frac{5}{14}j - N}$, and the main term being a linear combination of terms of the form
\[x^{\beta} \frac{d^{\alpha + \gamma}}{dx^{\alpha + \gamma}} x^n h_{\spec}\left(\frac{x}{2}\right) \ll \frac{x^{\beta}}{T^{4N}} e^{-\left(\frac{x - T}{U}\right)^2} x^{4N + n + 1} \left(\frac{|x - T|}{U^2} + \frac{1}{x}\right)^{\alpha + \gamma},\]
where $0 \leq \alpha \leq \frac{3}{7} (6N - 2j - 3)$, $0 \leq \beta \leq \alpha/3$, $0 \leq n \leq j \leq N$, and $n \leq \gamma \leq \frac{2}{21} (14N + 9j - 7)$. For $|x - T| \geq U \log T$, this decays faster than any power of $T$. If $|x - T| \leq U^2/T$, then we have the bound $O(T)$. Finally, for $U^2/T \leq |x - T| \leq U \log T$, the bound
\[O\left(T \left(\frac{T}{U}\right)^j \left(1 + \frac{|x - T|}{U}\right)^{4N} e^{-\left(\frac{x - T}{U}\right)^2}\right)\]
holds provided that $U \gg T^{1/3}$.

Finally, for $x \geq T^{13/12}$, we use \eqref{(A.1)} and \eqref{(A.3)} and split the integral at $|r| = x/3\pi$, which is readily seen to give
\[x^j \frac{d^j}{dx^j} \left(\Ks^{-} h_{T,U}\right)(x) \ll_j e^{-2\pi x} + x^{j + 2} e^{-\left(\frac{2\pi x}{5U}\right)^2},\]
as in \cite[Proof of Lemma 4]{BLM19}, which is more than sufficient.
\end{proof}

\begin{corollary}
\label{KshTUMellinboundcor}
For $-N/2 < \sigma < N/2$ and $j \in \N_0$ with $j \leq N/2$,
\[\widehat{\Ks^{-} h_{T,U}}(s) \ll_N UT^{\sigma} \left(\frac{T}{U(1 + |\tau|)}\right)^j.\]
\end{corollary}

\begin{proof}
We estimate the integral
\[\widehat{\Ks^{-} h_{T,U}}(s) = \int_{0}^{\infty} (\Ks^{-} h_{T,U})(x) x^s \, \frac{dx}{x}\]
by breaking this into the three ranges $(0,T - U\log T)$, $[T - U\log T, T + U\log T]$, and $(T + U\log T,\infty)$. We then estimate each of these ranges via integration by parts and \hyperref[BLMLemma4lemma]{Lemma \ref*{BLMLemma4lemma}}; the main contribution comes from the middle range.
\end{proof}

\begin{lemma}[{Cf.~\hyperref[Lspmintegrandboundslemma]{Lemma \ref*{Lspmintegrandboundslemma}}}]
\label{Lspmintegrandbounds2lemma}
Define
\begin{align*}
\Omega^{+}(\tau,t,0) & \defeq \begin{dcases*}
2|t| & if $|\tau| \leq 2|t|$,	\\
|\tau| & if $|\tau| \geq 2|t|$,
\end{dcases*}	\\
\Omega^{-}(\tau,t,0) & \defeq \begin{dcases*}
0 & if $|\tau| \leq 2|t|$,	\\
|\tau| - 2|t| & if $|\tau| \geq 2|t|$,
\end{dcases*}	\\
\Omega^{\hol}(\tau,k,0) & \defeq |\tau|.
\end{align*}
For $s = \sigma + i\tau$ with $-N/2 < \sigma < 1$ and $j \in \N_0$ with $j \leq N/2$, proved that additionally $s$ is at least a bounded distance away from $\{2(\pm it - n) : n \in \N_0\}$,
\begin{multline*}
\widehat{\Ks^{-} h_{T,U}}(s) \widehat{\JJ_t^{\pm}}(s) \GG_0^{\pm}(1 - s) \ll_{\sigma,j} UT^{\sigma} \left(\frac{T}{U(1 + |\tau|)}\right)^j (1 + |\tau|)^{-2\sigma}	\\
\times \left(\left(1 + \left|\tau + 2t\right|\right) \left(1 + \left|\tau - 2t\right|\right)\right)^{\frac{1}{2} (\sigma - 1)} e^{-\frac{\pi}{2} \Omega^{\pm}(\tau,t,0)},
\end{multline*}
and
\begin{align*}
\Res_{s = 2(\pm it - n)} \widehat{\Ks^{-} h_{T,U}}(s) \widehat{\JJ_t^{-}}(s) \GG_0^{-}(1 - s) & \ll_{n,j} UT^{-2n} \left(\frac{T}{U(1 + |t|)}\right)^j (1 + |t|)^{3n - \frac{1}{2}},	\\
\Res_{s = 2(\pm it - n)} \widehat{\Ks^{-} h_{T,U}}(s) \widehat{\JJ_t^{-}}(s) \GG_0^{-}(1 - s) & \ll_{n,j} UT^{-2n} \left(\frac{T}{U(1 + |t|)}\right)^j (1 + |t|)^{3n - \frac{1}{2}} e^{-\frac{\pi}{2} |t|}.
\end{align*}
For $s = \sigma + i\tau$ with $-N/2 < \sigma < 1$ and $j \in \N_0$ with $j \leq N/2$, proved that additionally $s$ is at least a bounded distance away from $\{1 - k - 2n : n \in \N_0\}$,
\[\widehat{\Ks^{-} h_{T,U}}(s) \widehat{\JJ_k^{\hol}}(s) \GG_0^{+}(1 - s) \ll_{\sigma,j} UT^{\sigma} \left(\frac{T}{U(1 + |\tau|)}\right)^j (1 + |\tau|)^{-2\sigma} \left(k + |\tau|\right)^{\sigma - 1} e^{-\frac{\pi}{2} \Omega^{\hol}(\tau,k,0)},\]
and
\[\Res_{s = 1 - k - 2n} \widehat{\Ks^{-} h_{T,U}}(s) \widehat{\JJ_k^{\hol}}(s) \GG_0^{+}(1 - s) \ll_n UT^{1 - k - 2n} \left(\frac{k - 1}{2\pi e}\right)^{1 - k} k^{-1/2}.\]
\end{lemma}

\begin{proof}
This follows via the same method as the proof of \hyperref[Lspmintegrandboundslemma]{Lemma \ref*{Lspmintegrandboundslemma}}, using \hyperref[KshTUMellinboundcor]{Corollary \ref*{KshTUMellinboundcor}} in place of \cite[Lemma 4]{BLM19}.
\end{proof}

\begin{corollary}[{Cf.~\hyperref[LHboundscor]{Corollary \ref*{LHboundscor}}}]
\label{LHbounds2cor}
For fixed $j \in \N_0$ with $j \leq N/2$,
\begin{align*}
\left(\Ls^{+} H_{T,U}^{+}\right)(t) & \ll_j U \left(\frac{T}{U(1 + |t|)}\right)^j (1 + |t|)^{-1/2} e^{-\frac{\pi}{2} |t|},	\\
\left(\Ls^{-} H_{T,U}^{-}\right)(t) & \ll_j U \left(\frac{T}{U(1 + |t|)}\right)^j (1 + |t|)^{-1/2},	\\
\intertext{while for fixed $-\min\{N/2,1 - k\} < \sigma < 1$,}
\left(\Ls^{\hol} H_{T,U}^{+}\right)(k) & \ll_{\sigma} UT^{\sigma} k^{\sigma - 1}.
\end{align*}
\end{corollary}

\begin{proof}
By Mellin inversion,
\begin{align*}
\left(\Ls^{\pm} H_{T,U}^{\pm}\right)(t) & = \frac{2}{\pi i} \int_{\sigma_1 - i\infty}^{\sigma_1 + i\infty} \widehat{\Ks^{-} h_{T,U}}(s) \widehat{\JJ_t^{\pm}}(s) \GG_0^{\pm}(1 - s) \, ds,	\\
\left(\Ls^{\hol} H_{T,U}^{+}\right)(k) & = \frac{2}{\pi i} \int_{\sigma_1 - i\infty}^{\sigma_1 + i\infty} \widehat{\Ks^{-} h_{T,U}}(s) \widehat{\JJ_k^{\hol}}(s) \GG_0^{+}(1 - s) \, ds,
\end{align*}
where $0 < \sigma_1 < 1$. As in the proof of \hyperref[LHboundscor]{Corollary \ref*{LHboundscor}}, we use \hyperref[Lspmintegrandbounds2lemma]{Lemma \ref*{Lspmintegrandbounds2lemma}} to bound these integrals. For $(\Ls^{\pm} H_{T,U}^{\pm})(t)$, we shift the contour from $\Re(s) = \sigma_1$ to $\Re(s) = \sigma$ with $-2 < \sigma < 0$, with the dominant contribution combing from the residues at the poles at $s = \pm 2it$. We do the same with $(\Ls^{\hol} H_{T,U}^{+})(k)$ with $-\min\{N/2,1 - k\} < \sigma < 1$; the dominant contribution of the ensuing integral comes from when $|\tau|$ is small.
\end{proof}

\begin{lemma}
\label{tildeGsizelemma}
We have that $\undertilde{\GG}(\hf) \ll_{\e} (TU)^{1 + \e}$.
\end{lemma}

\begin{proof}
Via Mellin inversion, we have that for $1/2 < \Re(s_1),\Re(s_2) < 1$ with $s_1 \neq s_2$,
\begin{multline}
\label{LsHcontoureq}
\left(\Ls^{+} H_{s_1,s_2,0}^{+}\right)(\pm i(s_1 - 1)) = \frac{2}{\pi i} \int_{\sigma_1 - i\infty}^{\sigma_1 + i\infty} \widehat{\Ks^{-} h}(s) \widehat{\JJ_{\pm i(s_1 - 1)}^{+}}(s + 2(s_1 + s_2 - 1))	\\
\times \left(\widehat{\JJ_0^{+}}(2 - s - 2s_1) \widehat{\JJ_0^{-}}(2 - s - 2s_2) + \widehat{\JJ_0^{-}}(2 - s - 2s_1) \widehat{\JJ_0^{+}}(2 - s - 2s_2)\right) \, ds,
\end{multline}
where $4(1 - \Re(s_1) - \Re(s_2)) + 2\max\{\Re(s_1),\Re(s_2)\} < \sigma < 2(1 - \max\{\Re(s_1),\Re(s_2)\})$. We shift the contour to $\Re(s) = \sigma_2$ with $\sigma_2$ slightly to the left of $4(1 - \Re(s_1) - \Re(s_2)) + 2\max\{\Re(s_1),\Re(s_2)\}$, picking up a residue at $s = 4 - 4s_1 - 2s_2$ equal to
\begin{multline*}
4 \widehat{\Ks^{-} h}(2(2 - 2s_1 - s_2)) (2\pi)^{2(s_1 - 1)} \Gamma(2(1 - s_1)) \cos \pi(s_1 - 1)	\\
\times \left(\widehat{\JJ_0^{+}}(2(s_1 + s_2 - 1)) \widehat{\JJ_0^{-}}(2(2s_1 - 1)) + \widehat{\JJ_0^{-}}(2(s_1 + s_2 - 1)) \widehat{\JJ_0^{+}}(2(s_1 - 1))\right).
\end{multline*}
Similar calculations hold for the terms $(\Ls^{+} H_{s_1,s_2,0}^{+})(\pm i(s_2 - 1))$, $(\Ls^{-} H_{s_1,s_2,0}^{-})(\pm i(s_1 - 1))$, and $(\Ls^{-} H_{s_1,s_2,0}^{-})(\pm i(s_2 - 1))$.

Now we let $s_1 = 1/2$ and consider the Laurent expansions about $s_2 = 1/2$ of $\undertilde{\NN}(1/2,s_2;\hf)$, $-\undertilde{\RR}(1/2,s_2;\hf)$, and $\undertilde{\RR}(1/2,s_2;\Ts_{1/2,s_2,0}^{\pm} \hf)$. Since $\undertilde{\GG}(1/2,s_2;\hf)$ is holomorphic at $s_2 = 1/2$, the principal parts must sum to zero, and so it suffices to bound the constant term in each Laurent expansion. For $\undertilde{\RR}(1/2,s_2;\Ts_{1/2,s_2,0}^{\pm} \hf)$, we use \hyperref[KshTUMellinboundcor]{Corollary \ref*{KshTUMellinboundcor}} to bound \eqref{LsHcontoureq} with $\sigma_1$ replaced by $\sigma_2 \in (0,1)$. For the remaining terms, it is readily seen that the dominant contribution is bounded by a constant multiple dependent on $D$ of
\begin{align*}
\left|\widehat{\Ks^{-} h_{T,U}}''(1)\right| & \leq \int_{-\infty}^{\infty} \left|\widehat{\JJ_r^{-}}''(1)\right| h_{T,U}(r) \, d_{\spec} r	\\
& \ll \int_{-\infty}^{\infty} (1 + |r|) (\log (1 + |r|))^2 h_{T,U}(r) \, dr	\\
& \ll_{\e} (TU)^{1 + \e}.
\qedhere
\end{align*}
\end{proof}

\section{Proof of \texorpdfstring{\hyperref[dihedralmomentsprop]{Proposition \ref*{dihedralmomentsprop} (3)}}{Proposition \ref{dihedralmomentsprop} (3)}: the Short Transition Range}
\label{(3)sect}

\begin{proof}[Proof of {\hyperref[dihedralmomentsprop]{Proposition \ref*{dihedralmomentsprop} (3)}}]
Via the approximate functional equation, \hyperref[approxfunclemma]{Lemma \ref*{approxfunclemma}}, and the large sieve, \hyperref[largesievethm]{Theorem \ref*{largesievethm}},
\begin{multline*}
\sum_{d_1 d_2 = D} 2^{\omega(d_2)} \frac{\varphi(d_2)}{d_2} \sum_{\substack{f \in \BB_0^{\ast}(\Gamma_0(d_1)) \\ T - U \leq t_f \leq T + U}} \frac{L\left(\frac{1}{2},f \otimes g_{\psi^2}\right)^2}{L_{d_2}\left(\frac{1}{2}, f\right) L^{d_2}(1,\sym^2 f)}	\\
+ \frac{2^{\omega(D)}}{2\pi} \int\limits_{T - U \leq |t| \leq T + U} \left|\frac{L\left(\frac{1}{2} + it, g_{\psi^2}\right)^2}{\zeta_D\left(\frac{1}{2} + it\right) \zeta^D(1 + 2it)}\right|^2 \, dt \ll_{\e} (TU)^{1 + \e}
\end{multline*}
for $1 + |2t_g - T| \ll U \leq T \ll t_g$. Next, we claim that
\begin{multline}
\label{f,fchiDmomentboundeq}
\sum_{d_1 d_2 = D} 2^{\omega(d_2)} \frac{\varphi(d_2)}{d_2} \sum_{\substack{f \in \BB_0^{\ast}(\Gamma_0(d_1)) \\ T - U \leq t_f \leq T + U}} \frac{L\left(\frac{1}{2},f\right)^2 L\left(\frac{1}{2},f \otimes \chi_D\right)^2}{L_{d_2}\left(\frac{1}{2}, f\right) L^{d_2}(1,\sym^2 f)}	\\
+ \frac{2^{\omega(D)}}{2\pi} \int\limits_{T - U \leq |t| \leq T + U} \left|\frac{\zeta\left(\frac{1}{2} + it\right)^2 L\left(\frac{1}{2} + it, \chi_D\right)^2}{\zeta_D\left(\frac{1}{2} + it\right) \zeta^D(1 + 2it)}\right|^2 \, dt \ll_{\e} (TU)^{1 + \e}
\end{multline}
for $T^{1/3} \ll U \leq T$. To see this, we use \hyperref[spectralreciprocity2prop]{Proposition \ref*{spectralreciprocity2prop}} with $\hf = (h_{T,U},0)$, where $h_{T,U}$ is as in \eqref{hTUeq}. \hyperref[tildeGsizelemma]{Lemma \ref*{tildeGsizelemma}} shows that $\undertilde{\GG}(\hf) \ll_{\e} (TU)^{1 + \e}$. For $\undertilde{\MM}^{\pm} (\Ts_0^{\pm} \hf)$, we break up each term into dyadic intervals and use \hyperref[LHbounds2cor]{Corollary \ref*{LHbounds2cor}} to bound $(\Ls^{\pm} H_{T,U}^{\pm})(t)$ and $(\Ls^{\hol} H_{T,U}^{+})(t)$ and the approximate functional equation and large sieve to bound each spectral sum of $L$-functions. The largest contributions come from $\undertilde{\MM}^{\Maass,-}(\Ls^{-} H_{T,U}^{-})$ when $t_f \asymp T/U$ and $\undertilde{\MM}^{\Eis}(\Ls^{-} H_{T,U}^{-})$ when $|t| \asymp T/U$, which give terms of size $O_{\e}(T^{3/2 + \e} U^{-1/2 + \e})$. Since $U \geq T^{1/3}$, this is $O_{\e}((TU)^{1 + \e})$.

The result now follows from the Cauchy--Schwarz inequality.
\end{proof}

\begin{remark}
\label{Weylsubconvexityremark}
Taking $U = T^{1/3}$ and dropping all but one term in \eqref{f,fchiDmomentboundeq} implies that
\begin{equation}
\label{f,fchiDsubconvexeq}
\begin{split}
L\left(\frac{1}{2},f\right) L\left(\frac{1}{2},f \otimes \chi_D\right) & \ll_{\e} D^{\frac{3}{4} + \e} t_f^{\frac{2}{3} + \e},	\\
\left|\zeta\left(\frac{1}{2} + it\right) L\left(\frac{1}{2} + it, \chi_D\right)\right|^2 & \ll_{\e} D^{\frac{3}{4} + \e} |t|^{\frac{2}{3} + \e}
\end{split}
\end{equation}
for $f \in \BB_0^{\ast}(\Gamma_0(d_1))$ and $t \in \R$, where we have additionally kept track of the $D$-dependence. This is a Weyl-strength subconvex bound in the $t_f$- and $t$-aspects and a convex bound in the $D$-aspect. For $D = 1$, \eqref{f,fchiDmomentboundeq} and its corollary \eqref{f,fchiDsubconvexeq} are results of Jutila \cite[Theorem]{Jut01}; the proof is not wholly dissimilar, though it is perhaps slightly less direct, for it passes through the spectral decomposition of shifted convolution sums.
\end{remark}

\section{Proof of \texorpdfstring{\hyperref[dihedralmomentsprop]{Proposition \ref*{dihedralmomentsprop} (4)}}{Proposition \ref{dihedralmomentsprop} (4)}: the Tail Range}
\label{(4)sect}

\begin{proof}[{Proof of {\hyperref[dihedralmomentsprop]{Proposition \ref*{dihedralmomentsprop} (4)}}}]
This follow simply via the Cauchy--Schwarz inequality, the approximate functional equation, \hyperref[approxfunclemma]{Lemma \ref*{approxfunclemma}}, and the large sieve, \hyperref[largesievethm]{Theorem \ref*{largesievethm}}.
\end{proof}

\section{Proof of \texorpdfstring{\hyperref[dihedralmomentsprop]{Proposition \ref*{dihedralmomentsprop} (5)}}{Proposition \ref{dihedralmomentsprop} (5)}: the Exceptional Range}
\label{(5)sect}

\begin{proof}[Proof of {\hyperref[dihedralmomentsprop]{Proposition \ref*{dihedralmomentsprop} (5)}}]
This follows directly from the subconvex bounds in \hyperref[Youngsubconvthm]{Theorems \ref*{Youngsubconvthm}} and \ref{MVsubconvthm}, noting that there are only finitely many exceptional eigenvalues (and conjecturally none).
\end{proof}

\appendix

\section{Automorphic Machinery}
\label{toolboxappendix}

In this appendix, we detail the many tools that are used in the course of proving \hyperref[dihedralmomentsprop]{Proposition \ref*{dihedralmomentsprop}}. These are the following: the explicit relation between dihedral Maa\ss{} newforms and Hecke Gr\"{o}\ss{}encharaktere; several root number calculations; the approximate functional equation; explicit forms of the Kuznetsov, Petersson, Kloosterman, and Vorono\u{\i} summation formul\ae{}; details on Mellin transforms of certain functions arising in the aforementioned summation formul\ae{}; the large sieve; and pre-existing subconvexity estimates for certain $L$-functions.

\subsection{Dihedral Maa\ss{} Newforms and Hecke Gr\"{o}\ss{}encharaktere}
\label{Grosssect}

Let $D \equiv 1 \pmod{4}$ be a positive squarefree fundamental discriminant of a real quadratic field $K = \Q(\sqrt{D})$ with ring of integers $\OO_K$ and let $\chi_D$ be the quadratic character modulo $D$ associated to the extension $K/\Q$ via class field theory. We record here the fact that the Gauss sum $\tau(\chi_D)$ of $\chi_D$ is equal to $\sqrt{D}$.

The Hecke Gr\"{o}\ss{}encharaktere $\psi$ of conductor $\OO_K$ satisfy
\[\psi((\alpha)) = \sgn(\alpha \sigma(\alpha))^{\kappa} \left|\frac{\alpha}{\sigma(\alpha)}\right|^{\frac{\pi i \ell}{\log \epsilon_K}}\]
for every principal ideal $(\alpha)$ of $\OO_K$, with $\ell \in \Z$ and $\kappa \in \{0,1\}$ subject to the restriction that $\kappa = 0$ if $\epsilon_K \sigma(\epsilon_K) = -1$, where $\sigma$ denotes the nontrivial element of $\Gal(K/\Q)$ and $\epsilon_K > 0$ is the fundamental unit of $K$. Moreover, every Hecke Gr\"{o}\ss{}encharakter is determined by $\ell$, $\kappa$, and a class group character, and such a Hecke Gr\"{o}\ss{}encharakter does not factor through the norm map $N_{K/\Q}$ if and only if either $\ell$ is positive or the class group character associated to $\psi$ is complex.

A dihedral Maa\ss{} newform $g = g_{\psi}$ is the automorphic induction of a Hecke Gr\"{o}\ss{}encharakter $\psi$ of $K$ for which $\psi$ does not factor through the norm map $N_{K/\Q}$. When $\psi$ has conductor $\OO_K$, $g_{\psi}$ is an element of $\BB_0^{\ast}(D,\chi_D)$ whose Fourier expansion about the cusp at infinity is given by
\begin{align*}
g_{\psi}(z) & = \sum_{\substack{n = -\infty \\ n \neq 0}}^{\infty} \rho_{g_{\psi}}(n) W_{0,it_g}(4\pi |n| y) e(nx)	\\
& = \rho_{g_{\psi}}(1) \sum_{\substack{\aa \subset \OO_K \\ \aa \neq \{0\}}} \frac{\psi(\aa)}{\sqrt{N(\aa)}} W_{0,it_g} \left(4\pi N(\aa) y\right) \left(e\left(N(\aa) x\right) + (-1)^{\kappa} e\left(-N(\aa) x\right)\right),
\end{align*}
where
\[t_g = \frac{\pi |\ell|}{\log \epsilon_K}, \qquad \rho_{g_{\psi}}(n) = \sgn(n)^{\kappa} \lambda_{g_{\psi}}(|n|) \frac{\rho_{g_{\psi}}(1)}{\sqrt{|n|}}, \qquad \lambda_{g_{\psi}}(n) = \sum_{N(\aa) = n} \psi(\aa),\]
and $N(\aa) \defeq \# \OO_K / \aa$ denotes the absolute norm of a nonzero ideal $\aa \subset \OO_K$. Note that $\rho_{g_{\psi}}(-n) = (-1)^{\kappa} \rho_{g_{\psi}}(n)$; that is, $(-1)^{\kappa}$ is the parity of $g_{\psi}$. In particular, $g_{\psi}$ is even if $\psi$ is the square of another Hecke Gr\"{o}\ss{}encharakter.

The Satake parameters $\alpha_{g_{\psi}}(p),\beta_{g_{\psi}}(p)$ of $g_{\psi}$ at a prime $p$ are related to the Hecke eigenvalue $\lambda_{g_{\psi}}(p)$ and nebentypus $\chi_D(p)$ via
\[\alpha_{g_{\psi}}(p) + \beta_{g_{\psi}}(p) = \lambda_{g_{\psi}}(p), \qquad \alpha_{g_{\psi}}(p) \beta_{g_{\psi}}(p) = \chi_D(p).\]
The relationship between the Satake parameters of $g_{\psi}$ at a prime $p$ and the values of the Hecke Gr\"{o}\ss{}encharakter $\psi$ on prime ideals $\pp \mid p \OO_K$ is as follows:
\begin{itemize}
\item If $\chi_D(p) = 1$, then $p$ splits in $K$, so that $p \OO_K = \pp \sigma(\pp)$, and its Satake parameters are $\alpha_f(p) = \psi(\pp)$ and $\beta_f(p) = \psi(\sigma(\pp)) = \overline{\psi}(\pp)$.
\item If $\chi_D(p) = -1$, then $p$ is inert in $K$, so that $p \OO_K = \pp$, and $\alpha_{g_{\psi}}(p) = -\beta_{g_{\psi}}(p) = 1$.
\item If $\chi_D(p) = 0$, then $p$ ramifies in $K$, so that $p \OO_K = \pp^2$, and $\alpha_{g_{\psi}}(p) = \psi(\pp)$ while $\beta_{g_{\psi}}(p) = 0$.
\end{itemize}
In all cases, $|\alpha_{g_{\psi}}(p)| = 1$. We record the following useful consequences.

\begin{lemma}
\label{dihedralHeckeeigenlemma}
The Hecke eigenvalues of a dihedral newform $g_{\psi} \in \BB_0^{\ast}(D,\chi_D)$ satisfy $\lambda_{g_{\psi}}(p) \in [-2,2]$ and $\lambda_{g_{\psi}}(n) \in \{\pm 1\}$ when $n \mid D^{\infty}$; moreover, $\lambda_{g_{\psi}}(n) = 1$ when $n \mid D^{\infty}$ if $g_{\psi}$ is even. We have that $g_{\psi} \otimes \chi_D = \overline{g_{\psi}} = g_{\psi}$ and $\lambda_{g_{\psi}}(n) \chi_D(n) = \delta_{(D,n),1} \lambda_{g_{\psi}}(n)$.
\end{lemma}

\subsection{Root Number Calculations}

Since \hyperref[dihedralmomentsprop]{Proposition \ref*{dihedralmomentsprop}} involves moments of $L$-functions of level greater than $1$, we must explicitly determine the root numbers and conductors of these $L$-functions in order to precisely utilise the approximate functional equation.

Recall that the Atkin--Lehner pseudo-eigenvalue $\eta_f(w)$ of $f \in \BB_0^{\ast}(\Gamma_0(q))$ with $w \mid q$ is independent of the choice of integer entries $a,b,c,d \in \Z$ in the definition of the Atkin--Lehner operator $W_w$ provided that $\det W_w = 1$ (cf.~\hyperref[AtkinLehnersect]{Section \ref*{AtkinLehnersect}}).

\begin{lemma}[Cf.~{\cite[Section 2.3]{HT14}}]
\label{rootnumberlemma}
Let $f$ be either a member of $\BB_0^{\ast}(\Gamma_0(d_1))$ or $\BB_{\hol}^{\ast}(\Gamma_0(d_1))$ with $d_1 d_2 = D$. Then the conductors and root numbers of $f$, $f \otimes \chi_D$, and $f \otimes g_{\psi^2}$ are given by
\begin{align*}
q(f) = d_1, \qquad & \epsilon(f) = \begin{dcases*}
\epsilon_f \eta_f(d_1) & if $f \in \BB_0^{\ast}(\Gamma_0(d_1))$,	\\
i^{k_f} \eta_f(d_1) & if $f \in \BB_{\hol}^{\ast}(\Gamma_0(d_1))$,
\end{dcases*}	\\
q(f \otimes \chi_D) = D^2, \qquad & \epsilon(f \otimes \chi_D) = \begin{dcases*}
\epsilon_f & if $f \in \BB_0^{\ast}(\Gamma_0(d_1))$,	\\
i^{k_f} & if $f \in \BB_{\hol}^{\ast}(\Gamma_0(d_1))$,
\end{dcases*}	\\
q(f \otimes g_{\psi^2}) = D^2 d_1, \qquad & \epsilon(f \otimes g_{\psi^2}) = \eta_f(d_1).
\end{align*}
\end{lemma}

\begin{proof}
This follows by a local argument studying the local components of $\pi_f$, $\pi_f \otimes \omega_D$, and $\pi_f \otimes \pi_{g_{\psi^2}}$, where $\pi_f, \pi_{g_{\psi^2}}$ are the cuspidal automorphic representations of $\GL_2(\A_{\Q})$ associated to the newforms $f, g_{\psi^2}$ and $\omega_D$ is the Hecke character of $\Q^{\times} \backslash \A_{\Q}^{\times}$ that is the id\`{e}lic lift of $\chi_D$. We give only the proof for the root number and conductor of $f \otimes g_{\psi^2}$, for the other two cases are similar but simpler.
\begin{itemize}
\item At the archimedean place, $f \in \BB_0^{\ast}(\Gamma_0(d_1))$ implies that $\pi_{f,\infty}$ is a principal series representation $\sgn^{\kappa_f} |\cdot|^{it_f} \boxplus \sgn^{\kappa_f} |\cdot|^{-it_f}$ and $\pi_{g_{\psi^2},\infty} = |\cdot|^{2it_g} \boxplus |\cdot|^{-2it_g}$, where $\kappa_f$ is zero if $f$ is even and one if $f$ is odd, and so
\[\pi_{f,\infty} \otimes \pi_{g_{\psi^2},\infty} = \bigboxplus_{\pm_1} \bigboxplus_{\pm_2} \sgn^{\kappa_f} |\cdot|^{\pm_1 i(t_f \pm_2 2t_g)}.\]
The local epsilon factor $\epsilon(s,\pi_{f,\infty} \otimes \pi_{g_{\psi^2},\infty},\psi_{\infty})$ is $i^{4\kappa_f} = 1$. Similarly, $f \in \BB_{\hol}^{\ast}(\Gamma_0(d_1))$ implies that $\pi_{f,\infty} = D_{k_f}$ where $k_f \in 2\N$ is the weight of $f$ and $D_k$ is the discrete series representation of weight $k$. Then
\[\pi_{f,\infty} \otimes \pi_{g_{\psi^2},\infty} = D_{k_f} \left|\det\right|^{2it_g} \boxplus D_{k_f} \left|\det\right|^{-2it_g}.\]
The local epsilon factor is $i^{2k_f} = 1$.
\item At a prime $p \mid d_1$, $\pi_{f,p}$ is a special representation $\omega_{f,p} \St_p$, where $\omega_{f,p}$ is either trivial or the unramified quadratic character, and $\pi_{g_{\psi^2},p} = \omega_{g_{\psi^2},p} \boxplus 1$, where $\omega_{g_{\psi^2},p}$ is the local component of $\chi_D$ (and hence a ramified character of $\Q_p^{\times}$ of conductor exponent $c(\omega_{g_{\psi^2},p}) = 1$). It follows that
\[\pi_{f,p} \otimes \pi_{g_{\psi^2},p} = \omega_{g_{\psi^2},p} \omega_{f,p} \St_p \boxplus \omega_{f,p} \St_p,\]
and so the local conductor exponent $c(\pi_{f,p} \otimes \pi_{g_{\psi^2},p})$ is
\[c\left(\omega_{g_{\psi^2},p} \omega_{f,p} \St_p\right) + c\left(\omega_{f,p} \St_p\right) = 2 + 1 = 3,\]
while the local epsilon factor $\epsilon(s,\pi_{f,p} \otimes \pi_{g_{\psi^2},p},\psi_p)$ is equal to
\[\epsilon\left(s,\omega_{g_{\psi^2},p} \omega_{f,p} \St_p,\psi_p\right) \epsilon\left(s,\omega_{f,p} \St_p,\psi_p\right) = -\omega_{f,p}(p) \epsilon\left(\frac{1}{2},\omega_{g_{\psi^2},p},\psi_p\right)^2 p^{-3\left(s - \frac{1}{2}\right)},\]
and $\epsilon(1/2,\omega_{g_{\psi^2},p},\psi_p)$ is $\tau(\chi_p) p^{-1/2}$, where $\chi_p$ is the quadratic character modulo $p$, while $\omega_{f,p}(p)$ is $\lambda_f(p) \sqrt{p}$.
\item At a prime $p \mid d_2$, $\pi_{f,p} = \omega_{f,p} \boxplus \omega_{f,p}^{-1}$, where both characters are unramified, and $\pi_{g_{\psi^2},p} = \omega_{g_{\psi^2},p} \boxplus 1$, where $\omega_{g_{\psi^2},p}$ is the local component of $\chi_D$. It follows that
\[\pi_{f,p} \otimes \pi_{g_{\psi^2},p} = \omega_{g_{\psi^2},p} \omega_{f,p} \boxplus \omega_{g_{\psi^2},p} \omega_{f,p}^{-1} \boxplus \omega_{f,p} \boxplus \omega_{f,p}^{-1},\]
and so the local conductor exponent $c(\pi_{f,p} \otimes \pi_{g_{\psi^2},p})$ is
\[c\left(\omega_{g_{\psi^2},p} \omega_{f,p}\right) + c\left(\omega_{g_{\psi^2},p} \omega_{f,p}^{-1}\right) + c\left(\omega_{f,p}\right) + c\left(\omega_{f,p}^{-1}\right) = 1 + 1 + 0 + 0 = 2,\]
while the local root number $\epsilon(s,\pi_{f,p} \otimes \pi_{g_{\psi^2},p},\psi_p)$ is equal to
\[\epsilon\left(s,\omega_{g_{\psi^2},p} \omega_{f,p}, \psi_p\right) \epsilon\left(s,\omega_{g_{\psi^2},p} \omega_{f,p}^{-1}, \psi_p\right) \epsilon\left(s,\omega_{f,p}, \psi_p\right) \epsilon\left(s,\omega_{f,p}^{-1}, \psi_p\right) = \epsilon\left(\frac{1}{2},\omega_{g_{\psi^2},p},\psi_p\right)^2 p^{-2\left(s - \frac{1}{2}\right)},\]
and again $\epsilon(1/2,\omega_{g_{\psi^2},p},\psi_p)$ is $\tau(\chi_p) p^{-1/2}$.
\item At a prime $p \nmid D$, both $\pi_{f,p}$ and $\pi_{g_{\psi^2},p}$ are spherical principal series representations, so that $c(\pi_{f,p} \otimes \pi_{g_{\psi^2},p}) = 0$ and $\epsilon(s,\pi_{f,p} \otimes \pi_{g_{\psi^2},p},\psi_p) = 1$.
\end{itemize}
With this, we see that
\[q(f \otimes g_{\psi^2}) = \prod_p p^{c(\pi_{f,p} \otimes \pi_{g_{\psi^2},p})} = \prod_{p \mid d_1} p^3 \prod_{p \mid d_2} p^2 = D^2 d_1,\]
while the fact that
\[\tau(\chi_p) = \begin{dcases*}
\sqrt{p} & if $p \equiv 1 \pmod{4}$,	\\
i\sqrt{p} & if $p \equiv 3 \pmod{4}$,
\end{dcases*}\]
and $D$ being $1$ modulo $4$ ensuring that it has an even number of prime divisors that are $3$ modulo $4$ implies that the root number $\epsilon(f \otimes g_{\psi^2}) = \epsilon(1/2,\pi_f \otimes \pi_{g_{\psi^2}})$ is
\[\epsilon\left(\frac{1}{2},\pi_{f,\infty} \otimes \pi_{g_{\psi^2},\infty},\psi_{\infty}\right) \prod_p \epsilon\left(\frac{1}{2},\pi_{f,p} \otimes \pi_{g_{\psi^2},p},\psi_p\right) = \mu(d_1) \lambda_f(d_1) \sqrt{d_1}.\]
As $\tau(\chi_{0(d_1)}) = \mu(d_1)$ and $\lambda_f(d_1) \sqrt{d_1} \in \{1,-1\}$, this is precisely $\eta_f(d_1)$.
\end{proof}

\subsection{The Approximate Functional Equation}

First, we recall some standard identities for writing Rankin--Selberg $L$-functions as Dirichlet series. Let $\chi$ be an even primitive character modulo $q$ with $q > 1$, and denote by $E_{\infty}(z,s,\chi)$ the Eisenstein series of weight $0$, level $q$, and nebentypus $\chi$ associated to the cusp at infinity, which is given by
\[E_{\infty}(z,s,\chi) \defeq \sum_{\gamma \in \Gamma_{\infty} \backslash \Gamma_0(q)} \overline{\chi}(\gamma) \Im(\gamma z)^s\]
for $\Re(s) > 1$ and extends by meromorphic continuation to the entire complex plane. In particular, $E_{\infty}(z,1/2 + it,\chi)$ is an Eisenstein series newform \cite{You19} with Hecke eigenvalues
\begin{equation}
\label{lambdachi1eq}
\lambda_{\chi,1}(m,t) \defeq \sum_{ab = m} \chi(a) a^{it} b^{-it}.
\end{equation}

\begin{lemma}
\label{Ramanujanlemma}
For $f$ either in $\BB_0^{\ast}(\Gamma_0(d_1))$ or $\BB_{\hol}^{\ast}(\Gamma_0(d_1))$ with $d_1 \mid D$ and $t \in \R$, we have the identities
\begin{align*}
L(s,f) L(s,f \otimes \chi_D) & = L(2s,\chi_D) \sum_{m = 1}^{\infty} \frac{\lambda_f(m) \lambda_{\chi_D,1}(m,0)}{m^s},	\\
L(s,f \otimes g_{\psi^2}) & = L(2s,\chi_D) \sum_{n = 1}^{\infty} \frac{\lambda_f(n) \lambda_{g_{\psi^2}}(n)}{n^s},	\\
\zeta(s + it) \zeta(s - it) L(s + it, \chi_D) L(s - it, \chi_D) & = L(2s,\chi_D) \sum_{m = 1}^{\infty} \frac{\lambda(m,t) \lambda_{\chi_D,1}(m,0)}{m^s},	\\
L(s + it, g_{\psi^2}) L(s - it, g_{\psi^2}) & = L(2s,\chi_D) \sum_{n = 1}^{\infty} \frac{\lambda(n,t) \lambda_{g_{\psi^2}}(n)}{n^s}
\end{align*}
for $\Re(s) > 1$.
\end{lemma}

\begin{lemma}
\label{approxfunclemma}
Fix $X > 0$. For $f \in \BB_0^{\ast}(\Gamma_0(d_1))$ and $t \in \R \setminus \{0\}$, we have that
\begin{align*}
L\left(\frac{1}{2},f\right) L\left(\frac{1}{2},f \otimes \chi_D\right) & = \sum_{m = 1}^{\infty} \sum_{k = 1}^{\infty} \frac{\lambda_f(m) \lambda_{\chi_D,1}(m,0) \chi_D(k)}{\sqrt{m} k} V_1^{\epsilon_f}\left(\frac{mk^2}{X D\sqrt{d_1}},t_f\right)	\\
& \quad + \eta_f(d_1) \sum_{m = 1}^{\infty} \sum_{k = 1}^{\infty} \frac{\lambda_f(m) \lambda_{\chi_D,1}(m,0) \chi_D(k)}{\sqrt{m} k} V_1^{\epsilon_f}\left(\frac{mk^2 X}{D\sqrt{d_1}},t_f\right),	\\
L\left(\frac{1}{2},f \otimes g_{\psi^2}\right) & = \sum_{n = 1}^{\infty} \sum_{\ell = 1}^{\infty} \frac{\lambda_f(n) \lambda_{g_{\psi^2}}(n) \chi_D(\ell)}{\sqrt{n} \ell} V_2^{\epsilon_f}\left(\frac{n\ell^2}{X D\sqrt{d_1}},t_f\right)	\\
& \quad + \eta_f(d_1) \sum_{n = 1}^{\infty} \sum_{\ell = 1}^{\infty} \frac{\lambda_f(n) \lambda_{g_{\psi^2}}(n) \chi_D(\ell)}{\sqrt{n} \ell} V_2^{\epsilon_f}\left(\frac{n\ell^2 X}{D\sqrt{d_1}},t_f\right),	\\
\left|\zeta\left(\frac{1}{2} + it\right) L\left(\frac{1}{2} + it,\chi_D\right)\right|^2 & = \sum_{m = 1}^{\infty} \sum_{k = 1}^{\infty} \frac{\lambda(m,t) \lambda_{\chi_D,1}(m,0) \chi_D(k)}{\sqrt{m} k} V_1^1\left(\frac{mk^2}{X D},t\right)	\\
& \quad + \sum_{m = 1}^{\infty} \sum_{k = 1}^{\infty} \frac{\lambda(m,t) \lambda_{\chi_D,1}(m,0) \chi_D(k)}{\sqrt{m} k} V_1^1\left(\frac{mk^2 X}{D},t\right)	\\
& \qquad + R(X,D,t),	\\
\left|L\left(\frac{1}{2} + it,g_{\psi^2}\right)\right|^2 & = \sum_{n = 1}^{\infty} \sum_{\ell = 1}^{\infty} \frac{\lambda(n,t) \lambda_{g_{\psi^2}}(n) \chi_D(\ell)}{\sqrt{n} \ell} V_2^1\left(\frac{n\ell^2}{X D\sqrt{d_1}},t\right)	\\
& \quad + \sum_{n = 1}^{\infty} \sum_{\ell = 1}^{\infty} \frac{\lambda(n,t) \lambda_{g_{\psi^2}}(n) \chi_D(\ell)}{\sqrt{n} \ell} V_2^1\left(\frac{n\ell^2 X}{D\sqrt{d_1}},t\right),
\end{align*}
where for $\Gamma_{\R}(s) \defeq \pi^{-s/2} \Gamma(s/2)$,
\begin{multline*}
R(X,D,t) \defeq 2\Re\left(e^{\left(\frac{1}{2} + it\right)^2} \left((X\sqrt{D})^{\frac{1}{2} + it} + \left(\frac{\sqrt{D}}{X}\right)^{\frac{1}{2} + it}\right) \vphantom{e^{\left(\frac{1}{2} + it\right)^2} \left((X\sqrt{D})^{\frac{1}{2} + it} + \left(\frac{\sqrt{D}}{X}\right)^{\frac{1}{2} + it}\right) \left(\frac{\Gamma_{\R}\left(1 + 2it\right)}{\Gamma_{\R}\left(\frac{1}{2} + it\right) \Gamma_{\R}\left(\frac{1}{2} - it\right)}\right)^2 \zeta(1 + 2it) L(1 + 2it) L(1,\chi_D)}	\right.	\\
\left. \vphantom{e^{\left(\frac{1}{2} + it\right)^2} \left((X\sqrt{D})^{\frac{1}{2} + it} + \left(\frac{\sqrt{D}}{X}\right)^{\frac{1}{2} + it}\right) \left(\frac{\Gamma_{\R}\left(1 + 2it\right)}{\Gamma_{\R}\left(\frac{1}{2} + it\right) \Gamma_{\R}\left(\frac{1}{2} - it\right)}\right)^2 \zeta(1 + 2it) L(1 + 2it) L(1,\chi_D)} \times \left(\frac{\Gamma_{\R}\left(1 + 2it\right)}{\Gamma_{\R}\left(\frac{1}{2} + it\right) \Gamma_{\R}\left(\frac{1}{2} - it\right)}\right)^2 \zeta(1 + 2it) L(1 + 2it) L(1,\chi_D)\right),
\end{multline*}
and for $x,\sigma > 0$, $t \in \R$, and $\epsilon \in \{1,-1\}$,
\begin{align}
\label{V1eq}
V_1^{\epsilon}(x,t) & \defeq \frac{1}{2\pi i} \int_{\sigma - i\infty}^{\sigma + i\infty} e^{s^2} x^{-s} \left(\prod_{\pm} \frac{\Gamma_{\R}\left(1 - \frac{\epsilon}{2} + s \pm it\right)}{\Gamma_{\R}\left(1 - \frac{\epsilon}{2} \pm it\right)}\right)^2 \, \frac{ds}{s},	\\
\label{V2eq}
V_2^{\epsilon}(x,t) & \defeq \frac{1}{2\pi i} \int_{\sigma - i\infty}^{\sigma + i\infty} e^{s^2} x^{-s} \prod_{\pm_1} \prod_{\pm_2} \frac{\Gamma_{\R}\left(1 - \frac{\epsilon}{2} + s \pm_1 i(2t_g \pm_2 t)\right)}{\Gamma_{\R}\left(1 - \frac{\epsilon}{2} \pm_1 i(2t_g \pm_2 t)\right)} \, \frac{ds}{s}.
\end{align}
Finally,
\begin{multline*}
L\left(\frac{1}{2},f \otimes g_{\psi^2}\right) = \sum_{n = 1}^{\infty} \sum_{\ell = 1}^{\infty} \frac{\lambda_f(n) \lambda_{g_{\psi^2}}(n) \chi_D(\ell)}{\sqrt{n} \ell} V_2^{\hol}\left(\frac{n\ell^2}{X D\sqrt{d_1}},k_f\right)	\\
+ \eta_f(d_1) \sum_{n = 1}^{\infty} \sum_{\ell = 1}^{\infty} \frac{\lambda_f(n) \lambda_{g_{\psi^2}}(n) \chi_D(\ell)}{\sqrt{n} \ell} V_2^{\hol}\left(\frac{n\ell^2 X}{D\sqrt{d_1}},k_f\right)
\end{multline*}
for $f \in \BB_{\hol}^{\ast}(\Gamma_0(d_1))$, where
\[V_2^{\hol}(x,k) \defeq \frac{1}{2\pi i} \int_{\sigma - i\infty}^{\sigma + i\infty} e^{s^2} x^{-s} \prod_{\pm_1} \prod_{\pm_2} \frac{\Gamma_{\R}\left(s + \frac{k \pm_1 1}{2} \pm_2 2it_g\right)}{\Gamma_{\R}\left(\frac{1}{2} + \frac{k \pm_1 1}{2} \pm_2 2it_g\right)} \, \frac{ds}{s}.\]
\end{lemma}

\begin{proof}
This follows from \cite[Theorem 5.3]{IK04} coupled with \hyperref[rootnumberlemma]{Lemmata \ref*{rootnumberlemma}} and \ref{Ramanujanlemma}.
\end{proof}

We briefly mention the fact that \cite[Proposition 5.4]{IK04} implies that the functions $V(x,\cdot)$ appearing in \hyperref[approxfunclemma]{Lemma \ref*{approxfunclemma}} are of rapid decay in $x$ once $x$ is much larger than the square root of the archimedean part of the analytic conductor of the associated $L$-function.

\subsection{Explicit Expressions for Spectral Sums}

For a function $h : \R \cup i(-1/2,1/2)$ and $m,n,q \in \N$, define
\begin{align*}
\AA_q^{\Maass}(m,\pm n;h) & \defeq 2q \sum_{f \in \BB_0(\Gamma_0(q))} \frac{\sqrt{mn} \overline{\rho_f}(m) \rho_f(\pm n)}{\cosh \pi t_f} h(t_f),	\\
\AA_q^{\Eis}(m,\pm n;h) & \defeq \frac{q}{2\pi} \sum_{\aa} \int_{-\infty}^{\infty} \frac{\sqrt{mn} \overline{\rho_{\aa}}(m,t) \rho_{\aa}(\pm n,t)}{\cosh \pi t} h(t) \, dt,
\end{align*}
where $\BB_0(\Gamma_0(q))$ is an orthonormal basis of the space of Maa\ss{} cusp forms of weight zero, level $q$, and principal nebentypus, and the Fourier expansion of such a Maa\ss{} cusp form $f$ with Laplacian eigenvalue $\lambda_f = 1/4 + t_f^2$ about the cusp at infinity is
\[f(z) = \sum_{\substack{n = -\infty \\ n \neq 0}}^{\infty} \rho_f(n) W_{0,it_f}(4\pi|n|y) e(nz).\]
Similarly, for a sequence $h^{\hol} : 2\N \to \C$, define
\[\AA_q^{\hol}\left(m,n;h^{\hol}\right) \defeq \frac{2 q}{\pi} \sum_{f \in \BB_{\hol}(\Gamma_0(q))} \Gamma(k_f) \sqrt{mn} \overline{\rho_f}(m) \rho_f(n) h^{\hol}(k_f),\]
where $\BB_{\hol}(\Gamma_0(q)) \ni f$ is an orthonormal basis of holomorphic cusp forms of weight $k_f \in 2\N$, level $q$, and principal nebentypus, and the Fourier expansion of such a holomorphic cusp form $f$ about the cusp at infinity is
\[f(z) = \sum_{n = 1}^{\infty} \rho_f(n) (4\pi n)^{k_f/2} e(nz).\]

\begin{lemma}
\label{Kuznetsovlemma}
For squarefree $q$, $\AA_q^{\Maass}(m,\pm n;h)$ is equal to
\begin{multline*}
\sum_{q_1 q_2 = q} \frac{q_2}{\nu(q_2)} \sum_{f \in \BB_0^{\ast}(\Gamma_0(q_1))} \epsilon_f^{\frac{1 \mp 1}{2}} \frac{h(t_f)}{L(1,\sym^2 f)} \sum_{\ell \mid q_2} L_{\ell}(1,\sym^2 f) \frac{\varphi(\ell)}{\ell}	\\
\times \sum_{\substack{v_1 w_1 = \ell \\ v_1 \mid m}} \frac{\nu(v_1)}{\sqrt{v_1}} \frac{\mu(w_1) \lambda_f(w_1)}{\sqrt{w_1}} \lambda_f\left(\frac{m}{v_1}\right) \sum_{\substack{v_2 w_2 = \ell \\ v_2 \mid n}} \frac{\nu(v_2)}{\sqrt{v_2}} \frac{\mu(w_2) \lambda_f(w_2)}{\sqrt{w_2}} \lambda_f\left(\frac{n}{v_2}\right),
\end{multline*}
$\AA_q^{\Eis}(m,\pm n;h)$ is equal to
\begin{multline*}
\frac{q}{2\pi \nu(q)} \int_{-\infty}^{\infty} \frac{h(t)}{\zeta(1 + 2it) \zeta(1 - 2it)} \sum_{\ell \mid q} \zeta_{\ell}(1 + 2it) \zeta_{\ell}(1 - 2it)	\\
\times \sum_{\substack{v_1 w_1 = \ell \\ v_1 \mid m}} \frac{\nu(v_1)}{\sqrt{v_1}} \frac{\mu(w_1) \lambda(w_1,t)}{\sqrt{w_1}} \lambda\left(\frac{m}{v_1},t\right) \sum_{\substack{v_2 w_2 = \ell \\ v_2 \mid n}} \frac{\nu(v_2)}{\sqrt{v_2}} \frac{\mu(w_2) \lambda(w_2,t)}{\sqrt{w_2}} \lambda\left(\frac{n}{v_2},t\right) \, dt,
\end{multline*}
and $\AA_q^{\hol}(m,n; h^{\hol})$ is equal to
\begin{multline*}
\sum_{q_1 q_2 = q} \frac{q_2}{\nu(q_2)} \sum_{f \in \BB_{\hol}^{\ast}(\Gamma_0(q_1))} \frac{h^{\hol}(k_f)}{L(1,\sym^2 f)} \sum_{\ell \mid q_2} L_{\ell}(1,\sym^2 f) \frac{\varphi(\ell)}{\ell}	\\
\times \sum_{\substack{v_1 w_1 = \ell \\ v_1 \mid m}} \frac{\nu(v_1)}{\sqrt{v_1}} \frac{\mu(w_1) \lambda_f(w_1)}{\sqrt{w_1}} \lambda_f\left(\frac{m}{v_1}\right) \sum_{\substack{v_2 w_2 = \ell \\ v_2 \mid n}} \frac{\nu(v_2)}{\sqrt{v_2}} \frac{\mu(w_2) \lambda_f(w_2)}{\sqrt{w_2}} \lambda_f\left(\frac{n}{v_2}\right).
\end{multline*}
\end{lemma}

\begin{proof}
For $\AA_q^{\Maass}(m,\pm n;h)$, we use the orthonormal basis in \hyperref[ILSlemma]{Lemma \ref*{ILSlemma}} and make use of \eqref{cusprholambda}, so that for $f \in \BB_0^{\ast}(\Gamma_0(q_1))$ and $\ell \mid q_2$,
\begin{align*}
\rho_{f_{\ell}}(n) & = \left(L_{\ell}(1,\sym^2 f) \frac{\varphi(\ell)}{\ell}\right)^{1/2} \sum_{vw = \ell} \frac{\nu(v)}{v} \frac{\mu(w) \lambda_f(w)}{\sqrt{w}} \rho_{\iota_v f}(n)	\\
& = \left(L_{\ell}(1,\sym^2 f) \frac{\varphi(\ell)}{\ell}\right)^{1/2} \frac{\rho_f(1)}{\sqrt{n}} \sum_{\substack{v w = \ell \\ v \mid n}} \frac{\nu(v)}{\sqrt{v}} \frac{\mu(w) \lambda_f(w)}{\sqrt{w}} \lambda_f\left(\frac{n}{v}\right).
\end{align*}
\hyperref[rho(1)^2lemma]{Lemma \ref*{rho(1)^2lemma}} gives an explicit expression for $|\rho_f(1)|^2$, which gives the desired identity.

The orthonormal basis in \hyperref[YoungEislemma]{Lemma \ref*{YoungEislemma}} similarly gives the identity for $\AA_q^{\Eis}(m,\pm n;h)$.

Finally, an orthonormal basis of $\BB_{\hol}(\Gamma_0(q))$ is given by
\[\BB_{\hol}\left(\Gamma_0(q)\right) = \left\{f_{\ell} : f \in \BB_{\hol}^{\ast}\left(\Gamma_0(q_1)\right), \ q_1 q_2 = q, \ \ell \mid q_2\right\}\]
via \cite[Proposition 2.6]{ILS00}, where
\[f_{\ell} \defeq \left(L_{\ell}(1,\sym^2 f) \frac{\varphi(\ell)}{\ell}\right)^{1/2} \sum_{vw = \ell} \frac{\nu(v)}{v^{1 - k_f}} \frac{\mu(w) \lambda_f(w)}{\sqrt{w}} \iota_v f\]
with $f \in \BB_{\hol}^{\ast}(\Gamma_0(q_1))$ normalised such that $\langle f, f \rangle_q = 1$, so that
\[\rho_{f_{\ell}}(n) = \left(L_{\ell}(1,\sym^2 f) \frac{\varphi(\ell)}{\ell}\right)^{1/2} \frac{\rho_f(1)}{\sqrt{n}} \sum_{\substack{vw = \ell \\ v \mid n}} \frac{\nu(v)}{\sqrt{v}} \frac{\mu(w) \lambda_f(w)}{\sqrt{w}} \lambda_f\left(\frac{n}{v}\right).\]
Moreover, via the same method of proof of \hyperref[rho(1)^2lemma]{Lemma \ref*{rho(1)^2lemma}},
\[|\rho_f(1)|^2 = \frac{\pi q_2 \langle f, f\rangle_q}{2q \nu(q_2) \Gamma(k_f) L(1,\sym^2 f)}\]
for $f \in \BB_{\hol}^{\ast}(\Gamma_0(q_1))$ with $q_1 q_2 = q$. The result then follows.
\end{proof}

The terms $\AA_q^{\Maass}(m,\pm n;h)$, $\AA_q^{\Eis}(m,\pm n;h)$, and $\AA_q^{\hol}(m,n;h^{\hol})$ arise from the spectral expansion of the inner product of two Poincar\'{e} series associated to the pair of cusps $(\aa,\bb) = (\infty,\infty)$. We require similar identities for $\bb \sim 1$, for which we choose the scaling matrix
\[\sigma_{\bb} = \begin{pmatrix} \sqrt{q} & b/\sqrt{q} \\ \sqrt{q} & d\sqrt{q} \end{pmatrix},\]
where $b,d \in \Z$ are such that $dq - b = 1$. We define
\begin{align*}
\AA_q^{\Maass}(\sigma_{\bb};m,\pm n;h) & \defeq 2q \sum_{f \in \BB_0(\Gamma_0(q))} \frac{\sqrt{mn} \overline{\rho_f}(\sigma_{\bb};m) \rho_f(\pm n)}{\cosh \pi t_f} h(t_f),	\\
\AA_q^{\Eis}(\sigma_{\bb};m,\pm n;h) & \defeq \frac{q}{2\pi} \sum_{\aa} \int_{-\infty}^{\infty} \frac{\sqrt{mn} \overline{\rho_{\aa}}(\sigma_{\bb};m,t) \rho_{\aa}(\pm n,t)}{\cosh \pi t} h(t) \, dt,	\\
\AA_q^{\hol}(\sigma_{\bb};m,n;h^{\hol}) & \defeq \frac{2 q}{\pi} \sum_{f \in \BB_{\hol}(\Gamma_0(q))} \Gamma(k_f) \sqrt{mn} \overline{\rho_f}(\sigma_{\bb};m) \rho_f(n) h^{\hol}(k_f).
\end{align*}
Here $\rho_f(\sigma_{\bb};m)$ denotes the $m$-Fourier coefficient of $f(\sigma_{\bb} z)$ and $\rho_{\aa}(\sigma_{\bb};m,t)$ denotes the $m$-th Fourier coefficient of $E_{\aa}(\sigma_{\bb} z,1/2 + it)$.

\begin{lemma}
\label{infty1Kuznetsovlemma}
For squarefree $q$, $\AA_q^{\Maass}(\sigma_{\bb};m,\pm n;h)$ is equal to
\begin{multline*}
\sum_{q_1 q_2 = q} \frac{q_2^{3/2}}{\nu(q_2)} \sum_{f \in \BB_0^{\ast}(\Gamma_0(q_1))} \epsilon_f^{\frac{1 \mp 1}{2}} \eta_f(q_1) \frac{h(t_f)}{L(1,\sym^2 f)} \sum_{\substack{\ell \mid q_2 \\ \frac{q_2}{\ell} \mid m}} L_{\ell}(1,\sym^2 f) \frac{\varphi(\ell)}{\ell}	\\
\times \sum_{\substack{v_1 w_1 = \ell \\ v_1 \mid m}} \frac{\nu(w_1)}{w_1^{3/2}} \frac{\mu(v_1) \lambda_f(v_1)}{\sqrt{v_1}} \lambda_f\left(\frac{w_1 m}{q_2}\right) \sum_{\substack{v_2 w_2 = \ell \\ v_2 \mid n}} \frac{\nu(v_2)}{\sqrt{v_2}} \frac{\mu(w_2) \lambda_f(w_2)}{\sqrt{w_2}} \lambda_f\left(\frac{n}{v_2}\right),
\end{multline*}
$\AA_q^{\Eis}(\sigma_{\bb};m,\pm n;h)$ is equal to
\begin{multline*}
\frac{q^{3/2}}{2\pi \nu(q)} \int_{-\infty}^{\infty} \frac{h(t)}{\zeta(1 + 2it) \zeta(1 - 2it)} \sum_{\substack{\ell \mid q \\ \frac{q}{\ell} \mid m}} \zeta_{\ell}(1 + 2it) \zeta_{\ell}(1 - 2it)	\\
\times \sum_{\substack{v_1 w_1 = \ell \\ v_1 \mid m}} \frac{\nu(w_1)}{w_1^{3/2}} \frac{\mu(v_1) \lambda(v_1,t)}{\sqrt{v_1}} \lambda\left(\frac{w_1 m}{q},t\right) \sum_{\substack{v_2 w_2 = \ell \\ v_2 \mid n}} \frac{\nu(v_2)}{\sqrt{v_2}} \frac{\mu(w_2) \lambda(w_2,t)}{\sqrt{w_2}} \lambda\left(\frac{n}{v_2},t\right) \, dt,
\end{multline*}
and $\AA_q^{\hol}(\sigma_{\bb};m,n;h^{\hol})$ is equal to
\begin{multline*}
\sum_{q_1 q_2 = q} \frac{q_2^{3/2}}{\nu(q_2)} \sum_{f \in \BB_{\hol}^{\ast}(\Gamma_0(q_1))} \eta_f(q_1) \frac{h^{\hol}(k_f)}{L(1,\sym^2 f)} \sum_{\substack{\ell \mid q_2 \\ \frac{q_2}{\ell} \mid m}} L_{\ell}(1,\sym^2 f) \frac{\varphi(\ell)}{\ell}	\\
\times \sum_{\substack{v_1 w_1 = \ell \\ v_1 \mid m}} \frac{\nu(w_1)}{w_1^{3/2}} \frac{\mu(v_1) \lambda_f(v_1)}{\sqrt{v_1}} \lambda_f\left(\frac{w_1 m}{q_2}\right) \sum_{\substack{v_2 w_2 = \ell \\ v_2 \mid n}} \frac{\nu(v_2)}{\sqrt{v_2}} \frac{\mu(w_2) \lambda_f(w_2)}{\sqrt{w_2}} \lambda_f\left(\frac{n}{v_2}\right).
\end{multline*}
\end{lemma}

\begin{proof}
If $vw = q_2$ with $q_1 q_2 = q$,
\[\begin{pmatrix} \sqrt{v} & 0 \\ 0 & 1/\sqrt{v} \end{pmatrix} \begin{pmatrix} \sqrt{q} & b/\sqrt{q} \\ \sqrt{q} & d\sqrt{q} \end{pmatrix} = \begin{pmatrix} v\sqrt{q_1} & b/\sqrt{q_1} \\ \sqrt{q_1} & dw\sqrt{q_1} \end{pmatrix} \begin{pmatrix} \sqrt{w} & 0 \\ 0 & 1/\sqrt{w} \end{pmatrix}.\]
So if $f$ is a member of $\BB_0^{\ast}(\Gamma_0(q_1))$ or $\BB_{\hol}^{\ast}(\Gamma_0(q_1))$,
\[(\iota_v f)(\sigma_{\bb} z) = \eta_f(q_1)(\iota_w f)(z)\]
by \hyperref[AtkinLehnerlemma]{Lemma \ref*{AtkinLehnerlemma}}. The Fourier coefficients $\rho_{\iota_v f}(\sigma_{\bb};n)$ of $(\iota_v f)(\sigma_{\bb} z)$ therefore satisfy
\[\rho_{\iota_v f}(\sigma_{\bb};n) = \begin{dcases*}
\eta_f(q_1) \rho_f(1) \lambda_f\left(\frac{n}{w}\right) \sqrt{\frac{w}{n}} & if $n \equiv 0 \pmod{w}$,	\\
0 & otherwise
\end{dcases*}\]
via \eqref{cusprholambda}. It follows that for $\ell \mid q_2$,
\begin{align*}
\rho_{f_{\ell}}(\sigma_{\bb};n) & = \left(L_{\ell}(1,\sym^2 f) \frac{\varphi(\ell)}{\ell}\right)^{1/2} \sum_{vw = \ell} \frac{\nu(v)}{v} \frac{\mu(w) \lambda_f(w)}{\sqrt{w}} \rho_{\iota_v f}(\sigma_{\bb};n)	\\
& = \sqrt{q_2} \eta_f(q_1) \left(L_{\ell}(1,\sym^2 f) \frac{\varphi(\ell)}{\ell}\right)^{1/2} \frac{\rho_f(1)}{\sqrt{n}} \sum_{\substack{vw = \ell \\ \frac{q_2}{v} \mid n}} \frac{\nu(v)}{v^{3/2}} \frac{\mu(w) \lambda_f(w)}{\sqrt{w}} \lambda_f\left(\frac{vn}{q_2}\right).
\end{align*}
Now the proof follows in the same way as the proof of \hyperref[Kuznetsovlemma]{Lemma \ref*{Kuznetsovlemma}}.
\end{proof}

\subsection{Spectral Summation Formul\ae{}}

\subsubsection{The Kuznetsov Formula}

The Kuznetsov formula is an identity between a spectral sum of Fourier coefficients of Maa\ss{} cusps forms and integral of Fourier coefficients of Eisenstein series and a delta term and weighted sum of Kloosterman sums.

\begin{theorem}[{\cite[Theorem 9.3]{Iwa02}}]
\label{Kuznetsovthm}
Let $\delta > 0$, and let $h$ be a function that is even, holomorphic in the horizontal strip $|\Im(t)| \leq 1/4 + \delta$, and satisfies $h(t) \ll (1 + |t|)^{-2 - \delta}$. Then for $m,n \in \N$,
\[\AA_q^{\Maass}(m,\pm n;h) + \AA_q^{\Eis}(m,\pm n;h) = \DD_q(m,\pm n;\Ns h) + \OO_q(m,\pm n;\Ks^{\pm} h),\]
where
\begin{align}
\notag
\DD_q(m,\pm n;\Ns h) & \defeq \delta_{m, \pm n} q \Ns h,	\\
\label{OOqeq}
\OO_q(m,\pm n;\Ks^{\pm} h) & \defeq q \sum_{\substack{c = 1 \\ c \equiv 0 \hspace{-.25cm} \pmod{q}}}^{\infty} \frac{S(m,\pm n;c)}{c} (\Ks^{\pm} h)\left(\frac{\sqrt{mn}}{c}\right).
\end{align}
Here
\begin{gather}
\label{Kloosteq}
S(m,n;c) \defeq \sum_{d \in (\Z/c\Z)^{\times}} e\left(\frac{md + n\overline{d}}{c}\right),	\\
\label{NsKspmeq}
\Ns h \defeq \int_{-\infty}^{\infty} h(r) \, d_{\spec}r, \qquad (\Ks^{\pm} h)(x) \defeq \int_{-\infty}^{\infty} \JJ_r^{\pm}(x) h(r) \, d_{\spec}r,	\\
\label{Jrpmeq}
\JJ_r^{+}(x) \defeq \frac{\pi i}{\sinh \pi r} \left(J_{2ir}(4\pi x) - J_{-2ir}(4\pi x)\right), \qquad \JJ_r^{-}(x) \defeq 4 \cosh \pi r K_{2ir}(4\pi x),	\\
\label{dspeceq}
d_{\spec} r \defeq \frac{1}{2\pi^2} r \tanh \pi r \, dr,
\end{gather}
where $K_{\nu}(z)$ denotes the modified Bessel function of the second kind.
\end{theorem}

This is the Kuznetsov formula associated to the pair of cusps $(\aa,\bb) = (\infty,\infty)$. We also require the Kuznetsov formula associated to the pair of cusps $(\aa,\bb) = (\infty,1)$.

\begin{theorem}[{\cite[Theorem 9.3]{Iwa02}}]
\label{infty1Kuznetsovthm}
Let $\delta > 0$, and let $h$ be a function that is even, holomorphic in the horizontal strip $|\Im(t)| \leq 1/4 + \delta$, and satisfies $h(t) \ll (1 + |t|)^{-2 - \delta}$. Then for $m,n \in \N$ and $q > 1$,
\[\AA_q^{\Maass}(\sigma_{\bb};m,\pm n;h) + \AA_q^{\Eis}(\sigma_{\bb};m,\pm n;h) = \OO_q(\sigma_{\bb};m,\pm n;\Ks^{\pm} h),\]
where for $\overline{q} \in \Z$ such that $\overline{q} q \equiv 1 \pmod{c}$,
\[\OO_q(\sigma_{\bb};m,\pm n;\Ks^{\pm} h) \defeq \sqrt{q} \sum_{\substack{c = 1 \\ (c,q) = 1}}^{\infty} \frac{S(m,\pm n\overline{q};c)}{c} (\Ks^{\pm} h)\left(\frac{\sqrt{mn}}{c\sqrt{q}}\right).\]
\end{theorem}

The weakening of the requirement that $h$ need only be holomorphic in the strip $|\Im(t)| \leq 1/4 + \delta$ instead of $1/2 + \delta$ is due to Yoshida \cite[Theorem]{Yos97}, where this is proven only in the case $q = 1$; the proof generalises immediately to all cases of the Kuznetsov formula for which the Kloosterman sums appearing in the Kloosterman term satisfy the Weil bound.

\subsubsection{The Petersson Formula}

The Petersson formula is an identity between a sum of Fourier coefficients of holomorphic cusps forms and a delta term and weighted sum of Kloosterman sums.

\begin{theorem}[{\cite[Theorem 9.6]{Iwa02}}]
\label{Peterssonthm}
Let $h^{\hol} : 2\N \to \C$ be a sequence satisfying $h^{\hol}(k) \ll k^{-2 - \delta}$ for some $\delta > 0$. Then for $m,n \in \N$,
\[\AA_q^{\hol}\left(m,n;h^{\hol}\right) = \DD_q^{\hol}\left(m,n; \Ns h^{\hol}\right) + \OO_q^{\hol}\left(m,n; \Ks^{\hol} h^{\hol}\right),\]
where
\begin{align*}
\DD_q^{\hol}\left(m,n; \Ns h^{\hol}\right) & \defeq \delta_{m,n} q \sum_{\substack{k = 2 \\ k \equiv 0 \hspace{-.25cm} \pmod{2}}}^{\infty} \frac{k - 1}{2 \pi^2} h^{\hol}(k),	\\
\OO_q^{\hol}\left(m,n; \Ks^{\hol} h^{\hol}\right) & \defeq q \sum_{\substack{c = 1 \\ c \equiv 0 \hspace{-.25cm} \pmod{q}}}^{\infty} \frac{S(m,n;c)}{c} \left(\Ks^{\hol} h^{\hol}\right)\left(\frac{\sqrt{mn}}{c}\right).
\end{align*}
Here
\begin{equation}
\label{KsholJkholeq}
\left(\Ks^{\hol} h^{\hol}\right)(x) \defeq \sum_{\substack{k = 2 \\ k \equiv 0 \hspace{-.25cm} \pmod{2}}}^{\infty} \frac{k - 1}{2 \pi^2} \JJ_k^{\hol}(x) h^{\hol}(k), \qquad \JJ_k^{\hol}(x) \defeq 2\pi i^{-k} J_{k - 1}(4\pi x).
\end{equation}
\end{theorem}

We also require the Petersson formula associated to $(\aa,\bb) = (\infty,1)$.

\begin{theorem}[{\cite[Theorem 9.6]{Iwa02}}]
\label{infty1Peterssonthm}
Let $h^{\hol} : 2\N \to \C$ be a sequence satisfying $h^{\hol}(k) \ll k^{-2 - \delta}$ for some $\delta > 0$. Then for $m,n \in \N$ and $q > 1$,
\[\AA_q^{\hol}\left(\sigma_{\bb};m,n;h^{\hol}\right) = \OO_q^{\hol}\left(\sigma_{\bb};m,n; \Ks^{\hol} h^{\hol}\right),\]
where
\[\OO_q^{\hol}\left(\sigma_{\bb};m,n; \Ks^{\hol} h^{\hol}\right) \defeq \sqrt{q} \sum_{\substack{c = 1 \\ (c,q) = 1}}^{\infty} \frac{S(m,n\overline{q};c)}{c} \left(\Ks^{\hol} h^{\hol}\right)\left(\frac{\sqrt{mn}}{c\sqrt{q}}\right).\]
\end{theorem}

\subsubsection{The Kloosterman Summation Formula}

The Kloosterman summation formula (due to Kuznetsov and often referred to as the Kuznetsov formula, though differing from \hyperref[Kuznetsovthm]{Theorem \ref*{Kuznetsovthm}}) gives an expression in reverse to \hyperref[Kuznetsovthm]{Theorems \ref*{Kuznetsovthm}} and \ref{Peterssonthm}. Rather than expressing sums of Fourier coefficients of automorphic forms weighted by functions $h$ or $h^{\hol}$ in terms of a delta term and sums of Kloosterman sums weighted by transformed functions $\Ks^{\pm} h$ and $\Ks^{\hol} h^{\hol}$, it expresses sums of Kloosterman sums weighted by a function $H$ in terms of sums of automorphic forms weighted by transformed functions $\Ls^{\pm} H$ and $\Ls^{\hol} H$. Notably, there is no delta term in the Kloosterman summation formula.

\begin{theorem}[{\cite[Theorem 16.5]{IK04}}]
\label{Kloostermanthm}
For $H \in C^3((0,\infty))$ satisfying
\[x^j \frac{d^j}{dx^j} H(x) \ll \min\left\{x,x^{-\frac{3}{2}}\right\}\]
for $j \in \{0,1,2,3\}$ and $m,n \geq 1$, we have that
\[\AA_q^{\Maass}\left(m,\pm n; \Ls^{\pm} H\right) + \AA_q^{\Eis}\left(m,\pm n; \Ls^{\pm} H\right) + \delta_{\pm,+} \AA_q^{\hol}\left(m,n; \Ls^{\hol} H\right) = \OO_q(m,\pm n; H),\]
where
\begin{equation}
\label{Lseq}
(\Ls^{\pm} H)(t) \defeq \int_{0}^{\infty} \JJ_t^{\pm}(x) H(x) \, \frac{dx}{x}, \qquad (\Ls^{\hol} H)(k) \defeq \int_{0}^{\infty} \JJ_k^{\hol}(x) H(x) \, \frac{dx}{x}.
\end{equation}
\end{theorem}

Once more, we will require the Kloosterman summation formula associated to the pair of cusps $(\aa,\bb) = (\infty,1)$.

\begin{theorem}
\label{infty1Kloostermanthm}
For $H \in C^3((0,\infty))$ satisfying
\[x^j \frac{d^j}{dx^j} H(x) \ll \min\left\{x,x^{-\frac{3}{2}}\right\}\]
for $j \in \{0,1,2,3\}$ and $m,n \geq 1$, we have that
\begin{multline*}
\AA_q^{\Maass}\left(\sigma_{\bb};m,\pm n; \Ls^{\pm} H\right) + \AA_q^{\Eis}\left(\sigma_{\bb};m,\pm n; \Ls^{\pm} H\right) + \delta_{\pm,+} \AA_q^{\hol}\left(\sigma_{\bb};m,n; \Ls^{\hol} H\right)	\\
= \OO_q(\sigma_{\bb};m,\pm n; H).
\end{multline*}
\end{theorem}

\subsection{The Mellin Transform}

We recall the following definitions and properties of the Mellin transform; see \cite[Section 2.1]{BlK19b}. Let $W : [0,\infty) \to \C$ be a $J$-times continuously differentiable function satisfying $x^j W^{(j)}(x) \ll_{J,a,b} \min\{x^{-a},x^{-b}\}$ for some $-\infty < a < b < \infty$ and $j \in \{0,\ldots,J\}$. The Mellin transform $\widehat{W}$ of $W$ is
\[\widehat{W}(s) \defeq \int_{0}^{\infty} W(x) x^s \, \frac{dx}{x}.\]
This is defined initially as an absolutely convergent integral for $a < \Re(s) < b$ and satisfies $\widehat{W}(s) \ll_J (1 + |s|)^{-J}$ in this region. Similarly, the inverse Mellin transform of a holomorphic function $\WW :\{z \in \C : a < \Re(s) < b\} \to \C$ satisfying $\WW(s) \ll_r (1 + |s|)^{-r}$ for some $r > 1$ is given by
\[\widecheck{\WW}(x) \defeq \frac{1}{2\pi i} \int_{\sigma - i\infty}^{\sigma + i\infty} \WW(s) x^{-s} \, ds,\]
where $a < \sigma < b$. This is a $J$-times continuously differentiable function on $[0,\infty)$, where $J = \lceil r \rceil - 1$, and satisfies $x^j \widecheck{\WW}^{(j)}(x) \ll_{J,a,b} \min\{x^{-a},x^{-b}\}$ for $j \in \{0,\ldots,J\}$.

\begin{lemma}[{\cite[(A.7)]{BLM19}, \cite[(3.13)]{BlK19b}}]
We have that
\begin{align}
\notag
\widehat{\JJ_r^{+}}(s) & = \frac{\pi i (2\pi)^{-s}}{2 \sinh \pi r} \left(\frac{\Gamma\left(\frac{s}{2} + ir\right)}{\Gamma\left(1 - \frac{s}{2} + ir\right)} - \frac{\Gamma\left(\frac{s}{2} - ir\right)}{\Gamma\left(1 - \frac{s}{2} - ir\right)}\right)	\\
\label{J_r+Mellin}
& = (2\pi)^{-s} \Gamma\left(\frac{s}{2} + ir\right) \Gamma\left(\frac{s}{2} - ir\right) \cos \frac{\pi s}{2},	\\
\label{J_r-Mellin0}
\widehat{\JJ_r^{-}}(s) & = \frac{\pi i (2\pi)^{-s}}{2 \tanh \pi r \cos \frac{\pi s}{2}} \left(\frac{\Gamma\left(\frac{s}{2} + ir\right)}{\Gamma\left(1 - \frac{s}{2} + ir\right)} - \frac{\Gamma\left(\frac{s}{2} - ir\right)}{\Gamma\left(1 - \frac{s}{2} - ir\right)}\right)	\\
\label{J_r-Mellin}
& = (2\pi)^{-s} \Gamma\left(\frac{s}{2} + ir\right) \Gamma\left(\frac{s}{2} - ir\right) \cosh \pi r,	\\
\notag
\widehat{\JJ_k^{\hol}}(s) & = \pi i^{-k} (2\pi)^{-s} \frac{\Gamma\left(\frac{s + k - 1}{2}\right)}{\Gamma\left(\frac{1 - s + k}{2}\right)}	\\
\notag
& = (2\pi)^{-s} \Gamma\left(\frac{s + k - 1}{2}\right) \Gamma\left(\frac{s - k + 1}{2}\right) \cos \frac{\pi s}{2}.
\end{align}
\end{lemma}

From Stirling's formula \eqref{Stirlingeq}, we obtain the following.

\begin{corollary}
\label{JJrMellinasympcor}
The functions $\widehat{\JJ_r^{\pm}}(s)$ extend meromorphically to $\C$ with simple poles at $s = 2(\pm ir - n)$ for $n \in \N_0$. For $s = \sigma + i\tau \in \C$ in bounded vertical strips at least a bounded distance away from $\{2(\pm ir - n) : n \in \N_0\}$ and $r = u + iv$ in bounded horizontal strips,
\begin{align*}
\widehat{\JJ_r^{+}}(s) & \ll_{\sigma,v} \left(1 + \left|\tau + 2u\right|\right)^{\frac{1}{2} (\sigma - 2v - 1)} \left(1 + \left|\tau - 2u\right|\right)^{\frac{1}{2} (\sigma + 2v - 1)} \times \begin{dcases*}
e^{-\frac{\pi}{2}(2|u| - |\tau|)} & if $|\tau| \leq 2|u|$,	\\
1 & if $|\tau| \geq 2|u|$,
\end{dcases*}	\\
\widehat{\JJ_r^{-}}(s) & \ll_{\sigma,v} \left(1 + \left|\tau + 2u\right|\right)^{\frac{1}{2} (\sigma - 2v - 1)} \left(1 + \left|\tau - 2u\right|\right)^{\frac{1}{2} (\sigma + 2v - 1)} \times \begin{dcases*}
1 & if $|\tau| \leq 2|u|$,	\\
e^{-\frac{\pi}{2}(|\tau| - 2|u|)} & if $|\tau| \geq 2|u|$.
\end{dcases*}
\end{align*}
Moreover,
\[\Res_{s = 2(\pm ir - n)} \widehat{\JJ_r^{+}}(s) = (-1)^n \Res_{s = 2(\pm ir - n)} \widehat{\JJ_r^{-}}(s) \ll_{\sigma,v} (1 + |u|)^{-n \mp v - \frac{1}{2}}.\]
For $s = \sigma + i\tau \in \C$ in bounded vertical strips, at least a bounded distance away from $\{1 - k - 2n : n \in \N_0\}$,
\[\widehat{\JJ_k^{\hol}}(s) \ll_{\sigma} (k + |\tau|)^{\sigma - 1}.\]
Moreover,
\[\Res_{s = 1 - k - 2n} \widehat{\JJ_k^{\hol}}(s) = \frac{(2\pi i)^{k + 2n}}{\Gamma(k + n) \Gamma(n + 1)}.\]
\end{corollary}

We require the following result on properties of $\widehat{\Ks^{-} h}(s)$.

\begin{lemma}[{\cite[Section 3.3]{Mot97}}]
\label{MellinKsextendlemma}
Suppose that $h(r)$ is an even holomorphic function in the strip $-3/2 < \Im(r) < 3/2$ with zeroes at $\pm i/2$ and satisfies $h(r) \ll (1 + |r|)^{-4 - \delta}$ in this region for some $\delta > 0$. Then the Mellin transform of $\Ks^{-} h$ extends to a holomorphic function in the strip $-3 < \Re(s) < 1$.
\end{lemma}

\begin{proof}
Since $h$ is even and recalling \eqref{J_r-Mellin0}, we have that for $0 < \Re(s) < 1$,
\[\widehat{\Ks^{-} h}(s) = \int_{-\infty}^{\infty} \widehat{\JJ_r^{-}}(s) h(r) \, d_{\spec}r = \frac{i (2\pi)^{-s - 1}}{\cos \frac{\pi s}{2}} \int_{-\infty}^{\infty} \frac{\Gamma\left(\frac{s}{2} + ir\right)}{\Gamma\left(1 - \frac{s}{2} + ir\right)} r h(r) \, dr.\]
Indeed, standard bounds for $\JJ_r^{-}(x)$ (see, for example, \cite[(A.3)]{BLM19}) allow us to interchange the order of integration. For $\Re(s) > -\sigma_0$, we may shift the contour to $\Im(r) = -\sigma_0/2 - \e$; provided that the integral converges, we see that the integral extends holomorphically to $-\sigma_0 < \Re(s) < 1$. \hyperref[JJrMellinasympcor]{Corollary \ref*{JJrMellinasympcor}} then implies that the integral over $r$ converges provided that $h(r) \ll (1 + |r|)^{-1 - \Re(s) - \delta}$ for some $\delta > 0$.

This proves the analytic continuation of the integral to $-3 < \Re(s) < 1$. The Mellin transform of $\Ks^{-} h$ may have a pole at $s = 1$, however, due to the presence of the term $\cos (\pi s/2)$. The integral in this case is
\[-\int_{\Im(r) = -1/2 - \e} \frac{r h(r)}{\left(\frac{1}{4} + r^2\right)} \, dr.\]
We move the contour back to $\Im(r) = 0$. The resulting integral vanishes, while we pick up a residue at $r = -i/2$ given by $-\pi i h(i/2)$. By assumption, this vanishes, which completes the proof.
\end{proof}

\subsection{Vorono\u{\i} Summation Formul\ae{}}

For $\Re(s) > 1$, $c \in \N$, and $d \in (\Z/c\Z)^{\times}$, we define the Vorono\u{\i} $L$-series
\begin{equation}
\label{VoronoiLseriesdefeq}
\begin{split}
L\left(s, E_{\chi,1}, \frac{d}{c}\right) & \defeq \sum_{m = 1}^{\infty} \frac{\lambda_{\chi,1}(m,0) e\left(\frac{md}{c}\right)}{m^s},	\\
L\left(s, g, \frac{d}{c}\right) & \defeq \sum_{n = 1}^{\infty} \frac{\lambda_g(n) e\left(\frac{nd}{c}\right)}{n^s}.
\end{split}
\end{equation}
These functions are associated to the automorphic forms $E_{\chi,1}(z) \defeq E_{\infty}(z,1/2,\chi)$ and even $g \in \BB_0^{\ast}(q,\chi)$ respectively.

\begin{lemma}
\label{Voronoilemma}
For $c \equiv 0 \pmod{q}$ or $(c,q) = 1$, the Vorono\u{\i} $L$-series $L(s,E_{\chi,1},d/c)$ extends to a meromorphic function on $\C$ with a simple pole at $s = 1$ with residue
\[\begin{dcases*}
\frac{\tau(\chi) \overline{\chi}(d) L(1, \overline{\chi})}{c} & if $c \equiv 0 \pmod{q}$,	\\
\frac{\chi(c) L(1,\chi)}{c} & if $(c,q) = 1$,
\end{dcases*}\]
while the Vorono\u{\i} $L$-series $L(s,g,d/c)$ extends to an entire function. We have the functional equations
\begin{align*}
L\left(s, E_{\chi,1}, \frac{d}{c}\right) & = \frac{2 \overline{\chi}(d)}{c^{2s - 1}} \sum_{\pm} \widehat{\JJ_0^{\pm}}(2(1 - s)) L\left(1 - s, E_{\chi,1}, \mp\frac{\overline{d}}{c}\right),	\\
L\left(s, g, \frac{d}{c}\right) & = \frac{2 \overline{\chi}(d)}{c^{2s - 1}} \sum_{\pm} \widehat{\JJ_{t_g}^{\pm}}(2(1 - s)) L\left(1 - s, g,\mp\frac{\overline{d}}{c}\right)
\end{align*}
if $c \equiv 0 \pmod{q}$, while for $(c,q) = 1$,
\begin{align*}
L\left(s, E_{\chi,1}, \frac{d}{c}\right) & = \frac{2 \chi(-c) \tau(\chi)}{c^{2s - 1} q^s} \sum_{\pm} \widehat{\JJ_0^{\pm}}(2(1 - s)) L\left(1 - s, E_{\chi,1},\mp\frac{\overline{dq}}{c}\right),	\\
L\left(s, g, \frac{d}{c}\right) & = \frac{2 \chi(-c) \tau(\chi)}{\lambda_g(q) c^{2s - 1} q^s} \sum_{\pm} \widehat{\JJ_{t_g}^{\pm}}(2(1 - s)) L\left(1 - s, g,\mp\frac{\overline{dq}}{c}\right).
\end{align*}
\end{lemma}

\begin{proof}
For $L(s,g,d/c)$, this follows from \cite[Appendix A.4]{KMV02} and \cite[Section 2.4]{HM06}. After Mellin inversion, the identities for $L(s,E_{\chi,1},d/c)$ are shown in \cite[Theorems 4.13 and 4.14]{IK04} and also \cite[Theorem A]{LT05}.
\end{proof}

A useful tool to couple with the Vorono\u{\i} summation formula is the following identity for Gauss sums.

\begin{lemma}[{\cite[Lemma 3.1.3]{Miy06}}]
\label{Miyakelemma}
Let $\chi$ be a primitive Dirichlet character modulo $q$ and $c \equiv 0 \pmod{q}$. We have that
\[\sum_{d \in (\Z/c\Z)^{\times}} \chi(d) e\left(\frac{md}{c}\right) = \tau(\chi) \sum_{a \mid \left(\frac{c}{q}, m\right)} a \mu\left(\frac{c}{aq}\right) \chi\left(\frac{c}{aq}\right) \overline{\chi}\left(\frac{m}{a}\right).\]
\end{lemma}

\subsection{The Large Sieve}

\begin{theorem}[{\cite[Theorems 2.2 and 2.6]{Lam14}}]
\label{largesievethm}
For squarefree $q$, $1 \ll U \ll T$, and $N \geq 1$, each of the quantities
\begin{gather*}
\sum_{\substack{f \in \BB_0^{\ast}(\Gamma_0(q)) \\ T - U \leq t_f \leq T + U}} \frac{1}{L(1,\sym^2 f)} \left|\sum_{N \leq n \leq 2N} a_n \lambda_f(n)\right|^2,	\\
\frac{\delta_{q,1}}{2\pi} \int\limits_{T - U \leq |t| \leq T + U} \frac{1}{\zeta(1 + 2it) \zeta(1 - 2it)} \left|\sum_{N \leq n \leq 2N} a_n \lambda(n,t)\right|^2 \, dt,	\\
\sum_{\substack{f \in \BB_{\hol}^{\ast}(\Gamma_0(q)) \\ T - U \leq k_f \leq T + U}} \frac{1}{L(1,\sym^2 f)} \left|\sum_{N \leq n \leq 2N} a_n \lambda_f(n)\right|^2
\end{gather*}
is bounded by a constant multiple depending on $\e$ of
\[(qTU + N)(qTN)^{\e} \sum_{N \leq n \leq 2N} |a_n|^2.\]
\end{theorem}

\subsection{Subconvexity Estimates}

We record the following subconvexity estimates.

\begin{theorem}[{\cite[Theorem 1.1]{You17}}]
\label{Youngsubconvthm}
Let $\chi_q$ be the primitive quadratic Dirichlet character modulo $q$ for squarefree odd $q$. Then for $q_1 \mid q$,
\begin{align*}
L\left(\frac{1}{2}, f \otimes \chi_q\right) & \ll_{\e} \begin{dcases*}
(q(1 + |t_f|))^{\frac{1}{3} + \e} & if $f \in \BB_0^{\ast}(\Gamma_0(q_1))$,	\\
(qk_f))^{\frac{1}{3} + \e} & if $f \in \BB_{\hol}^{\ast}(\Gamma_0(q_1))$,
\end{dcases*}	\\
\left|L\left(\frac{1}{2} + it, \chi_q\right)\right|^2 & \ll_{\e} (q(1 + |t|))^{\frac{1}{3} + \e}.
\end{align*}
\end{theorem}

\begin{theorem}[{\cite[Theorems 1.1 and 1.2]{MV10}; see also \cite[Theorem 1 and Remarks, p.~114]{Blo05}, and cf.~\cite[Corollary 1.2 and Remark 1.3]{LLY06a}}]
\label{MVsubconvthm}
Let $g \in \BB_0^{\ast}(q,\chi)$ and $t \in \R$. Then for $q_1 \mid q$, there exist absolute constants $A > 0$ and $\delta > 0$ such that
\begin{align*}
L\left(\frac{1}{2}, f \otimes g\right) & \ll \begin{dcases*}
(1 + |t_f|)^A t_g^{1 - \delta} & if $f \in \BB_0^{\ast}(\Gamma_0(q_1))$,	\\
k_f^A t_g^{1 - \delta} & if $f \in \BB_{\hol}^{\ast}(\Gamma_0(q_1))$,
\end{dcases*}	\\
\left|L\left(\frac{1}{2} + it, g\right)\right|^2 & \ll (|t| + t_g)^{1 - \delta},	\\
L\left(\frac{1}{2}, f\right) & \ll \begin{dcases*}
(1 + |t_f|)^{\frac{1}{2} - \delta} & if $f \in \BB_0^{\ast}(\Gamma_0(q_1))$,	\\
k_f^{\frac{1}{2} - \delta} & if $f \in \BB_{\hol}^{\ast}(\Gamma_0(q_1))$,
\end{dcases*}	\\
\left|\zeta\left(\frac{1}{2} + it\right)\right|^2 & \ll (1 + |t|)^{\frac{1}{2} - \delta}.
\end{align*}
\end{theorem}

\begin{remark}
More explicit subconvex bounds are known for $\zeta(1/2 + it)$, as well as for $L(1/2,f)$ when $q = 1$, but all we truly require are subconvex bounds
\begin{align*}
L\left(\frac{1}{2},f\right) L\left(\frac{1}{2},f \otimes \chi_D\right) & \ll \begin{dcases*}
(1 + |t_f|)^{1 - \delta} & for $f \in \BB_0^{\ast}(\Gamma_0(q_1))$,	\\
k_f^{1 - \delta} & for $f \in \BB_{\hol}^{\ast}(\Gamma_0(q_1))$,
\end{dcases*}	\\
\left|\zeta\left(\frac{1}{2} + it\right) L\left(\frac{1}{2} + it,\chi_D\right)\right|^2 & \ll (1 + |t|)^{1 - \delta}.
\end{align*}
\end{remark}

\phantomsection
\addcontentsline{toc}{section}{Acknowledgements}

\subsection*{Acknowledgements}

The first author would like to thank Paul Nelson for suggesting the usage of \hyperref[MVlemma]{Lemma \ref*{MVlemma}} in place of the method sketched in \hyperref[matrixcoeffremark]{Remark \ref*{matrixcoeffremark}}, Dan Collins for many useful conversations about local constants in the Watson--Ichino formula, and above all Peter Sarnak for his encouragement on working on this problem. The authors would also like to thank Bingrong Huang and the anonymous referee for helpful suggestions.

\end{document}